%% file: p13negulescu.tex
\documentclass[12pt]{amsart}
\usepackage{a4wide,amsmath,amssymb,graphicx}
\usepackage{color}
\allowdisplaybreaks

\let\pa\partial
\let\na\nabla
\let\eps\varepsilon
\newcommand{\R}{\mathbb{R}}
\newcommand{\N}{\mathbb{N}}
\newcommand{\C}{{\mathbb C}}  
\newcommand{\diver}{\textnormal{div}}

\newcommand{\tr}{\operatorname{tr}}

\newtheorem{theorem}{Theorem}
\newtheorem{proposition}[theorem]{Proposition}
\newtheorem{lemma}[theorem]{Lemma}
\theoremstyle{remark}
\newtheorem{remark}[theorem]{Remark}

\usepackage{pgfplots}           

\pgfplotsset{compat=newest}
\pgfplotsset{plot coordinates/math parser=false}

\graphicspath{{../figures/}}

\pgfplotsset{compat=newest}
\pgfplotsset{plot coordinates/math parser=false}

\newlength\spdenslinwidth
\newlength\denslinwidth
\newlength\artifactwidth
\newlength\poispotwidth
\newlength\sppoispotwidth
\newlength\smalldevwidth
\newlength\entropyqlogwidth

\begin{document}
\title[Bounded weak solutions to a matrix drift-diffusion model]{Bounded weak
solutions to \\
a matrix drift-diffusion model for spin-coherent \\
electron transport in semiconductors}

\author[A. J\"ungel]{Ansgar J\"ungel}
\address{Institute for Analysis and Scientific Computing, Vienna University of  
	Technology, Wiedner Hauptstra\ss e 8--10, 1040 Wien, Austria}
\email{juengel@tuwien.ac.at} 

\author[C. Negulescu]{Claudia Negulescu}
\address{Institut de Math\'ematiques de Toulouse (IMT), Universit\'e Paul Sabatier, 
118 route de Narbonne, 31062 Toulouse cedex, France}
\email{claudia.negulescu@math.univ-toulouse.fr}

\author[P. Shpartko]{Polina Shpartko}
\address{Institute for Analysis and Scientific Computing, Vienna University of  
	Technology, Wiedner Hauptstra\ss e 8--10, 1040 Wien, Austria}
\email{polina.shpartko@tuwien.ac.at}

\date{\today}

\thanks{The authors acknowledge partial support from  
the Austrian-French Project Amad\'ee of the Austrian Exchange Service (\"OAD).
The first and last authors have been partially supported by 
the Austrian Science Fund (FWF), grants P20214, P22108, I395, and W1245.
The authors thank Stefan Possanner for helpful advices.
} 

\begin{abstract}
The global-in-time existence and uniqueness
of bounded weak solutions to a spinorial matrix 
drift-diffusion model for semiconductors is proved. 
Developing the electron density matrix in
the Pauli basis, the coefficients (charge density and spin-vector density) satisfy a
parabolic $4\times 4$ cross-diffusion system.
The key idea of the existence proof is to work with different variables: 
the spin-up and spin-down densities as well as
the parallel and perpendicular components of the spin-vector density with
respect to the magnetization. 
In these variables, the diffusion matrix becomes diagonal.
The proofs of the $L^\infty$ estimates are based on Stampacchia truncation 
as well as Moser- and Alikakos-type iteration arguments.
The monotonicity of the entropy (or free energy) is also proved.
Numerical experiments in one space dimension using a finite-volume discretization
indicate that the entropy decays exponentially fast to the equilibrium state.
\end{abstract}

 

\maketitle


\section{Introduction}

The aim of this work is to analyze a spinorial matrix drift-diffusion model. 
The evolution of the (Hermitian) electron density matrix
$N\in\C^{2\times 2}$ and the current density matrix $J\in\C^{2\times 2}$ 
is governed by the matrix equations
\begin{align}
  & \pa_t N + \diver J + i\gamma[N,\vec m\cdot\vec\sigma]
  = \frac{1}{\tau}\left(\frac12\mbox{tr}(N)\sigma_0-N\right), \label{1.N} \\
  & J = -DP^{-1/2}(\na N+N\na V)P^{-1/2}, \label{1.J}
\end{align}
where $[A,B]=AB-BA$ is the commutator for matrices $A$ and $B$.
The scaled physical parameters are the strength of the pseudo-exchange field
$\gamma>0$, the (normalized) direction of the magnetization 
$\vec m=(m_1,m_2,m_3)\in\R^3$,
the spin-flip relaxation time $\tau>0$, and the space-dependent
diffusion coefficient $D=D(x)\in(0,\infty)$. 
In the analytic part of this paper, we assume that the magnetization vector
$\vec m$ is constant.
Furthermore, $\vec\sigma=(\sigma_1,\sigma_2,\sigma_3)$ is the triple of the 
Pauli matrices (see \cite[Formula (1)]{PoNe11}), $\sigma_0$ is the 
identity matrix in $\C^{2\times 2}$,
$\mbox{tr}(N)$ denotes the trace of the matrix $N$, and
$P=\sigma_0+p\vec m\cdot\vec\sigma$, where $p=p(x)\in[0,1)$ represents
the spin polarization of the scattering rates. 
The product $\vec m\cdot\vec\sigma$ is defined as the matrix 
$m_1\sigma_1+m_2\sigma_2+m_3\sigma_3$.
System \eqref{1.N}-\eqref{1.J} is solved in the bounded cylinder
$\Omega\times[0,T)\subset\R^3\times[0,\infty)$, 
supplemented with the boundary and initial conditions
\begin{equation}\label{1.Nbic}
  N=\frac12 n_D\sigma_0\quad\mbox{on }\pa\Omega,\ t>0, 
  \quad N(0)=N^0\quad\mbox{in }\Omega.
\end{equation}
The electric potential $V$ is given by the Poisson equation
\begin{equation}\label{1.V}
  -\lambda_D^2\Delta V = \mbox{tr}(N)-C(x)\quad\mbox{in }\Omega, \quad
  V=V_D\quad\mbox{on }\pa\Omega,
\end{equation}
where $\lambda_D>0$ is the scaled Debye length and $C(x)$ the doping profile
\cite{Jue09}.

Equations \eqref{1.N}-\eqref{1.J} describe the time evolution of the
density matrix of the electrons, coupling the charge and spin degrees of freedom.
The coupling is linear in the polarization $p$ of the scattering states.
The commutator $[N,\vec m\cdot\vec\sigma]$ in \eqref{1.N}
models the precession of the spin polarization on the macroscopic level. 
Furthermore, the right-hand side in \eqref{1.N} describes the relaxation of
the spin density to an equilibrium density due to so-called spin-flip processes.

The above model was derived in \cite{PoNe11} from a
matrix Boltzmann equation involving the precession of the spin polarization
in the diffusion limit. In this derivation, 
the scattering operator is assumed to consist of a (dominant)
symmetric collision operator from the Stone model and a spin-flip 
operator with the relaxation time $\tau>0$. Generally, $D$ is a diffusion
matrix in $\R^{3\times 3}$. However, under the assumption that the 
scattering rate in the Stone model is smooth and invariant under isometric
transformations, Proposition 1 in \cite{Pou91} shows that $D$ is a multiple
of the identity matrix with a positive factor which is identified with 
the positive number $D$.

\subsection{State of the art}

In the mathematical literature, various spinorial diffusion models 
have been investigated. 
El Hajj \cite{ElH10} derived from a linear spinor Boltzmann equation 
a two-component drift-diffusion model assuming
strong spin-orbit coupling (yielding \eqref{1.nplus}-\eqref{1.nminus} below).
Two-component models are well known in the physical community
\cite[Formulas (II.39)-(II.40)]{FMESZ07}. They are parabolic equations which are
only weakly coupled through spin-flip interaction terms.
The existence, uniqueness, and boundedness of weak solutions to
such models (for spin-polarized electrons and holes) 
was proved by Glitzky in two space dimensions \cite{Gli08}. 
In three space dimensions, the well-posedness of the stationary system
was shown in \cite{GaGl10}.

El Hajj derived in \cite{ElH10} also a spin-vector drift-diffusion model assuming 
that the spin-orbit coupling is moderate compared to the mean free path. 
Mathematically, the model consists of 
matrix-valued linear parabolic equations which are
strongly coupled. The strong coupling usually makes the analysis very difficult since
standard tools like maximum principles and regularity theory generally do not apply.
In \cite{ElH10}, the scattering rates are supposed to be scalar quantities.
Possanner and Negulescu \cite{PoNe11} assumed that the scattering rates are 
positive definite Hermitian matrices, which yields spin-dependent mean free paths
and hence more general model equations, which are analyzed in this paper.

Spinorial semiconductor models beyond drift-diffusion have been also investigated. 
For instance,
Ben Abdallah and El Hajj \cite{BeEl09} derived spinorial energy-transport 
and spin-vector drift-diffusion models with Fermi-Dirac statistics.
Quantum drift-diffusion equations for a lateral superlattice with Rashba spin-orbit 
interaction were deduced by Bonilla et al.\ \cite{BaMe10,BBA08} and numerically
solved by Possanner et al.\ \cite{PBMN13}.

Because of the strong coupling of the model equations and the quadratic-type
nonlinearity of the drift term, there are no analytical
results available for spin-vector drift-diffusion systems like
\eqref{1.N}-\eqref{1.V}, at least up to our knowledge. In this work, we
provide for the first time analytical results for the above spin-vector model.

\subsection{Various formulations}

The key idea of the analysis is to work with different variables.

\medskip\noindent{\em 1.\ Charge and spin-vector densities:}
It turns out that the analysis of \eqref{1.N}-\eqref{1.J} is easier, when
we develop $N$ and $J$ in the Pauli basis $(\sigma_0,\ldots,\sigma_3)$
via $N=\frac12 n_0\sigma_0 + \vec n\cdot\vec\sigma$ and 
$J=\frac12 j_0\sigma_0 + \vec j\cdot\vec\sigma$, where $n_0$ is the electron
charge density and $\vec n$ the spin-vector density. 
These quantities are real since $N$ is Hermitian and positive semi-definite.
Setting $\vec n=(n_1,n_2,n_3)$ and $\vec j=(j_1,j_2,j_3)$,
system \eqref{1.N}-\eqref{1.J} can be written equivalently as \cite[Remark 1]{PoNe11}
\begin{align}
  & \pa_t n_0 - \diver\left(\frac{D}{\eta^2}(J_0 - 2p\vec J\cdot\vec m)\right)
	= 0, \label{1.n0} \\
  & \pa_t n_k - \diver\left(\frac{D}{\eta^2}\Big(\eta J_k 
  + (1-\eta)(\vec J\cdot\vec m)m_k - \frac{p}{2}J_0m_k\Big)\right) 
	- 2\gamma(\vec n\times\vec m)_k = -\frac{n_k}{\tau}, \label{1.vecn} \\
  & J_0 = \na n_0+n_0\na V, \quad 
  \vec J = (J_1,J_2,J_3) = \na\vec n + \vec n\na V, \quad x\in\Omega,\ t>0, 
	\label{1.J0}
\end{align}
where $k=1,2,3$, $\eta=\sqrt{1-p^2}$ (which is generally space dependent), 
and $(j_0,\vec j)$ and $(J_0,\vec J)$ are related by
$$
  j_0 = -\frac{D}{\eta^2}(J_0 - 2p\vec J\cdot\vec m), \quad
  j_k = -\frac{D}{\eta^2}\left(\eta J_k + (1-\eta)
  (\vec J\cdot\vec m)m_k - \frac{p}{2}J_0m_k\right).
$$
The boundary and initial data is given by
\begin{equation}\label{1.bic}
  n_0=n_D, \quad \vec n=0 \quad\mbox{on }\pa\Omega,\ t>0, 
  \quad n_0(0)=n_0^0, \quad \vec n(0)=\vec{n}^0\quad\mbox{in }\Omega,
\end{equation}
where $N^0=\frac12 n_0^0\sigma_0+\vec{n}^0\cdot\vec\sigma$.
System \eqref{1.V}-\eqref{1.vecn} is strongly coupled due to the cross-diffusion
terms in $n_0$ and $\vec n$ with linear diffusion coefficients. 
Note that any solution $(n_0,\vec n)$ to \eqref{1.n0}-\eqref{1.bic}
defines a solution $N$ to \eqref{1.N}-\eqref{1.Nbic}.

\medskip\noindent{\em 2.\ Spin-up and spin-down densities:}
The coupling becomes weaker by working in the 
spin-up and spin-down densities $n_\pm=\frac12 n_0\pm \vec n\cdot\vec m$. 
Indeed, multiplying \eqref{1.vecn}
by $\vec m$, some terms cancel, and combining the resulting expression 
with \eqref{1.n0}, we find that $(n_+,n_-)$ solves
\begin{align}
  & \pa_t n_+ - \diver\big(D(1+p)(\na n_+ + n_+\na V)\big)
  = \frac{1}{2\tau}(n_- - n_+), \label{1.nplus} \\
  & \pa_t n_- - \diver\big(D(1-p)(\na n_- + n_-\na V)\big)
  = \frac{1}{2\tau}(n_+ - n_-), \quad x\in\Omega,\ t>0, \label{1.nminus} \\
  & n_+ = n_- = \frac12 n_D\quad\mbox{on }\pa\Omega,\ t>0, \quad
  n_\pm(0) = \frac12n_0^0\pm\vec n^0\cdot\vec m\quad\mbox{in }\Omega. 
	\label{nplus.bic}
\end{align}
This formulation is only possible if the magnetization $\vec m$ is constant.
Its advantage is that the above system 
is only coupled in the source terms (and through the electric potential)
such that we can apply a maximum principle. 
Note, however, that it is {\em not} equivalent to \eqref{1.n0}-\eqref{1.J0} since
we are losing the information on the complete spin-vector density $\vec n$.
Thus, this model contains less information than the full model
\eqref{1.n0}-\eqref{1.J0}, but it is easier to analyze mathematically.
In fact, drift-diffusion equations similar to \eqref{1.nplus}-\eqref{1.nminus}
have been thoroughly investigated in, e.g., \cite{GaGr86}.

\medskip\noindent{\em 3.\ Parallel and perpendicular densities:}
For the proof of the boundedness of $\vec n$, we employ a third formulation,
the decomposition in the parallel and perpendipular components of $\vec n$ with
respect to $\vec m$. For this, let $\vec n_\parallel = (\vec n\cdot\vec m)\vec m$
and $\vec n_\perp=\vec n-(\vec n\cdot\vec m)\vec m$. Then an elementary
computation, using that $\vec m$ is constant, 
shows that $(\vec n_\parallel,\vec n_\perp)$ solves
\begin{align}
  & \pa_t\vec n_\parallel - \diver\left(\frac{D}{\eta^2}\Big((\vec J\cdot\vec m)\vec m
	- \frac{p}{2}J_0\vec m\Big)\right) = -\frac{\vec n_\parallel}{\tau}, 
	\label{npara} \\
	& \pa_t \vec n_\perp - \diver\left(\frac{D}{\eta^2}(
  \na\vec n_\perp + \vec n_\perp\na V)\right) - 2\gamma(\vec n_\perp\times\vec m) 
	= - \frac{\vec n_\perp}{\tau}. \label{nperp}
\end{align}
The second equation depends on $\vec n_\perp$ only, which makes possible the
application of a maximum principle.

\subsection{Main results}

Our first result is the global-in-time existence of weak solutions to
\eqref{1.N}-\eqref{1.V} (or equivalently, \eqref{1.V}-\eqref{1.bic})
under the assumption that the diffusion coefficient $D$ and the spin polarization
$p$ are constant. We introduce the space
$$
  W^{1,2}(0,T;H_0^1,L^2)
  = H^1(0,T;H^{-1}(\Omega))\cap L^2(0,T;H_0^1(\Omega)).
$$
Recall that $\vec m$ is assumed to be a constant vector.

\begin{theorem}[Existence of bounded weak solutions I]\label{thm.ex}
Let $T>0$ and let $\Omega\subset\R^3$ be a bounded domain with $\pa\Omega\in C^{1,1}$.
Furthermore, let $\lambda_D$, $\gamma$, $D>0$, $0\le p<1$,
and $\vec m\in\R^3$ with $|\vec m|=1$.
The data satisfies $C\in L^\infty(\Omega)$ and
\begin{align*}
  & 0\le n_D\in H^1(\Omega)\cap L^\infty(\Omega), \quad
  V_D\in W^{2,q_0}(\Omega), \quad q_0>3, \\
  & n_0^0,\, \vec{n}^0\cdot\vec m\in L^\infty(\Omega), \quad 
  \frac12 n_0^0\pm\vec{n}^0\cdot\vec m\ge 0.
\end{align*}
Then there exists a unique solution $(N,V)$ to \eqref{1.N}-\eqref{1.V} such that
$N=\frac12 n_0\sigma_0+\vec{n}\cdot\vec\sigma$ satisfies
\begin{align*}
  & n_0,\, n_k\in W^{1,2}(0,T;H_0^1,L^2), \quad
  V\in L^\infty(0,\infty;W^{2,q_0}(\Omega)), \quad q_0>3, \\
  & 0\le \frac12 n_0\pm\vec n\cdot\vec m\in L^\infty(0,\infty;L^\infty(\Omega)), 
	\quad k=1,2,3.
\end{align*}
In particular, $n_0$ and $\vec n\cdot\vec m$ are bounded uniformly in $t>0$. 
If additionally
$|\vec n^0|\in L^\infty(\Omega)$, then $|\vec n|\in L^\infty(0,T;L^\infty(\Omega))$.
\end{theorem}

For simplicity, the boundary data is assumed to be independent of time.
The general situation can also be treated but is more technical; 
see, e.g., \cite{ZaJu13}. The proof of the theorem is based on the Leray-Schauder
fixed-point theorem. 
The key idea is to employ the variables $(n_0,\vec n\cdot\vec m)$ for the 
ellipticity argument. More precisely, consider
the main part of the differential operator in its weak formulation
(Equation \eqref{1.n0} is divided by four),
\begin{align*}
  I &= \frac{D}{\eta^2}\int_\Omega\Big(\frac14\na n_0\cdot\na\phi_0
	- \frac{p}{2}\na(\vec n\cdot\vec m)\cdot\na\phi_0 \\
	&\phantom{xx}{}+ \eta\na\vec n:\na\vec\phi + (1-\eta)\na(\vec n\cdot\vec m)\cdot
	\na(\vec\phi\cdot\vec m) - \frac{p}{2}\na n_0\cdot\na(\vec\phi\cdot\vec m)\Big)dx,
\end{align*}
where $(\phi_0,\vec\phi)$ is some test function. Then, choosing 
$\phi_0=n_0$, $\vec\phi=\vec n$ and using $\eta(1-\eta/2)\|\na\vec n\|^2
\ge\eta(1-\eta/2)|\na\vec n\cdot\vec m|^2$, the above integral
can be estimated by
\begin{align}
   I &= \frac{D}{\eta^2}\int_\Omega\Big(\frac14|\na n_0|^2
	- p\na(\vec n\cdot\vec m)\cdot\na n_0 \label{1.I} \\
	&\phantom{xx}{}+ \eta\Big(1-\frac{\eta}{2}\Big)\|\na \vec n\|^2
	+ \frac{\eta^2}{2}\|\na \vec n\|^2 + (1-\eta)|\na(\vec n\cdot\vec m)|^2\Big)dx
	\nonumber \\
	&\ge \frac{D}{\eta^2}\int_\Omega
	\begin{pmatrix} \na n_0 \\ \na(\vec n\cdot\vec m) \end{pmatrix}^\top
	\begin{pmatrix}
	1/4 & -p/2 \\ -p/2 & 1-\eta^2/2 \end{pmatrix}
	\begin{pmatrix} \na n_0 \\ \na(\vec n\cdot\vec m) \end{pmatrix}dx 
	+ \frac{D}{2}\int_\Omega\|\na \vec n\|^2 dx. \nonumber
\end{align}
The $2\times 2$ matrix on the right-hand side is positive definite (see the discussion
after \eqref{ex.coer}), which allows us to apply the Lax-Milgram lemma.
Although the matrix is positive definite in the variables $(n_0,\vec n\cdot\vec m)$
only, we achieve gradient estimates also for $\vec n$. This is the key estimate.
Note that the assumption that $\vec m$ is constant is crucial here. 

The positivity and boundedness of $n_0$ is proved by applying a Stampacchia 
truncation argument to system \eqref{1.nplus}-\eqref{1.nminus} in the variables 
$n_\pm=\frac12 n_0\pm\vec n\cdot\vec m$. The boundedness of $\vec n$ does not
follow from this argument. The idea is to prove the boundedness of
$\vec n_\perp=\vec n-(\vec n\cdot\vec m)\vec m$, since it satisfies
the decoupled drift-diffusion equation \eqref{nperp}.
Then, because of $\vec n\cdot\vec m\in L^\infty$, we infer that $|\vec n|\in L^\infty$.
A standard Stampacchia truncation method
cannot be employed here since the term $\vec n\times\vec m$ mixes the components
of $\vec n$. Therefore, we use a Moser-type iteration method, i.e., we derive 
$L^q$ estimates for $\vec n_\perp$ uniform in $q<\infty$ and pass to the 
limit $q\to\infty$.

For nonconstant diffusion coefficients $D(x)$ and spin polarizations $p(x)$,
we are able to prove the existence of solutions with given electric potential
only. The reason is that our proof is based on a truncation in the Poisson equation
(see Section \ref{sec.ex}) and not in the drift term as was done in, e.g.\
\cite{GaGr86}. This truncation yields $W^{1,\infty}$ regularity for the
potential, but we lose the monotonicity of the nonlocal quadratic
drift term. We assume that there exists $\delta_0>0$ such that
\begin{equation}\label{hypo.Dp}
  D, p\in L^{\infty}(\Omega), \quad D(x)\ge \delta_0>0, \ 0\le p(x)\le 1-\delta_0
	\quad\mbox{for }x\in\Omega.
\end{equation}

\begin{theorem}[Existence of bounded weak solutions II]\label{thm.ex2}
Let $|\na V|\in L^\infty(0,\infty;L^\infty(\Omega))$ be given and let the
assumptions of Theorem \ref{thm.ex} on the parameters and data hold with the
exception that $D(x)$ and $p(x)$ satisfy \eqref{hypo.Dp}.
Then there exists a unique solution to \eqref{1.N}-\eqref{1.Nbic}
such that $N=\frac12 n_0\sigma_0+\vec n\cdot\vec\sigma$ satisfies for any $T>0$,
$n_0$, $n_k\in W^{1,2}(0,T;H_0^1,L^2)$ ($k=1,2,3$). Furthermore,
$$
	0\le \frac12 n_0\pm\vec n\cdot\vec m\in L^\infty(0,\infty;L^\infty(\Omega)).
$$
If additionally $|\vec n^0|\in L^\infty(\Omega)$, then $|\vec n|\in
L^\infty(0,T;L^\infty(\Omega))$.
\end{theorem}

The well-posedness follows
from an abstract existence result since the system is linear. The difficulty
is to show the $L^\infty$ bounds. Here, we employ a variant of Alikakos'
iteration technique, modified by Kowalczyk \cite{Kow05} (see Lemma \ref{lem.bd}
in the appendix). 

\medskip\noindent
Our second result is the existence of an entropy (more precisely, free energy)
for the spinorial drift-diffusion system. This result holds also for
nonconstant diffusion coefficients $D(x)$ and spin polarizations $p(x)$.
The entropy is formulated
in terms of the solution to \eqref{1.nplus}-\eqref{1.nminus}:
\begin{align}
  H_0(t) &= \int_\Omega\left(h(n_+) + h(n_-) + \frac{\lambda_D^2}{2}|\na(V-V_D)|^2
	\right)dx, \label{H0} \\
  h(n_\pm) &= \int_{n_D/2}^{n_\pm}(\log s-\log(n_D/2))ds. \nonumber
\end{align}
The first two terms in the definition of $H_0(t)$
describe the internal energy of the two spin
components and the last term is the electric energy, relative to the boundary values.
From the results on the standard drift-diffusion model (see, e.g., \cite{GaGr86}),
it is not surprising that the entropy $H_0$ is nonincreasing in time if the initial
data are in thermal equilibrium. We say that $(n_{\rm th},V_{\rm th})$ is
a thermal equilibrium state if $n_{\rm th}=\rho \exp(-V_{\rm th})$ 
for some constant $\rho>0$ and $V_{\rm th}$ is the unique solution to
$$
  -\lambda_D^2\Delta V_{\rm th} = \rho e^{-V_{\rm th}}-C(x)\quad\mbox{in }\Omega,
	\quad V_{\rm th}=V_D\quad\mbox{on }\pa\Omega.
$$
For given $(n_D,V_D)$, defined on $\pa\Omega$ and satisfying $\log(n_D/2)+V_D=c\in\R$ 
on $\pa\Omega$, we extend these functions to $\Omega$ by setting
$n_D=n_{\rm th}$, $V_D=V_{\rm th}$ and $\rho=2\exp(c)$. Then $\log(n_D/2)+V_D=c$
in $\Omega$. The following result holds.

\begin{proposition}[Monotonicity of $H_0$]\label{prop.ent1}
Let \eqref{hypo.Dp} hold, $\log(n_D/2)+V_D=\mbox{\rm const.}$ in $\Omega$,
and $n_D\in W^{1,\infty}(\Omega)$ with $n_D\ge n_*>0$ in $\Omega$.
Let $(n_+,n_-,V)$ be a weak solution to \eqref{1.V}, 
\eqref{1.nplus}-\eqref{nplus.bic} in the sense of
Theorem \ref{thm.ex}. Then $t\mapsto H_0(t)$ is nonincreasing for $t>0$.
\end{proposition}

Using the techniques of \cite{GaGr86}, it is possible to infer the
exponential decay of $(n_+(t),n_-(t))$ to equilibrium as $t\to\infty$.
Since the proof is very similar to that of \cite{GaGr86}, we omit it.
However, we give a numerical example which illustrates the exponential decay.

One may ask if the quantum or von-Neumann entropy (or free energy)
\begin{equation}\label{HQ}
  H_Q(t) = \int_\Omega\Big(\tr[N(\log N-\log N_D-1)+N_D]
	+ \frac{\lambda_D^2}{2}|\na(V-V_D)|^2\Big)dx, 
\end{equation}
where $N_D=\frac12 n_D\sigma_0$ (see \eqref{1.Nbic}), is also nonincreasing in time.
Since the eigenvalues of $\log N$ are given by $\frac12 n_0\pm|\vec n|$, we
need to suppose that $\frac12 n_0>|\vec n|$ to have well-posedness of the
expression $\log N$. Because of the drift-diffusion structure of \eqref{1.N}
such a functional would be a natural candidate for an entropy.
However, we will show that this may be not the case.
In fact, a formal computation (detailed in Remark \ref{rem.HQ}) shows that
\begin{equation}\label{1.dHQdt}
  \frac{dH_Q}{dt}
	= -\int_\Omega Dn_0\sum_{j=1}^3\mbox{tr}\left[N\big(\pa_j(\log N+V\sigma_0)
	P^{-1/2}\big)^2\right]dx 
	- \frac{1}{\tau}\int_\Omega|\vec n|
	\log\frac{\tfrac12n_0+|\vec n|}{\tfrac12 n_0-|\vec n|}dx. 
\end{equation}
While the second integral is nonnegative, this may be not true for the first one.
We show in Remark \ref{rem.neg} that the integrand of the first integral 
may be negative for certain values of the variables.

In the special case $\vec m=(0,0,1)^\top$ and $\vec n_0=(0,0,n_3^0)^\top$, 
which we assume in our numerical simulations in Section \ref{sec.numer}, 
the spin-vector density only depends on the third component,
$\vec n(t)=(0,0,n_3(t))^\top$ for all $t>0$. 
In this situation, $H_Q$ is nonincreasing. This can be
seen by writing $H_Q$ equivalently as
\begin{align*}
  H_Q(t) &= \int_\Omega\Big((\tfrac12 n_0+|\vec n|)
	\big(\log(\tfrac12 n_0+|\vec n|)-1\big)
	+ (\tfrac12 n_0-|\vec n|)\big(\log(\tfrac12 n_0-|\vec n|)-1\big) \\
	&\phantom{xx}{}+ n_D - n_0\log(\tfrac12 n_D)
	+ \frac{\lambda_D^2}{2}|\na(V-V_D)|^2\Big)dx,
\end{align*}
This formulation follows from spectral theory and the fact that the
eigenvalues of $N$ are given by $\frac12 n_0\pm|\vec n|$
\cite[Section 2]{PoNe11}. We recall that
for any continuous function $f:\R\to\R$ and any Hermitian matrix $A$
with (real) eigenvalues $\lambda_j$, it holds that $\tr[f(A)]=\sum_j f(\lambda_j)$.
Now, $\frac12n_0\pm|\vec n|=\frac12n_0\pm n_3=n_\pm$, and consequently,
$H_Q$ coincides with $H_0$, which is monotone by Proposition \ref{prop.ent1}.

The paper is organized as follows. Theorems \ref{thm.ex} and \ref{thm.ex2} are proved
in Section \ref{sec.ex}. Proposition \ref{prop.ent1} and formula \eqref{1.dHQdt}
are shown in Section \ref{sec.ent}. Some numerical results for a 
one-dimensional ballistic diode in a multilayer structure 
using a finite-volume scheme are presented in Section \ref{sec.numer}.
The appendix is concerned with the proof of a
general boundedness result needed for the proof of Theorem \ref{thm.ex2}.


\section{Existence of solutions}\label{sec.ex}


\subsection{Proof of Theorem \ref{thm.ex}}

The existence proof is based on the Leray-Schauder fixed-point theorem and
a truncation argument. It is divided into several steps.

{\em Step 1: Reformulation.}
We introduce the variable $w_0=n_0-n_D(x)$ whose trace vanishes on $\pa\Omega$. 
Then equations \eqref{1.n0}-\eqref{1.J0} are equivalent to
\begin{align}
  & \frac14\pa_t w_0 - \frac{D}{4\eta^2}\diver(J_w - 2p\vec J\cdot\vec m) 
  = \frac{D}{4\eta^2}\diver(\na n_D+n_D\na V), \label{ex.w0} \\
  & \pa_t n_k - \frac{D}{\eta^2}\diver\left(\eta J_k 
  + (1-\eta)(\vec J\cdot\vec m)m_k
  - \frac{p}{2}J_wm_k\right) - 2\gamma(\vec n\times\vec m)_k \label{ex.nj} \\
  &\phantom{xx}{}
  = -\frac{n_k}{\tau} -\frac{Dp}{2\eta^2}\diver\big((\na n_D+n_D\na V)m_k\big),
  \quad k=1,2,3, \nonumber
\end{align}
where $J_w = \na w_0+w_0\na V$. The boundary and initial conditions are given by
\begin{equation}\label{ex.bic}
  w_0 = n_k = 0\quad\mbox{on }\pa\Omega,\ k=1,2,3, \quad 
  w_0(\cdot,0)=n_0^0-n_D, \ \vec n(0)=\vec n^0\quad\mbox{in }\Omega.
\end{equation}

{\em Step 2: Definition of the fixed-point operator.}
The idea is to fix a density $(\rho_0,\vec\rho)$, to solve the Poisson
equation including $\rho_0$ on its right-hand side, and finally to solve
a linearized version of \eqref{ex.w0}-\eqref{ex.nj} 
for the density $(w_0,\vec w)$, where $\vec w=(w_1,w_2,w_3)$. The
fixed-point operator is then defined by the mapping 
$(\rho_0,\vec\rho)\mapsto V\mapsto (w_0,\vec w)$. More precisely,
let $\rho=(\rho_0,\vec\rho)\in L^2(0,T;L^2(\Omega))^4$ 
and $\delta\in[0,1]$ be given and 
introduce the truncation $[x]=\max\{0,\min\{x,2M\}\}$ for $x\in\R$, where
$$
  M = \max\Big\{\sup_{\pa\Omega}n_D,\ 
	\sup_\Omega\Big(\frac12n_0^0+|\vec n^0\cdot\vec m|\Big),\
  \sup_\Omega C\Big\}.
$$
Let $V(t)\in H^1(\Omega)$ be the unique solution to
\begin{equation}\label{ex.V}
  -\lambda^2\Delta V(t) = [\rho_0(t)+n_D]-C(x)\quad \mbox{in }\Omega, \quad
  V(t)=V_D\quad\mbox{on }\pa\Omega.
\end{equation}
Then $t\mapsto V(t)$ is Bochner measurable and $V\in L^2(0,T;H^1(\Omega))$.
In fact, by elliptic regularity and $\pa\Omega\in C^{1,1}$, 
$V(t)\in W^{2,q_0}(\Omega)$ for $q_0>3$.
Since the right-hand side of \eqref{ex.V} is an element of 
$L^\infty(0,T;L^\infty(\Omega))$,
$V\in L^\infty(0,T;W^{2,q_0}(\Omega))$. This implies, because of the 
Sobolev embedding $W^{2,q_0}(\Omega)\hookrightarrow W^{1,\infty}(\Omega)$
in dimensions $d\le 3$, that $|\na V|\in L^\infty(0,T;L^\infty(\Omega))$.

Next, we define the bilinear form 
$a(\cdot,\cdot;t):H_0^1(\Omega)^4\times H_0^1(\Omega)^4\to\R$ 
for all $w=(w_0,\vec w)=(w_0,w_1,w_2,w_3)$, 
$\phi=(\phi_0,\vec\phi)=(\phi_0,\phi_1,\phi_2,\phi_3)\in H_0^1(\Omega)^4$ by
\begin{equation}\label{ex.a}
  a(w,\phi;t) = a_0(w,\phi) + a_V(w,\phi;t) + a_1(w,\phi) + a_2(w,\phi),
\end{equation}
where 
\begin{align*}
  a_0(w,\phi) &= \frac{D}{\eta^2}\int_\Omega\Big(\frac14\na w_0\cdot\na\phi_0
  + \eta\na\vec w:\na\vec\phi
  + (1-\eta)\na(\vec w\cdot\vec m)\cdot\na(\vec\phi\cdot\vec m) \\
  &\phantom{xx}{}
  - \frac{p}{2}\na(\vec w\cdot\vec m)\cdot\na\phi_0 
  - \frac{p}{2}\na w_0\cdot\na(\vec\phi\cdot\vec m)\Big)dx, \\
  a_V(w,\phi;t) &=
  \frac{D\delta}{\eta^2}\int_\Omega\na V(t)\cdot\Big(\frac14 w_0\na\phi_0
  + \eta \na\vec\phi\cdot\vec w + (1-\eta)(\vec w\cdot\vec m)
  \na(\vec\phi\cdot\vec m) \\
  &\phantom{xx}{}
  - \frac{p}{2}(\vec w\cdot\vec m)\na\phi_0
  - \frac{p}{2}w_0\na(\vec\phi\cdot\vec m)\Big)dx, \\
  a_1(w,\phi) &=
  -2\gamma\delta\int_\Omega(\vec w\times\vec m)\cdot\vec\phi dx, \\
  a_2(w,\phi) &=
  \frac{1}{\tau}\int_\Omega \vec w\cdot\vec\phi dx, 
\end{align*}
and $\na\vec w:\na\vec\phi=\sum_{k=1}^3\na w_k\cdot\na\phi_k$, and the linear
mapping $F(\cdot;t):H_0^1(\Omega)^4\to\R$ by
$$
  F(\phi;t) = -\frac{D\delta}{4\eta^2}\int_\Omega(\na n_D+n_D\na V)\cdot\na\phi_0
  + \frac{D\delta p}{2\eta^2}\int_\Omega
  (\na n_D+n_D\na V)\cdot\na(\vec\phi\cdot\vec m)dx.
$$
Then the weak formulation of \eqref{ex.w0}-\eqref{ex.bic} (for $\delta=1$) reads as
\begin{equation}\label{ex.weak}
  \frac{d}{dt}\int_\Omega w(t)\cdot\phi dx + a(w,\phi;t) = F(\phi;t)
	\quad\mbox{for }\phi\in H_0^1(\Omega),\ t>0.
\end{equation}

Since $|\na V|\in L^\infty(0,T;L^\infty(\Omega))$, an elementary estimation shows that
$a$ and $F$ are bounded in the sense
$$
  a(w,\phi;t) \le K_0\|w\|_{H_0^1(\Omega)}\|\phi\|_{H_0^1(\Omega)}, \quad
	|F(\phi;t)|\le K_0\|\phi\|_{H_0^1(\Omega)}
$$
for all $w$, $\phi\in H_0^1(\Omega)^4$, where $K_0>0$ depends on
the $L^\infty(0,T;L^\infty(\Omega))$ norm of $\na V$ (and hence on $M$)
but is independent of $t>0$.
We claim that $a$ satisfies an abstract
G\r{a}rding inequality, i.e., there exist $K_1$, $K_2>0$ such that for all
$w\in H_0^1(\Omega)^4$,
$$
  a(w,w) \ge K_1\|w\|_{H_0^1(\Omega)}^2 - K_2\|w\|_{L^2(\Omega)}^2.
$$
To this end, we estimate the forms $a_0$, $a_V$, $a_1$, and $a_2$.
The first form equals
\begin{align*}
  a_0(w,w) &= \frac{D}{\eta^2}\int_\Omega\Big(\frac14|\na w_0|^2
  + \frac{\eta^2}{2}\|\na\vec w\|^2 + \eta\Big(1-\frac{\eta}{2}\Big)\|\na\vec w\|^2 \\
  &\phantom{xx}{}+ (1-\eta)|\na(\vec w\cdot\vec m)|^2
  - p\na w_0\cdot\na(\vec w\cdot\vec m)\Big)dx.
\end{align*}
Estimate \eqref{1.I} in the introduction shows that
\begin{equation}\label{ex.coer}
  a_0(w,w) \ge \frac{D}{\eta^2}\int_\Omega 
  \begin{pmatrix} \na w_0 \\ \na \vec w\cdot\vec m \end{pmatrix}^\top
  \begin{pmatrix} 1/4 & -p/2 \\ -p/2 & 1-\eta^2/2\end{pmatrix}
  \begin{pmatrix} \na w_0 \\ \na\vec w\cdot\vec m \end{pmatrix}dx
  + \frac{D}{2}\int_\Omega \|\na \vec w\|^2 dx.
\end{equation}
The above symmetric matrix is positive definite because its eigenvalues
$$
  \lambda_\pm = \frac18(5-2\eta^2) \pm \frac18\sqrt{(5-2\eta^2)^2-8\eta^2}
$$
are real (since $\eta\le 1$) and positive (since $\eta>0$). This leads to
$$
  a_0(w,w;t) \ge \frac{D}{\eta^2}\int_\Omega \Big(\lambda_-|\na w_0|^2
  + \lambda_-|\na(\vec w\cdot\vec m)|^2 + \frac{\eta^2}{2}\|\na\vec w\|^2\Big)dx
  \ge K_1\|w\|_{H_0^1(\Omega)}^2,
$$
where $K_1=\min\{\lambda_-,\eta^2/2\}$.

In order to estimate the second form $a_V(w,w)$, we employ the Poisson equation:
\begin{align*}
  a_V( & w,w) = \frac{D\delta}{2\eta^2}\int_\Omega\na V\cdot\na\Big(
  \frac14 w_0^2 + \eta|\vec w|^2 + (1-\eta)(\vec w\cdot\vec m)^2
  - p w_0(\vec w\cdot\vec m)\Big)dx \\
  &= \frac{D\delta}{2\eta^2\lambda_D^2}\int_\Omega([\rho_0+n_D]-C(x))
  \Big(\frac14 w_0^2 + \eta|\vec w|^2 + (1-\eta)(\vec w\cdot\vec m)^2
  - p w_0(\vec w\cdot\vec m)\Big)dx.
\end{align*}
Observing that
\begin{align*}
  \frac14 w_0^2 &+ \eta|\vec w|^2 + (1-\eta)(\vec w\cdot\vec m)^2
  - p w_0(\vec w\cdot\vec m) \\
	&= \Big(\frac12 w_0-p\vec w\cdot\vec m\Big)^2
	+ \eta\big(|\vec w|^2-(\vec w\cdot\vec m)^2\big) + \eta^2(\vec w\cdot\vec m)^2
	\ge 0
\end{align*}
and employing the Cauchy-Schwarz inequality, it follows that
\begin{align*}
  a_V(w,w) &\ge -\frac{D}{2\eta^2\lambda_D^2}\int_\Omega |C(x)|
	\Big(\frac14 w_0^2 + \eta|\vec w|^2 + (1-\eta)(\vec w\cdot\vec m)^2
  - p w_0(\vec w\cdot\vec m)\Big)dx \\
	&\ge -K_2\|w\|_{L^2(\Omega)}^2,
\end{align*}
where $K_2$ depends on $\|C\|_{L^\infty(\Omega)}$ and the parameters 
$D$, $p$, and $\lambda_D$ but not on $M$ or $\rho$. 
Finally, the third form vanishes, $a_1(w,w)=0$,
and the fourth form is nonnegative, $a_2(w,w)\ge 0$. This shows the claim.

By Corollary 23.26 in \cite{Zei90}, there exists a unique solution
$w\in W^{1,2}(0,T;H_0^1,L^2)^4$ to \eqref{ex.weak}
satisfying $w(0)=(n_0^0-n_D,\vec n^0)$. This defines the fixed-point
operator $S:L^2(0,T;L^2(\Omega))^4\times[0,1]\to L^2(0,T;L^2(\Omega))^4$,
$S(\rho,\delta)=w$. By construction, $S(\rho,0)=0$. Furthermore, standard
arguments show that $S$ is continuous. By the Aubin lemma,
the space $W^{1,2}(0,T;H_0^1,$ $L^2)^4$
embeddes compactly into $L^2(0,T;L^2(\Omega))^4$. Consequently, $S$ is compact.
It remains to prove some uniform estimates for all fixed points of $S(\cdot,\delta)$
in $L^2(0,T;L^2(\Omega))^4$.
Let $w\in W^{1,2}(0,T;H_0^1,L^2)^4$ be such a fixed point.
Employing $w=(w_0,\vec w)$ as a test function in \eqref{ex.weak}, 
the above estimates show that
\begin{align*}
  \frac12\|w(\cdot,T)\|_{L^2(\Omega)}^2 + K_1\|w\|_{L^2(0,T;H_0^1(\Omega))}^2
  &\le K_2\|w\|_{L^2(0,T;L^2(\Omega))}^2 + K_0\|w\|_{L^2(0,T;H_0^1(\Omega))} \\
  &\le K_2\|w\|_{L^2(0,T;L^2(\Omega))}^2 
	+ \frac{K_1}{2}\|w\|_{L^2(0,T;H_0^1(\Omega))}^2 
  + \frac{K_0^2}{2K_1}.
\end{align*}
Absorbing the second summand on the right-hand side
by the corresponding term on the left-hand side
and applying Gronwall's lemma, we achieve the bound
$\|w\|_{L^2(0,T;L^2(\Omega))}\le K$ for some $K>0$ uniform in $\rho$ and $\delta$.
By the Leray-Schauder fixed-point theorem, there exists a fixed point of
$S(\cdot,1)$, i.e.\ a solution to \eqref{ex.w0}-\eqref{ex.V}
with $[\rho(t)+n_D]$ replaced by $[w(t)+n_D]$.

{\em Step 3: Lower and upper bounds.} We show that $0\le n_0:=w+n_D\le 2M$ in
$\Omega\times(0,T)$ which allows us to remove the truncation in the Poisson equation.
For this, we consider the variables $n_\pm = \frac12 n_0\pm \vec n\cdot\vec m$.
We claim that $n_\pm\ge 0$. Indeed, with the test functions
$[n_\pm]^-=\min\{0,n_\pm\}$, which satisfy $[n_\pm]^-=0$ on $\pa\Omega$ and
$[n_\pm(0)]^-=0$, in the weak formulation of \eqref{1.nplus}
and \eqref{1.nminus}, respectively, it follows from
$n_\pm\na V\cdot\na[n_\pm]^-=[n_\pm]^-\na V\cdot\na[n_\pm]^-=\frac12\na V\cdot
\na([n_\pm]^-)^2$ that
\begin{align*}
  \frac12\int_\Omega &([n_\pm(t)]^-)^2 dx
  + D(1\pm p)\int_0^t\int_\Omega|\na[n_\pm]^-|^2 dx\,ds \\
  &= -D(1\pm p)\int_0^t\int_\Omega [n_\pm]^-\na V\cdot\na[n_\pm]^- dx\,ds
  \mp \frac{1}{2\tau}\int_0^t\int_\Omega(n_+-n_-)[n_\pm]^- dx\,ds \\
  &= -\frac{D(1\pm p)}{2\lambda_D^2}\int_0^t\int_\Omega([n_0]-C(x))([n_\pm]^-)^2 
  dx\,ds
  \mp \frac{1}{2\tau}\int_0^t\int_\Omega(n_+-n_-)[n_\pm]^- dx\,ds.
\end{align*}
In the last step we have integrated by parts and employed the Poisson equation.
We add both equations and neglect the integrals involving $|\na[n_\pm]^-|^2$ to obtain
\begin{align*}
  \frac12\int_\Omega &\big(([n_+(t)]^-)^2 + ([n_-(t)]^-)^2\big) dx \\
  &\le -\frac{D}{2\lambda_D^2}\int_0^t\int_\Omega([n_0]-C(x))\big((1+p)([n_+]^-)^2 
  + (1-p)([n_-]^-)^2\big)dx\,ds \\
  &\phantom{xx}{}
	- \frac{1}{2\tau}\int_0^t\int_\Omega(n_+-n_-)([n_+]^- - [n_-]^-)dx\,ds \\
  &\le \frac{D}{2\lambda_D^2}\|C\|_{L^\infty(\Omega)}\int_0^t\int_\Omega
  \big(([n_+]^-)^2 + ([n_+]^-)^2\big)dx\,ds,
\end{align*}
using the fact that $x\mapsto[x]^-$ is nondecreasing.
Gronwall's lemma shows that $[n_\pm(t)]^-=0$ in $\Omega$ for any $t>0$ and hence,
$n_\pm\ge 0$ in $\Omega\times(0,T)$. 

For the proof of the upper bound, we employ the test function 
$[n_\pm-M]^+=\max\{0,n_\pm-M\}$,
which is admissible since $M\ge \frac12\sup_{\Omega\times(0,T)}n_D$,
and which satisfies $[n_\pm(0)-M]^-=0$, in \eqref{1.nplus}
and \eqref{1.nminus}, respectively, and add both equations:
\begin{align*}
  \frac12\int_\Omega &\big([n_+(t)-M]^{+2} + [n_-(t)-M]^{+2}\big)dx \\
  &\phantom{xx}{}+ D\int_0^t\int_\Omega
  \big((1+p)|\na[n_+-M]^+|^2 + (1-p)|\na[n_--M]^+|^2\big)dx\,ds \\
  &= -D\int_0^t\int_\Omega\big((1+p)(n_+-M)\na V\cdot\na[n_+-M]^+ \\
  &\phantom{xx}{}+ (1-p)(n_--M)\na V\cdot\na[n_--M]^+\big)dx\,ds \\
  &\phantom{xx}{}- DM\int_0^t\int_\Omega\na V\cdot\big((1+p)\na[n_+-M]^+  
  + (1-p)\na[n_--M]^+\big)dx\,ds \\
  &\phantom{xx}{}
  - \frac{1}{2\tau}\int_0^t\int_\Omega(n_+-n_-)\big([n_+-M]^+ - [n_--M]^+\big)dx\,ds.
\end{align*}
Observing that $(n_\pm-M)\na V\cdot\na[n_\pm-M]^+=\frac12\na V\cdot\na([n_\pm-M]^+)^2$,
integrating by parts, and employing the Poisson equation, we find that
\begin{align*}
  \frac12\int_\Omega &\big(([n_+(t)-M]^{+})^2 + ([n_-(t)-M]^{+})^2\big)dx \\
  &\le -\frac{D}{2\lambda_D^2}\int_0^t\int_\Omega
  ([n_0]-C(x))\big((1+p)([n_+-M]^+)^2 + (1-p)([n_--M]^+)^2\big)dx\,ds \\
  &\phantom{xx}{}- \frac{DM}{\lambda_D^2}\int_0^t\int_\Omega
  ([n_0]-C(x))\big((1+p)[n_+-M]^+ + (1-p)[n_--M]^+\big)dx\,ds.
\end{align*}
The last integral is nonnegative since $[n_0]-C(x)=[n_++n_-]-C(x)\ge 0$ 
on $\{n_\pm\ge M\}$ by definition of $M$.
Then Gronwall's lemma gives $[n_\pm-M]^+=0$ and
$n_\pm\le M$ in $\Omega\times(0,T)$. We have shown that $n_0=n_++n_-\le 2M$,
$\vec n\cdot\vec m=\frac12(n_+-n_-)\le \frac12 n_+\le \frac12 M$, and
$\vec n\cdot\vec m\ge -\frac12 n_-\ge -\frac12 M$.
Thus, we can remove the truncation in the Poisson equation, since $[n_0]=n_0$.

{\em Step 4: Uniqueness of solutions.}
Let $(w,V)$ and $(w^*,V^*)$ with $w=(w_0,\vec w)$ and $w^*=(w_0^*,\vec w^*)$ 
be two weak solutions to \eqref{ex.V} and \eqref{ex.weak}. 
Taking the difference of the equations
\eqref{ex.weak} for $w$ and $w^*$, respectively, and employing the
admissible test function $w-w^*$, we obtain
\begin{align*}
  \frac12\|&(w-w^*)(T)\|_{L^2(\Omega)}^2
  + \int_0^T\big(a(w,w-w^*;t)-a^*(w^*,w-w^*;t)\big)dt \\
  &= \int_0^T(F(w-w^*;t)-F^*(w-w^*;t))dt,
\end{align*}
where $a^*$ and $F^*$ denote the forms with $V$ replaced by $V^*$.
The Cauchy-Schwarz and Young inequalities and the elliptic estimate 
$\|\na(V-V^*)\|_{L^2(\Omega)}\le K\|w_0-w_0^*\|_{L^2(\Omega)}$ yield
\begin{align*}
  \int_0^T &(F(w-w^*;t)-F^*(w-w^*;t))dt
  = -\frac{D}{4\eta^2}\int_0^T\int_\Omega n_D\na(V-V^*)\cdot\na(w_0-w_0^*)dx\,dt \\
  &\phantom{xx}{}+ \frac{Dp}{2\eta^2}\int_0^T\int_\Omega n_D\na(V-V^*)\cdot
  \na((\vec{w}-\vec{w}^*)\cdot\vec m)dx\,dt \\
  &\le \eps\|\na(w-w^*)\|_{L^2(0,T;L^2(\Omega))}^2 
  + K(\eps)\|w_0-w_0^*\|_{L^2(0,T;L^2(\Omega))}^2,
\end{align*}
where here and in the following, $K>0$ denotes a generic constant and $\eps>0$.

Next, we consider the difference $a(w,w-w^*;t)-a^*(w^*,w-w^*;t)$.
As in Step 2, we find that
\begin{align*}
  \int_0^T a_0(w-w^*,w-w^*)dt &\ge K\int_0^T\int_\Omega\big(|\na(w_0-w_0^*)|^2
+ \|\na(\vec{w}-\vec{w}^*)\|^2\big)dx\,dt, \\
  \int_0^T a_1(w-w^*,w-w^*)dt &= 0, \\
  \int_0^T a_2(w-w^*,w-w^*)dt &= \frac{1}{\tau}\int_0^T\int_\Omega 
  |\vec{w}-\vec{w}^*|^2 dx\,dt \ge 0.
\end{align*}
The remaining difference involving the electric potentials becomes
\begin{align*}
  \int_0^T & \big(a_V(w,w-w^*;t)-a_{V^*}(w^*,w-w^*;t)\big)dt \\
  &= \frac{D}{2\eta^2}\int_0^T\int_\Omega\na V\cdot\na\Big(\frac14(w_0-w_0^*)^2
  + \eta|\vec{w}-\vec{w}^*|^2 + (1-\eta)((\vec{w}-\vec{w}^*)\cdot\vec m)^2 \\
  &\phantom{xx}{}
  - \frac{p}{2}(w_0-w_0^*)(\vec{w}-\vec{w}^*)\cdot\vec m\Big)dx\,dt \\
  &\phantom{xx}{}+ \frac{D}{\eta^2}\int_0^T\int_\Omega\na(V-V^*)\cdot\Big(
  \frac14 w_0^*\na(w_0-w_0^*) + \eta\vec{w}^*\na(\vec{w}-\vec{w}^*) \\
  &\phantom{xx}{}
  + (1-\eta)(\vec{w}^*\cdot\vec m)\na((\vec{w}-\vec{w}^*)\cdot\vec m) \\
  &\phantom{xx}{}
  - \frac{p}{2}(\vec{w}^*\cdot\vec m)\na(w_0-w_0^*)
  - \frac{p}{2}w_0^*\na((\vec{w}-\vec{w}^*)\cdot\vec m)\Big)dx\,dt.
\end{align*}
Integrating by parts in the first integral on the right-hand side
and employing the Poisson equation
shows that the first integral can be estimated from above by
$$
  K\big(\|w_0-w_0^*\|_{L^2(0,T;L^2(\Omega))}^2
  + \|\vec{w}-\vec{w}^*\|_{L^2(0,T;L^2(\Omega))}^2
  + \|(\vec{w}-\vec{w}^*)\cdot\vec m\|_{L^2(0,T;L^2(\Omega))}^2\big),
$$
where $K>0$ depends on the $L^\infty(0,T;L^\infty(\Omega))$ norm of $w_0$.
We take into account the $L^\infty$ norms of
$w_0^*$ and $\vec{w}^*\cdot\vec m$ to estimate the second integral from above by
\begin{align}
  & K\|\na(V-V^*)\|_{L^2(0,T;L^2(\Omega))}\big(
  \|\na(w_0-w_0^*)\|_{L^2(0,T;L^2(\Omega))} 
  + \|(\vec{w}-\vec{w}^*)\cdot\vec m\|_{L^2(0,T;L^2(\Omega))}\big) \nonumber \\
  &\phantom{xx}{}
  + \|\na(V-V^*)\|_{L^2(0,T;L^6(\Omega))}\|\vec{w}\|_{L^\infty(0,T;L^2(\Omega))}
  \|\vec{w}-\vec{w}^*\|_{L^2(0,T;L^3(\Omega))}. \label{ex.aux}
\end{align}
By elliptic regularity and Sobolev embedding, it follows that
$$
  \|\na(V-V^*)\|_{L^6(\Omega)}
  \le K\|V-V^*\|_{H^2(\Omega)} \le K\|w_0-w_0^*\|_{L^2(\Omega)}.
$$
Furthermore, by the Gagliardo-Nirenberg inequality,
$$
  \|\vec{w}-\vec{w}^*\|_{L^2(0,T;L^3(\Omega))}
  \le K\|\vec{w}-\vec{w}^*\|_{L^2(0,T;L^2(\Omega))}^{1/2}
  \|\na(\vec{w}-\vec{w}^*)\|_{L^2(0,T;L^2(\Omega))}^{1/2}.
$$
Putting together these estimates and using Young's inequality as well as the
embedding $W^{1,2}(0,T;H_0^1,L^2)
\hookrightarrow L^\infty(0,T;L^2(\Omega))$, we obtain for any $\eps>0$,
\begin{align*}
  \|&\na(V-V^*)\|_{L^2(0,T;L^6(\Omega))}\|\vec{w}\|_{L^\infty(0,T;L^2(\Omega))}
  \|\vec{w}-\vec{w}^*\|_{L^2(0,T;L^3(\Omega))} \\
  &\le K\|w_0-w_0^*\|_{L^2(0,T;L^2(\Omega))}
  \|\vec{w}-\vec{w}^*\|_{L^2(0,T;L^2(\Omega))}^{1/2}
  \|\na(\vec{w}-\vec{w}^*)\|_{L^2(0,T;L^2(\Omega))}^{1/2} \\
  &\le \eps\|\na(\vec{w}-\vec{w}^*)\|_{L^2(0,T;L^2(\Omega))}^2
  + K(\eps)\|w_0-w_0^*\|_{L^2(\Omega)}^2
  + K(\eps)\|\vec{w}-\vec{w}^*\|_{L^2(0,T;L^2(\Omega))}^2.
\end{align*}
Hence, \eqref{ex.aux} can be estimated from above by
$$
  \eps\|\na(\vec{w}-\vec{w}^*)\|_{L^2(0,T;L^2(\Omega))}^2
  + K(\eps)\|w_0-w_0^*\|_{L^2(\Omega)}^2
  + K(\eps)\|\vec{w}-\vec{w}^*\|_{L^2(0,T;L^2(\Omega))}^2.
$$

Summarizing the above estimates, we infer that
\begin{align*}
  \|& (w-w^*)(T)\|_{L^2(\Omega)}^2 + K\int_0^T\big(\|\na(w_0-w_0^*)\|_{L^2(\Omega)}^2
+ \|\na(\vec{w}-\vec{w}^*)\|_{L^2(\Omega)}^2\big)dt \\
  &\le \eps\int_0^T\|\na(\vec{w}-\vec{w}^*)\|_{L^2(\Omega)}^2 dt 
  + K(\eps)\int_0^T\big(\|w_0-w_0^*\|_{L^2(\Omega)}^2
  + \|\vec{w}-\vec{w}^*\|_{L^2(\Omega)}^2
  \big)dt.
\end{align*}
Thus, choosing $\eps>0$ sufficiently small and employing Gronwall's lemma,
we infer that $w=w^*$ in $\Omega$, $t>0$.

{\em Step 5: $L^\infty$ bound for $\vec n$.}
Let $q\ge 1$. Since $|\vec n_\perp|^q \vec n_\perp$ 
is not an admissible test function, we need to regularize. 
For this, set $g_\eps(y) = y^q/(1+\eps y^q)$ and
$G_\eps(y)=\int_0^y g_\eps(z)dz$ for $y\ge 0$ and $\eps>0$ .
Then $g_\eps(|\vec n_\perp|^2)\vec n_\perp$ is an admissible test function in the weak 
formulation of \eqref{nperp} (since $\vec n=0$ on $\pa\Omega$):
\begin{align}
  \int_0^t \langle\pa_t & \vec n_\perp,g_\eps(|\vec n_\perp|^2)\vec n_\perp\rangle ds
	+ \frac{D}{\eta}\int_0^t\int_\Omega(\na\vec n_\perp+\na V\vec n_\perp)
	:\na(g_\eps(|\vec n_\perp|^2)\vec n_\perp)dx\,ds \nonumber \\
	&= 2\gamma\int_0^t\int_\Omega g_\eps(|\vec n_\perp|^2)(\vec n_\perp\times\vec m)
	\cdot\vec n_\perp dx\,ds
	- \frac{1}{\tau}\int_0^t\int_\Omega g_\eps(|\vec n_\perp|^2)|\vec n_\perp|^2 dx\,ds 
	\le 0, \label{b.aux}
\end{align}
since the first integral on the right-hand side vanishes and the second one
is nonnegative. Here,
$\langle\cdot,\cdot\rangle$ denotes the dual product between $H^{-1}(\Omega)$
and $H_0^1(\Omega)$. Taking into account a variant of Prop.~23.20 in \cite{Zei90},
the first integral on the left-hand side of \eqref{b.aux} equals
$$
  \langle\pa_t \vec n_\perp,g_\eps(|\vec n_\perp|^2)\vec n_\perp\rangle
	= \frac12\,\frac{d}{dt}\int_\Omega G_\eps(|\vec n_\perp|^2) dx.
$$
Setting $F_\eps(y) = \int_0^y g_\eps'(z)z dz$ and employing the Poisson
equation \eqref{1.V}, the second integral in \eqref{b.aux} can be estimated as
\begin{align*}
  \frac{D}{\eta}\int_0^t \int_\Omega &
	\Big(\frac14g_\eps'(|\vec n_\perp|^2)\big|\na(|\vec n_\perp|^2)\big|^2 
	+ g_\eps(|\vec n_\perp|^2)\|\na \vec n_\perp\|^2 \\
	&\phantom{xx}{}+ g_\eps'(|\vec n_\perp|^2)|\vec n_\perp|^2
	\na(|\vec n_\perp|^2)\cdot\na V
	+ \frac12 g_\eps(|\vec n_\perp|^2)\na(|\vec n_\perp|^2)\cdot\na V\Big)dx\,ds \\
	&\ge \frac{D}{2\eta}\int_0^t \int_\Omega
	\na\big(2F_\eps(|\vec n_\perp|^2)+G_\eps(|\vec n_\perp|^2)\big)\cdot\na V dx\,ds \\
	&= \frac{D}{2\eta\lambda_D^2}\int_0^t\int_\Omega
	\big(2F_\eps(|\vec n_\perp|^2)+G_\eps(|\vec n_\perp|^2)\big)(n_0-C(x))dx\,ds \\
	&\ge -\frac{D}{2\eta\lambda_D^2}\|C\|_{L^\infty(\Omega)}\int_0^t\int_\Omega
	\big(2F_\eps(|\vec n_\perp|^2)+G_\eps(|\vec n_\perp|^2)\big)dx\,ds.
\end{align*}
We have proved that
$$ 
  \frac{d}{dt}\int_\Omega G_\eps(|\vec n_\perp|^2) dx
	\le \frac{D}{\eta\lambda_D^2}\|C\|_{L^\infty(\Omega)}\int_0^t\int_\Omega
	\big(2F_\eps(|\vec n_\perp|^2)+G_\eps(|\vec n_\perp|^2)\big)dx\,ds.
$$
Elementary estimations show that $F_\eps(y)\le qG_\eps(y)$ 
for all $y\ge 0$, yielding
$$
  \frac{d}{dt}\int_\Omega G_\eps(|\vec n_\perp|^2) dx
	\le K(2q+1)\int_\Omega G_\eps(|\vec n_\perp|^2) dx, \quad
	K = \frac{D}{\eta\lambda_D^2}\|C\|_{L^\infty(\Omega)}.
$$
Then Gronwall's lemma and the assumption $|\vec n_\perp(0)|\in L^\infty(\Omega)$ give
for all $t>0$,
$$
  \int_\Omega G_\eps(|\vec n_\perp(t)|^2) dx 
	\le e^{K(2q+1)t}\int_\Omega G_\eps(|\vec n_\perp(0)|^2) dx
	\le e^{K(2q+1)t}\int_\Omega |\vec n_\perp(0)|^{2q} dx.
$$
By dominated convergence, we can pass to the limit $\eps\to 0$. Then, taking
the $2q$-th root, 
$$
  \|\vec n_\perp(t)\|_{L^{2q}(\Omega)} \le e^{2Kt}\|\vec n_\perp(0)\|_{L^{2q}(\Omega)},
	\quad t>0.
$$
The right-hand side is bounded uniformly in $q<\infty$. Therefore, the limit
$q\to\infty$ leads to
$$
   \|\vec n_\perp(t)\|_{L^{\infty}(\Omega)} 
	\le e^{2Kt}\|\vec n_\perp(0)\|_{L^{\infty}(\Omega)}, \quad t>0,
$$
which ends the proof.


\subsection{Proof of Theorem \ref{thm.ex2}}

Since for given $V$, system \eqref{1.n0}-\eqref{1.J0} is linear, 
we only need to prove the coercivity of the bilinear form \eqref{ex.a} 
in order to prove the existence of a weak solution to \eqref{ex.weak}. 
Compared to the proof of Theorem \ref{thm.ex}, the
quotient $D/\eta^2$ depends on the spatial variables and we cannot
employ the Poisson equation to estimate the form $a_V(w,w;t)$. Instead, we
use the Cauchy-Schwarz inequality to obtain for all $\eps>0$,
$$
  a_V(w,w;t) \ge \eps\|w\|_{H_0^1(\Omega)}^2 - K(\eps)\|w\|_{L^2(\Omega)}^2.
$$
The constant $K(\eps)>0$ depends on the $L^\infty$ norm of $\na V$.
Thus, the existence and uniqueness 
of a weak solution follows from Corollary 23.26 in \cite{Zei90}.
The nonnegativity of $n_\pm$ and hence of $n_0$ is shown as in the proof of 
Theorem \ref{thm.ex}.

The main difficulty of the proof of Theorem \ref{thm.ex2}
is to derive the $L^\infty$ bounds. To this end, 
we employ Lemma \ref{lem.bd} in the appendix, which is based on
the iterative technique of Alikakos \cite{Ali79},
modified by Kowalczyk \cite{Kow05}. 
First, we note that the entropy estimate from Proposition \ref{prop.ent1}
implies that $n_\pm\in L^\infty(0,T;L^1(\Omega))$. The proposition is proved for
special boundary data. However, the proof also shows that for
general boundary data, the entropy is bounded on each interval $(0,T)$. 

Next, let $k=\|n_+(0)\|_{L^\infty(\Omega)}+\|n_-(0)\|_{L^\infty(\Omega)}$.
$0<k<L$, $q\ge 1$, and set $[n_+]_L = \min\{n_+,L\}$. Then
$\phi(n_+) = (([n_+]_L-k)^+)^q$ is an admissibible test function in 
the weak formulation of \eqref{1.nplus}. Since $\phi(n_+(0))=0$ by construction
and 
$$
  \Phi(n_+) = \int_0^{n_+}\phi(y)dy \ge \frac{1}{q+1}(([n_+]_L-k)^+)^{q+1},
$$
it follows that
$$
  \int_0^t \langle \pa_t n_+,\phi(n_+)\rangle ds
	= \int_\Omega\big(\Phi(n_+(t))-\Phi(n_+(0))\big)dx
	\ge \frac{1}{q+1}\int_\Omega (([n_+(t)]_L-k)^+)^{q+1} dx.
$$
Therefore, setting $D_0(x)=D(x)(1+p(x))$, we infer from \eqref{1.nplus} that
\begin{align*}
  \frac{1}{q+1} & \int_\Omega (([n_+(t)]_L-k)^+)^{q+1} dx
	+ q\int_0^t\int_\Omega (([n_+]_L-k)^+)^{q-1}|\na n_+|^2 dx\,ds \\
	&= -q\int_0^t\int_\Omega D_0(x)(n_+ - k) (([n_+]_L-k)^+)^{q-1}
	\na V\cdot\na([n_+(t)]_L-k)^+ dx\,ds \\
	&\phantom{xx}{}- kq\int_0^t\int_\Omega D_0(x)(([n_+]_L-k)^+)^{q-1}
	\na V\cdot\na([n_+]_L-k)^+ dx\,ds \\
	&\phantom{xx}{}-\frac{1}{\tau}\int_0^t\int_\Omega (n_+-n_-)(([n_+]_L-k)^+)^{q}dx\,ds.
\end{align*}
Since $n_+-k=([n_+]_L-k)^+$ and $\na n_+=\na([n_+(t)]_L-k)^+$ in $\{k<n_+<L\}$,
this can be written as
\begin{align*}
  \int_\Omega & (([n_+(t)]_L-k)^+)^{q+1} dx
	+ \frac{4q}{q+1}\int_0^t\int_\Omega |\na(([n_+]_L-k)^+)^{(q+1)/2}|^2 dx\,ds \\
	&= -2q\int_0^t\int_\Omega D_0(x)(([n_+]_L-k)^+)^{(q+1)/2}
	\na V\cdot\na(([n_+]_L-k)^+)^{(q+1)/2} dx\,ds \\
	&\phantom{xx}{}- 2kq\int_0^t\int_\Omega D_0(x)(([n_+]_L-k)^+)^{(q-1)/2}
	\na V\cdot\na(([n_+]_L-k)^+)^{(q+1)/2} dx\,ds \\
	&\phantom{xx}{}-\frac{1}{\tau}\int_0^t\int_\Omega (n_+-n_-)(([n_+]_L-k)^+)^{q}dx\,ds.
\end{align*}
Employing the Cauchy-Schwarz inequality and absorbing the gradient terms by the
second integral on the left-hand side, we conclude that
\begin{align*}
  \int_\Omega & (([n_+(t)]_L-k)^+)^{q+1} dx
	+ \frac{2q}{q+1}\int_0^t\int_\Omega |\na(([n_+]_L-k)^+)^{(q+1)/2}|^2 dx\,ds \\
	&\le q(q+1)K_1\int_0^t\int_\Omega(([n_+]_L-k)^+)^{q+1}dx\,ds \\
	&\phantom{xx}{}+ q(q+1)k^2 K_1\int_0^t\int_\Omega(([n_+]_L-k)^+)^{q-1}dx\,ds \\
	&\phantom{xx}{}-\frac{1}{\tau}\int_0^t\int_\Omega (n_+-n_-)(([n_+]_L-k)^+)^{q}dx\,ds.
\end{align*}
where $K_1=\|D(1+p)\|_{L^\infty(\Omega)}^2\|\na V\|_{L^\infty(0,T;L^\infty(\Omega))}^2$.
We apply Young's inequality in the form $z^{q-1}\le ((q-1)/(q+1))z^{p+1} + 2/(q+1)$
to the second integral on the right-hand side to find that
\begin{align*}
  \int_\Omega & (([n_+(t)]_L-k)^+)^{q+1} dx
	+ \frac{2q}{q+1}\int_0^t\int_\Omega |\na(([n_+]_L-k)^+)^{(q+1)/2}|^2 dx\,ds \\
	&\le q(q+1)(1+k^2)K_1\int_0^t\int_\Omega(([n_+]_L-k)^+)^{q+1}dx\,ds 
  + 2qk^2 K_1 T\mbox{meas}(\Omega) \\
	&\phantom{xx}{}-\frac{1}{\tau}\int_0^t\int_\Omega (n_+-n_-)(([n_+]_L-k)^+)^{q}dx\,ds.
\end{align*}
A similar inequality can be derived by $n_-$. Adding both inequalities and
observing that
$$
  -\frac{1}{\tau}\int_0^t\int_\Omega (n_+-n_-)
	\big((([n_+]_L-k)^+)^{q} - (([n_+]_L-k)^+)^{q}\big)dx\,ds \le 0,
$$
it follows that
\begin{align*}
  \int_\Omega & \sum_{j=\pm}(([n_j(t)]_L-k)^+)^{q+1} dx
	+ \frac{2q}{q+1}\int_0^t\int_\Omega \sum_{j=\pm}
	|\na(([n_j]_L-k)^+)^{(q+1)/2}|^2 dx\,ds \\
	&\le q(q+1)K_2\int_0^t\int_\Omega\sum_{j=\pm}(([n_j]_L-k)^+)^{q+1}dx\,ds + qK_3.
\end{align*}
where $K_2=(1+k^2)K_1$ and $K_3=2k^2 K_1 T\mbox{meas}(\Omega)$.
By the Gronwall lemma, 
$$
  2^{-q}\bigg(\sum_{j=\pm}\|([n_j(t)]_L-k)^+\|_{L^{q+1}(\Omega)}\bigg)^{q+1}
	\le \sum_{j=\pm}\|([n_j(t)]_L-k)^+\|_{L^{q+1}(\Omega)}^{q+1} 
	\le qK_3 e^{K_2 q(q+1) t}.
$$
Since the right-hand side does not depend on $L>0$, we can perform the limit 
$L\to \infty$:
$$
  \|(n_+(t)-k)^+\|_{L^{q+1}(\Omega)} 
	+ \|(n_-(t)-k)^+\|_{L^{q+1}(\Omega)}
	\le 2(qK_3)^{1/(q+1)} e^{K_2 q t}.
$$
This shows that $n_+$, $n_-\in L^\infty(0,T;L^q(\Omega))$ for all $q<\infty$.
Unfortunately, we cannot perform the limit $q\to\infty$. In order to
derive an $L^\infty$ bound, we derive recursive inequalities. The idea
is to exploit the gradient term via the Gagliardo-Nirenberg inequality.

Since 
\begin{align*}
  \pa_t n_+ - \diver(D_0(x)\na n_+) 
	&= -\diver(D_0(x)n_+\na V)-\frac{n_+-n_-}{\tau} \\
  &\in L^\infty(0,T;(W^{1,q/(q-1)}(\Omega))')\quad\mbox{for all }q<\infty,
\end{align*}
maximal regularity \cite{Ama04} implies that $n_+\in L^q(0,T;W^{1,q}(\Omega))$.
This proves, together with the regularity $n_+\in L^\infty(0,T;L^q(\Omega))$ 
for all $q<\infty$, that $n_+^q\in L^2(0,T;H^1(\Omega))$. Hence $n_+^q - n_D^q
\in L^2(0,T;H_0^1(\Omega))$ is an admissible test function in \eqref{1.nplus}.
Estimating as above, we arrive at
\begin{align*}
  \frac{d}{dt}\int_\Omega & (n_+^{q+1}+n_-^{q+1})dx 
	+ \frac{2q}{q+1}\int_\Omega\big(|\na n_+^{(q+1)/2}|^2 + |\na n_-^{(q+1)/2}|^2\big)dx \\
	&\le q(q+1)K_2\int_\Omega (n_+^{q+1}+n_-^{q+1}) dx + qK_4,
\end{align*}
where $K_4>0$ depends on the boundary data $n_D$. Lemma \ref{lem.bd} in 
the appendix shows that $n_+(t)$, $n_-(t)$ are bounded in $L^\infty(\Omega)$
uniformly in $t>0$.
We infer that $n_0=n_++n_-$ and $\vec n\cdot\vec m=\frac12(n_+-n_-)$ are 
uniformly bounded in $L^\infty(\Omega)$. 

It remains to prove the boundedness of $\vec n$ which follows if we have shown
the boundedness of $\vec n_\perp$. An estimation as in Step 5 of the proof of
Theorem \ref{thm.ex} shows that $|\vec n_\perp|\in L^\infty(0,T;L^q(\Omega))$
for all $q<\infty$. Then the Alikakos iteration technique, applied to
\eqref{nperp}, shows that $|\vec n_\perp|\in 
L^\infty(0,T;L^\infty(\Omega))$.


\section{Monotonicity of the entropy}
\label{sec.ent}

In this section, we prove Proposition \ref{prop.ent1} and formula \eqref{1.dHQdt}.

\begin{proof}[Proof of Proposition \ref{prop.ent1}]
The idea is to employ $(\log(n_+/\widetilde n_D)+V-V_D,\log(n_-/\widetilde n_D)+V-V_D)$ 
as a test function in \eqref{1.nplus}-\eqref{1.nminus}, where 
$\widetilde n_D=\frac12 n_D$. Since the logarithm may be
undefined, we need to regularize. We set
\begin{align*}
  H_\delta(t) &= \int_\Omega\left(h_\delta(n_+) + h_\delta(n_-) 
	+ \frac{\lambda_D^2}{2}|\na(V-V_D)|^2\right)dx, \quad \delta>0, \\
	h_\delta(n_\pm) &= \int^{n_\pm}_{\widetilde n_D}\big(\log(s+\delta)
	-\log(\widetilde n_D+\delta)\big)ds \\
	&= (n_\pm+\delta)(\log(n_\pm+\delta)-1) 
	- (\widetilde n_D+\delta)(\log(\widetilde n_D+\delta)-1).
\end{align*}
Then the pointwise convergence 
$h_\delta(n_\pm)\to n_\pm(\log n_\pm-1)-\widetilde n_D(\log \widetilde n_D-1)$ 
as $\delta\to 0$ holds.
Since $n_\pm=\frac12 n_D=\widetilde n_D$ on $\pa\Omega$ and
$-\lambda_D^2\pa_t\Delta V=\pa_t(n_+ + n_-)\in L^2(0,T;H^{-1}(\Omega))$,
we can differentiate $H_\delta$ to obtain
\begin{align*}
  \frac{dH_\delta}{dt}
	&= \langle \pa_t n_+,\log(n_++\delta)-\log(\widetilde n_D+\delta)\rangle
	+ \langle \pa_t n_-,\log(n_-+\delta)-\log(\widetilde n_D+\delta)\rangle \\
	&\phantom{xx}{}+ \lambda_D^2\langle\pa_t\na V,\na(V-V_D)\rangle,
\end{align*}
where $\langle\cdot,\cdot\rangle$ is the duality pairing between $H^{-1}(\Omega)$
and $H_0^1(\Omega)$. Observing that
$$
  \lambda_D^2\langle\pa_t\na V,\na(V-V_D)\rangle
	= -\lambda_D^2\langle\pa_t\Delta V,V-V_D\rangle
	= \langle\pa_t(n_++n_-),V-V_D)\rangle,
$$
it follows that
\begin{align*}
  \frac{dH_\delta}{dt}
	&= \langle\pa_t n_+,\log(n_++\delta)+V-\log(\widetilde n_D+\delta)-V_D\rangle \\
	&\phantom{xx}{}+ \langle\pa_t n_-,\log(n_-+\delta)+V-\log(\widetilde n_D+\delta)-V_D\rangle \\
	&= I_1 + I_2.
\end{align*}
After inserting the evolution equation \eqref{1.nplus} for $n_+$, setting
$J_+=D_+(\na n_++n_+\na V)$ with $D_\pm=D(1\pm p)$, 
and integrating by parts, we find that the first term $I_1$ equals
\begin{align*}
	I_1 &= 
	-\int_\Omega \big(J_+\cdot\na(\log(n_++\delta)+V) 
	- J_+\cdot\na(\log(\widetilde n_D+\delta)+V_D)\big)dx \\
	&\phantom{xx}{}- \int_\Omega\frac{n_+-n_-}{\tau}
	\big(\log(n_++\delta)+V-\log(\widetilde n_D+\delta)-V_D\big)dx.
\end{align*}
By Young's inequality, the first integral becomes
\begin{align*}
	\int_\Omega & \left(-\frac{J_+}{n_++\delta}\cdot(\na n_+ + (n_++\delta)\na V)
	+ J_+\cdot\na(\log(\widetilde n_D+\delta)+V_D)\right)dx \\
	&= \int_\Omega\left(-\frac{|J_+|^2}{D_+(n_++\delta)} 
	- \frac{\delta J_+\cdot\na V}{D_+(n_++\delta)}
	+ J_+\cdot\na(\log(\widetilde n_D+\delta)+V_D)\right)dx \\
	&\le -\int_\Omega\frac{|J_+|^2}{2D_+(n_++\delta)}dx
	+ \delta^2\int_\Omega\frac{|\na V|^2}{D_+(n_++\delta)} dx
	+ \int_\Omega(n_++\delta)|\na(\log(\widetilde n_D+\delta)+V_D)|^2 dx \\
	&\le -\int_\Omega\frac{|J_+|^2}{2D_+(n_++\delta)}dx
	+ \delta\int_\Omega\frac{|\na V|^2}{D_+} dx
	+ \int_\Omega(n_++\delta)|\na(\log(\widetilde n_D+\delta)+V_D)|^2 dx.
\end{align*}
The integral $I_2$ can be estimated in a similar way, eventually leading to
\begin{align*}
  \frac{dH_\delta}{dt}
	&\le -\int_\Omega\left(\frac{|J_+|^2}{2D_+(n_++\delta)}
	+\frac{|J_-|^2}{2D_-(n_-+\delta)}\right)dx
	+ \delta\int_\Omega\left(\frac{1}{D_+}+\frac{1}{D_-}\right)|\na V|^2 dx \\
	&\phantom{xx}{}
	+ \int_\Omega\big((n_++\delta)+(n_-+\delta)\big)
	|\na(\log(\widetilde n_D+\delta)+V_D)|^2 dx \\
	&\phantom{xx}{}
	- \int_\Omega\frac{n_+-n_-}{\tau}\big(\log(n_++\delta)-\log(n_-+\delta)\big)dx.
\end{align*}
The last integral is nonnegative since $x\mapsto\log(x+\delta)$ is an increasing
function. Therefore,
\begin{align}
  H_\delta(t) &\le H_\delta(0)
	+ \delta\int_0^t\int_\Omega \left(\frac{1}{D_+}+\frac{1}{D_-}\right)
	|\na V|^2 dx\,dt \nonumber \\ 
	&\phantom{xx}{}+\int_0^t\int_\Omega \big((n_++\delta)+(n_-+\delta)\big)
	|\na(\log(\widetilde n_D+\delta)+V_D)|^2 dx\,ds. 	\label{3.dHdt}
\end{align}
Since $\log\widetilde n_D+V_D=\mbox{const.}$ and $n_D\ge n_*>0$, 
the gradient term in the last integral can be estimated by
\begin{align*}
  |\na(\log(\widetilde n_D+\delta)+V_D)|^2 
	&= \big|\na(\log\widetilde n_D+V_D)+\na(\log(\widetilde n_D+\delta)
	-\log\widetilde n_D)\big|^2 \\
	&\le \frac{\delta^2|\na\widetilde n_D|^2}{\tfrac12 n_*(\tfrac12 n_*+\delta)},
\end{align*}
and the right-hand side is bounded in $L^\infty(\Omega)$ uniformly in $\delta$.

Therefore, the last two integrals in \eqref{3.dHdt} can be bounded from above
by a multiple of $H_\delta$. Then Gronwall's
lemma shows that $H_\delta$ is bounded uniformly in $\delta\in(0,1)$.
By dominated convergence, we can pass to the limit $\delta\to 0$ in the integrals,
leading to $H_\delta(t) \le H_0(0)$ for $t>0$, where we used that
$\na(\log \widetilde n_D+V_D)=0$. This finishes the proof.
\end{proof}

\begin{remark}\label{rem.HQ}\rm
We prove formula \eqref{1.dHQdt}. For this, we first observe that the formula
$\pa_t\exp(A)=\exp(A)\pa_t A$ for any matrix $A$ and the fact that the matrix 
logarithm of $N$ is defined by the matrix $A$ satisfying $\exp(A)=N$ imply that
$\pa_t \log N=N^{-1}\pa_t N$ holds. We assume that $\log(n_D/2)+V_D=c\in\R$ 
in $\Omega$ (thermal equilibrium). Then $\log N_D+V_D\sigma_0=c\sigma_0$ in $\Omega$.

Now, using the Poisson equation \eqref{1.V}, a formal computation leads to
\begin{align*}
  \frac{dH_Q}{dt} &= \int_\Omega \big(\tr[\pa_t N(\log N-\log N_D)]
	+ \lambda_D^2\pa_t\na V\cdot\na (V-V_D)\big) dx \\
	&= \int_\Omega \big(\tr[\pa_t N(\log N-\log N_D)]dx
	+ \pa_t(\tr(N)-C(x))(V-V_D)\big)dx \\
	&= \int_\Omega \tr\big[\pa_t N(\log N-\log N_D+(V-V_D)\sigma_0)\big]dx.
\end{align*}
Then, with the evolution equation \eqref{1.N} for $N$, an integration by parts, 
and the property $\log N_D+V_D\sigma_0=c\sigma_0$ in $\Omega$, we obtain 
\begin{align*}
  \frac{dH_Q}{dt} &= -\int_\Omega\sum_{j=1}^3
	D(x)\tr\big[P^{-1/2}(\pa_j N+N\pa_j V)P^{-1/2}\pa_j(\log N+V\sigma_0)\big]dx \\
	&\phantom{xx}{}
	- i\gamma\int_\Omega\tr\big[[N,\vec m\cdot\vec\sigma](\log N+V\sigma_0)\big]dx \\
	&\phantom{xx}{}
	+ \frac{1}{\tau}\int_\Omega\tr\left[\left(\frac12\tr(N)\sigma_0-N\right)
	(\log N+V\sigma_0)\right]dx \\
	&= I_1+I_2+I_3.
\end{align*}

We compute the three integrals term by term.
Employing the invariance of the trace under
cyclic permutations and the commutativity relation $\log(N)N=N\log(N)$ gives
$$
  \tr\big[[N,\vec m\cdot\vec\sigma]\log N\big]
	= \tr\big[\vec m\cdot\vec\sigma\log(N)N-\vec m\cdot\vec\sigma N\log(N)\big] = 0,
$$
which shows that $I_2=0$. For the integral $I_3$, the spectral property
$\tr[f(A)]=\sum_j f(\lambda_j)$ for any continuous function $f$ and any Hermitian
matrix $A$ with (real) eigenvalues $\lambda_j$, we have
\begin{align*}
  \tr\bigg[ & \left(\frac12\tr(N)\sigma_0-N\right)(\log N+V\sigma_0)\bigg]
	= \frac12\tr[N]\tr[\log N] - \tr[N\log N] \\
	&= \frac12n_0\left(\log\left(\frac12 n_0+|\vec n|\right)
	+\log\left(\frac12 n_0-|\vec n|\right)\right) 
	- \left(\frac12 n_0+|\vec n|\right)\log\left(\frac12 n_0+|\vec n|\right) \\
	&\phantom{xx}{}
	- \left(\frac12 n_0-|\vec n|\right)\log\left(\frac12 n_0-|\vec n|\right) \\
	&= -|\vec n|\left(\log\left(\frac12 n_0+|\vec n|\right)
	-\log\left(\frac12 n_0-|\vec n|\right)\right) \le 0.
\end{align*}
Hence, $I_3\le 0$.
Finally, for the integral $I_1$, we employ the property
$\pa_j N=N\pa_j\log N$ and the cyclicity of the trace:
\begin{align*}
  I_1 &= -\int_\Omega D(x)\sum_{j=1}^3	\tr\big[N(\pa_j\log N+\pa_j V\sigma_0)
	P^{-1/2}\pa_j(\log N+V\sigma_0)P^{-1/2}\big]dx \\
	&=  -\int_\Omega D(x)\sum_{j=1}^3\tr\big[N(\pa_j(\log N+V\sigma_0)P^{-1/2})^2\big]dx.
\end{align*}
This proves \eqref{1.dHQdt}.
\qed
\end{remark}

\begin{remark}\label{rem.neg}\rm
We compute the trace of the matrix $N(\pa_j(\log N+V\sigma_0)P^{-1/2})^2$.
The matrix $P^{-1/2}$ can be expanded into the Pauli basis. Indeed,
using $P^{-1}=(1-p^2)^{-1}(\sigma_0-p\vec m\cdot\vec\sigma)$ and Lemma 2.3 in
\cite{PoNe11}, we infer that
$$
  P^{-1/2}=p_+\sigma_0-p_-\vec m\cdot\vec\sigma, \quad
  p_\pm=\frac{1}{2\sqrt{1-p^2}}\big(\sqrt{1+p}\pm\sqrt{1-p}\big).
$$
The matrix $A:=\pa_j(\log N+V\sigma_0)=a_0\sigma_0+\vec a\cdot\vec\sigma$
can be also expanded. We observe that
$N^{-1}=\beta(\frac12n_0\sigma_0-\vec n\cdot\vec\sigma)$, where
$\beta=1/(\frac14n_0^2-|\vec n|^2)$.
By Formula (8) in \cite{PoNe11}, we can write
$$
  (b_0\sigma_0+\vec b\cdot\vec\sigma)(c_0\sigma_0+\vec c\cdot\vec\sigma)
	= (b_0c_0+\vec b\cdot\vec c)\sigma_0
	+ (b_0\vec c+c_0\vec b+i\vec b\times\vec c)\cdot\vec\sigma.
$$
With this identity, we find that
\begin{align*}
  \pa_j\log N 
	&= N^{-1}\pa_j N \\
	&= \beta\left(\frac14n_0\pa_j n_0-\vec n\cdot\pa_j\vec n\right)\sigma_0
	+ \frac{\beta}{2}(n_0\pa_j\vec n-\pa_jn_0\vec n)\cdot\vec\sigma
	- i\beta(\vec n\times\pa_j\vec n)\cdot\vec\sigma.
\end{align*}
Consequently,
$$
  a_0 = \frac{\beta}{2}\pa_j\left(\frac{n_0^2}{4}-|\vec n|^2\right) + \pa_j V,
	\quad \vec a = \frac{\beta}{2}(n_0\pa_j\vec n-\pa_jn_0\vec n)
	- i\beta(\vec n\times\pa_j\vec n).
$$
In particular, $\vec a$ is generally a complex vector. We write
$\vec a=\vec a_1+i\vec a_2$ with real vectors $\vec a_1$ and $\vec a_2$.

With these preparations, we calculate
\begin{align*}
  (AP^{-1/2})^2
	&= \big[(a_0p_+-p_-\vec a\cdot\vec m)\sigma_0
	+ (p_+\vec a-a_0p_-\vec m-ip_-(\vec a\times\vec m))\cdot\vec\sigma\big]^2 \\
	&= \big[(a_0p_+-p_-\vec a\cdot\vec m)^2
	+(p_+\vec a-a_0p_-\vec m-ip_-(\vec a\times\vec m))^2\big]\sigma_0 \\
	&\phantom{xx}{}+ 2(a_0p_+-p_-\vec a\cdot\vec m)
	(p_+\vec a-a_0p_-\vec m-ip_-(\vec a\times\vec m))\cdot\vec\sigma \\
	&= \big[(a_0p_+-p_-\vec a\cdot\vec m)^2 + (p_+\vec a-a_0p_-\vec m)^2 \\
	&\phantom{xx}{}- 2ip_-(p_+\vec a-a_0p_-\vec m)\cdot(\vec a\times\vec m)
	-p_-^2(\vec a\times\vec m)^2\big]\sigma_0 \\	
	&\phantom{xx}{}+ 2(a_0p_+-p_-\vec a\cdot\vec m)
	(p_+\vec a-a_0p_-\vec m-ip_-(\vec a\times\vec m))\cdot\vec\sigma.
\end{align*}
Since $\vec a\cdot(\vec a\times\vec m)=\vec m\cdot(\vec a\times\vec m)=0$,
the right-hand side simplifies slightly:
\begin{align*}
  (AP^{-1/2})^2
	&= \big[(a_0p_+-p_-\vec a\cdot\vec m)^2 + (p_+\vec a-a_0p_-\vec m)^2 
	- p_-^2(\vec a\times\vec m)^2\big]\sigma_0 \\	
	&\phantom{xx}{}+ 2(a_0p_+-p_-\vec a\cdot\vec m)
	(p_+\vec a-a_0p_-\vec m-ip_-(\vec a\times\vec m))\cdot\vec\sigma.
\end{align*}
After multiplication by $N=\frac12n_0\sigma_0+\vec n\cdot\vec\sigma$, it follows that
\begin{align*}
  N(AP^{-1/2})^2
	&= \tfrac12 n_0\big((a_0p_+-p_-\vec a\cdot\vec m)^2
	+ (p_+\vec a-a_0p_-\vec m)^2 - p_-^2(\vec a\times\vec m)^2\big)\sigma_0 \\
	&\phantom{xx}{}+ 2(a_0p_+-p_-\vec a\cdot\vec m)
	\vec n\cdot(p_+\vec a-a_0p_-\vec m-ip_-(\vec a\times\vec m))\sigma_0
	+ \vec B\cdot\vec\sigma,
\end{align*}
for some vector $\vec B$ which disappears after taking the trace (since $\sigma_j$ is
traceless):
\begin{align*}
  \mbox{tr}\big[N(AP^{-1/2})^2\big]
	&= n_0\big((a_0p_+-p_-\vec a\cdot\vec m)^2
	+ (p_+\vec a-a_0p_-\vec m)^2 - p_-^2(\vec a\times\vec m)^2\big) \\
	&\phantom{xx}{}+ 4(a_0p_+-p_-\vec a\cdot\vec m)
	\vec n\cdot(p_+\vec a-a_0p_-\vec m-ip_-(\vec a\times\vec m)). 
\end{align*}
Since $dH_Q/dt$ is real, the above trace is real too. Thus,
\begin{align*}
  \mbox{tr}\big[N(AP^{-1/2})^2\big]
	&= n_0(a_0p_+-p_-\vec a_1\cdot\vec m)^2 - n_0p_-^2(\vec a_2\cdot\vec m)^2
	+ n_0(p_+\vec a_1-a_0p_-\vec m)^2 \\
	&\phantom{xx}{}- n_0p_+^2\vec a_2^2 
	- n_0p_-^2(\vec a_1\times\vec m)^2 + n_0p_-^2(\vec a_2\times\vec m)^2 \\
	&\phantom{xx}{}+ 4(a_0p_+ - p_-\vec a_1\cdot\vec m)
	\vec n\cdot(p_+\vec a_1 - a_0p_-\vec m + p_-(\vec a_2\times\vec m)) \\
	&\phantom{xx}{}+ 4p_-(\vec a_2\cdot\vec m)
	\vec n\cdot(p_+\vec a_2 - p_-\vec a_1\times\vec m) \\
	&= J_1+J_2,
\end{align*}
where
\begin{align*}
  J_1 &= n_0(a_0p_+-p_-\vec a_1\cdot\vec m)^2
	+ n_0(p_+\vec a_1-a_0p_-\vec m)^2 \\
	&\phantom{xx}{}
	+ 4(a_0p_+ - p_-\vec a_1\cdot\vec m)\vec n\cdot(p_+\vec a_1 - a_0p_-\vec m), \\
	J_2 &= -n_0p_-^2(\vec a_2\cdot\vec m)^2 - n_0p_+^2\vec a_2^2 
	- n_0p_-^2(\vec a_1\times\vec m)^2 + n_0p_-^2(\vec a_2\times\vec m)^2 \\
	&\phantom{xx}{}+ 4p_+p_-a_0\vec n\cdot(\vec a_2\times\vec m)
	- 4p_-^2(\vec a_1\cdot\vec m)\vec n\cdot(\vec a_2\times\vec m) \\
	&\phantom{xx}{}+ 4p_+p_-(\vec a_2\cdot\vec m)\vec n\cdot\vec a_2
	- 4p_-^2(\vec a_2\cdot\vec m)\vec n\cdot(\vec a_1\times\vec m).
\end{align*}

The term $J_1$ is nonnegative. Indeed, set $c_0=a_0p_+-p_-\vec a_1\cdot\vec m$ and 
$\vec c=p_+\vec a_1-a_0p_-\vec m$. Then
$$
  J_1
	= n_0c_0^2 + n_0\vec c^2 + 4c_0\vec n\cdot\vec c 
	= n_0\left(c_0+\frac{2}{n_0}\vec n\cdot\vec c\right)^2
	+ n_0\left(\vec c^2 - \frac{4}{n_0^2}(\vec n\cdot\vec c)^2\right).
$$
Since $\frac12 n_0>|\vec n|$ by assumption, the Cauchy-Schwarz inequality shows that
the last term is nonnegative and thus, $J_1\ge 0$.

As $\vec n\cdot\vec a_2=0$, the last summand in $J_2$ vanishes. Because of
$\vec n\cdot(\vec a_1\times\vec m)=\vec m\cdot(\vec n\times\vec a_1)
=\frac12\beta n_0\vec m\cdot(\vec n\times\pa_j\vec n)
=-\frac12n_0(\vec a_2\times\vec m)$, the next to last summand in $J_2$ 
can be combined with the first summand, yielding
\begin{align*}
  J_2 &= n_0p_-^2(\vec a_2\cdot\vec m)^2 - n_0p_+^2\vec a_2^2 
	- n_0p_-^2(\vec a_1\times\vec m)^2 + n_0p_-^2(\vec a_2\times\vec m)^2 \\
	&\phantom{xx}{}+ 4p_+p_-a_0\vec n\cdot(\vec a_2\times\vec m)
	- 4p_-^2(\vec a_1\cdot\vec m)\vec n\cdot(\vec a_2\times\vec m).
\end{align*}
Employing $(\vec a_1\times\vec m)^2=\vec a_1^2-(\vec a_1\cdot\vec m)^2$,
the third summand can be reformulated, and we end up with
\begin{align*}
  J_2 &= n_0p_-^2(\vec a_2\cdot\vec m)^2 - n_0p_+^2\vec a_2^2 
	- n_0p_-^2\vec a_1^2 + 4p_+p_-a_0\vec n\cdot(\vec a_2\times\vec m) \\
	&\phantom{xx}{}+ n_0p_-^2\left(\vec a_1\cdot\vec m-\frac{2}{n_0}
	\vec n\cdot(\vec a_2\times\vec m)\right)^2
	+ n_0p_-^2\left((\vec a_2\times \vec m)^2 - \frac{4}{n_0^2}
	(\vec n\cdot(\vec a_2\times\vec m))^2\right).
\end{align*}
Arguing similarly as for $J_1$, we see that the last two summands are nonnegative.

Suppose that $\vec n$ is parallel to $\vec c$, $\vec a_2$
is parallel to $\vec m$, and $c_0=0$. Then, if $n_0\gg 1$ and $|\vec a_2|\gg 1$, 
we obtain $J_1\approx 0$ and $J_2<0$. Hence, $\mbox{tr}[N(AP^{-1/2})^2]$ may be 
negative.
\qed
\end{remark}


\section{Numerical simulations}\label{sec.numer}

In this section, we solve equations \eqref{1.V}-\eqref{1.bic} numerically
for a one-dimensional ballistic diode in a multilayer structure. 
The aim is twofold: First, we present simulations of the stationary state
in order to compare our numerical results to those
of \cite{PoNe11} and to detail the differences
between the numerical solutions with and without Poisson equation.
Second, we show numerically that the entropy \eqref{H0} decreases exponentially
fast in time.

The multilayer consists of a ferromagnetic layer sandwiched between 
two nonmagnetic layers. More precisely, let $\Omega=(0,L)$. 
The nonmagnetic layers $(0,\ell_1)$ and $(\ell_2,L)$
are characterized by $\vec m=0$, $p=0$ and a high doping concentration,
$C(x)=C_{\rm max}$ for $x\in(0,\ell_1)\cup(\ell_2,L)$, 
whereas the magnetic layer $(\ell_1,\ell_2)$
features a nonvanishing magnetization, $\vec m=(0,0,1)^\top$ and $p>0$, and a low
doping, $C(x)=C_{\rm min}$ for $x\in(\ell_1,\ell_2)$. 
This multilayer structure was numerically
solved in \cite{PoNe11} but for constant electric fields only and without doping.

The boundary conditions read as
$$
  n_0(0,t)=n_0(L,t)=C_{\rm max}, \quad \vec n(0,t)=\vec n(L,t)=0, \quad
	V(0,t)=0, \ V(L,t)=U
$$
for all $t>0$, and the initial conditions are $n_0(x)=1$, $\vec n(x)=0$, and
$V(x)=Ux/L$ for $x\in\Omega$ and for some $U>0$.
We choose $C_{\rm max}=10^{21}$\,$m^{-3}$ and 
$C_{\rm min}=0.4\cdot 10^{19}$\,$m^{-3}$. This corresponds to an overall low-doping
situation. With larger doping concentrations, we observed numerical problems
in the Poisson equation due to the smallness of the scaled Debye length
$\lambda_D$. This problem is well known and can be solved by employing
more sophisticated numerical methods.

We choose the same physical parameters as in \cite{PoNe11}:
the diffusion coefficient $D=10^{-3}\,\mbox{m}^2\mbox{s}^{-1}$,
the thermal voltage $V_{\rm th}=0.0259$\,V, the spin-flip relaxation time
$\tau=10^{-12}$\,s, the strength of the pseudo-exchange field
$\gamma=2\hbar/\tau$, and the applied potential $U=1\,$V.
Typically, the spin-flip relaxation time
is of the order of a few to 100 picoseconds \cite{FZGM08}.
The layers have the same length of $0.4$\,$\mu$m such that the device length
equals $L=1.2$\,$\mu$m. 

\subsection{Implementation}

The equations are discretized in time by the implicit Euler approximation
and in space by a node-centered finite-volume method. This method is conservative
and is able to deal with discontinuous coefficients.
We choose a uniform grid of $M$ points
$x_i=i\triangle x$ with step size $\triangle x>0$,
containing the interface points $\ell_1$ and $\ell_2$,
and uniform time steps $t_k=k\triangle t$ with step size $\triangle t>0$. 
We also introduce the cell-center points
$x_{i+1/2}=(i+1/2)\triangle x$. Denoting by $n_{\ell,i}^k$
an approximation of $(\triangle x)^{-1}\int_{x_{i-1/2}}^{x_{i+1/2}}n_\ell(x,t_k)dx$
for $\ell=0,1,2,3$ (the index of the spin component), 
the discretization of \eqref{1.n0}-\eqref{1.vecn} reads as follows:
\begin{align}
  & \frac{1}{\triangle t}(n_{0,i}^{k}-n_{0,i}^{k-1})
	+ \frac{1}{\triangle x}(j_{0,i+1/2}^k-j_{0,i-1/2}^k) = 0, \label{4.eq1} \\
	& \frac{1}{\triangle t}(n_{\ell,i}^{k}-n_{\ell,i}^{k-1})
	+ \frac{1}{\triangle x}(j_{\ell,i+1/2}^k-j_{\ell,i-1/2}^k)
	- \frac{2\gamma}{\hbar}(\vec n_i^k\times\vec m_i)_\ell 
	+ \frac{1}{\tau}n_{\ell,i}^k = 0, \label{4.eq2} 
\end{align}
where $\ell=1,2,3$, $\vec n_i^k=(n_{1,i}^k,n_{2,i}^k,n_{3,i}^k)$, 
$\vec m_i=(m_{1,i},m_{2,i},m_{3,i})=\vec m(x_i)$, and
\begin{align*}
  j_{0,i+1/2}^k &= -\frac{D}{\eta_{i+1/2}^2}\big(J_{0,i+1/2}^k
	-2p_{i+1/2}\vec J_{i+1/2}^k\cdot\vec m_{i+1/2}\big), \\
	j_{\ell,i+1/2}^k &= -\frac{D}{\eta_{i+1/2}^2}\bigg(\eta_{i+1/2}J_{\ell,i+1/2}^k
	+ (1-\eta_{i+1/2})(\vec J_{i+1/2}^k\cdot\vec m_{i+1/2})m_{\ell,i+1/2} \\
	&\phantom{xx}{}- \frac{p_{i+1/2}}{2}J_{0,i+1/2}^k m_{\ell,i+1/2}\bigg), \\
  J_{\ell,i+1/2}^k &= \frac{1}{\triangle x}\bigg((n_{\ell,i+1}^k-n_{\ell,i}^k)
	+ \frac12(n_{\ell,i+1}^k+n_{\ell,i}^k)(V_{i+1}^k-V_i^k)\bigg), \quad \ell=0,1,2,3.
\end{align*}
where $\vec J_{i+1/2}^k=(J_{\ell,i+1/2}^k)_{\ell=1,2,3}$.
The discrete electric potential $V_i^k$, 
which is scaled by the thermal voltage $V_{\rm th}$, is obtained from the
Poisson equation, discretized by central finite differences. The nonlinear discrete
system is solved by the Gummel method, i.e., given $(n_{\ell,i}^{k-1},V_i^{k-1})$
and setting $\rho_i=n_{0,i}^{k-1}\exp(-V_i^{k-1})$, we solve
first the nonlinear system
$$
  -\lambda_D^2(V_{i+1}^k - 2V_i^k + V_{i-1}^k) = \rho_i e^{V_i^k} - C(x_i),
$$
by Newton's method. Then, using the updated potential $(V_i^k)$, 
the discrete system \eqref{4.eq1}-\eqref{4.eq2} is solved for $n_{\ell,i}^{k}$, 
and the procedure is iterated. For details on the Gummel method, we refer to
\cite{Jer96}.
We have chosen 60 grid points in each layer (i.e.\ $M=180$) and the time step
size $\triangle t=0.005\tau$. These values are chosen in such a way
that further refinement did not change the results.

\subsection{Numerical simulations}

We wish to compare the stationary numerical solutions to those presented in 
\cite{PoNe11}. To this end, we compute the solution to the
above numerical scheme for ``large'' time until the steady state is approximately
reached. As in \cite{PoNe11}, we fix a linear potential (i.e.\ without
Poisson equation). The (scaled) stationary
charge density $n_0$ and spin-vector density $\vec n=(0,0,n_3)$ 
are presented in Figure \ref{fig.potlin}. According to \cite{PoNe11}, the
discontinuity of the spin polarization $p$ at the interfaces is compensated by
a decrease of $n_0$ in the ferromagnetic layer, leading to a non-equilibrium 
spin density $n_3\neq 0$. Our numerical results correspond to those
in \cite[Figure 1]{PoNe11} except for the values of the charge density 
$n_0$ in the right layer. Possanner and Negulescu found that $n_0$ is 
strictly smaller than the doping concentration and that there is a small 
boundary layer at $x=L$. 

{\small
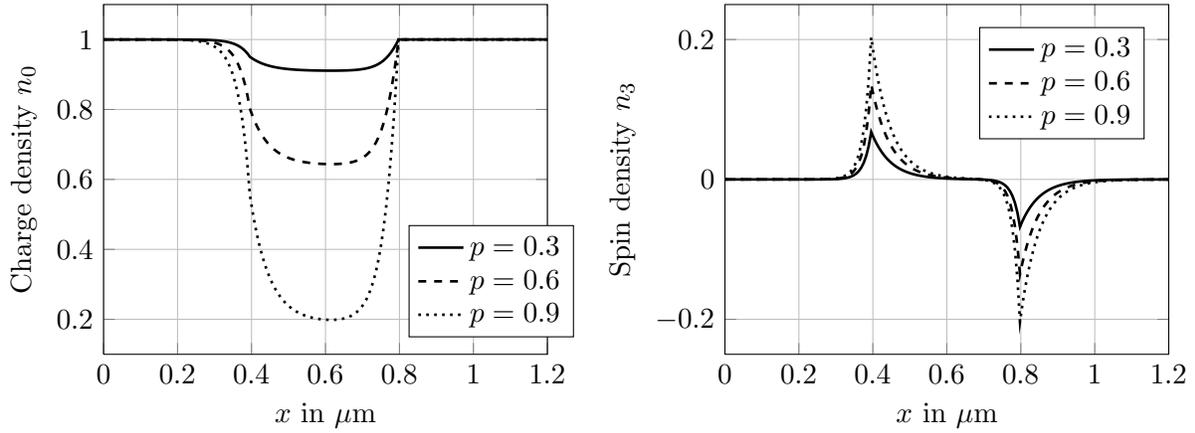
\begin{figure}[ht]
\input{n0_linpot.tikz}\hspace*{2mm}
\input{n3_linpot.tikz}
  \caption{Charge density $n_0$ (left) and spin density $n_3$ (right) in the 
	three-layer structure with linear potential and different spin polarizations $p$.}
  \label{fig.potlin}
\end{figure}
}

In order to understand this difference, we have
implemented the scheme of \cite{PoNe11}, which consists of a standard
Crank-Nicolson discretization in each of the layers, together with
continuity conditions for $n_0$, $\vec n$, $j_0$, and $\vec j=(j_1,j_2,j_3)$
at the interfaces $x=\ell_1$ and $x=\ell_2$. The numerical result for $n_0$
is shown in Figure \ref{fig.PoNe}. This does not correspond exactly to
the outcome of \cite{PoNe11} since our charge density is strictly {\em larger}
than the doping in the right layer. 
Refining the mesh, however, we see that
$n_0$ becomes closer to the doping and to the behavior
shown in Figure \ref{fig.potlin} (left). This supports the validity
of the numerical results computed from the finite-volume scheme. 

{\small
\begin{figure}[ht]
  {\centering\input{n0_artifact.tikz}}
  \caption{Charge density $n_0$ computed from the scheme in \cite{PoNe11}
	(without Poisson equation) for different grid point numbers $M$.}
  \label{fig.PoNe}
\end{figure}
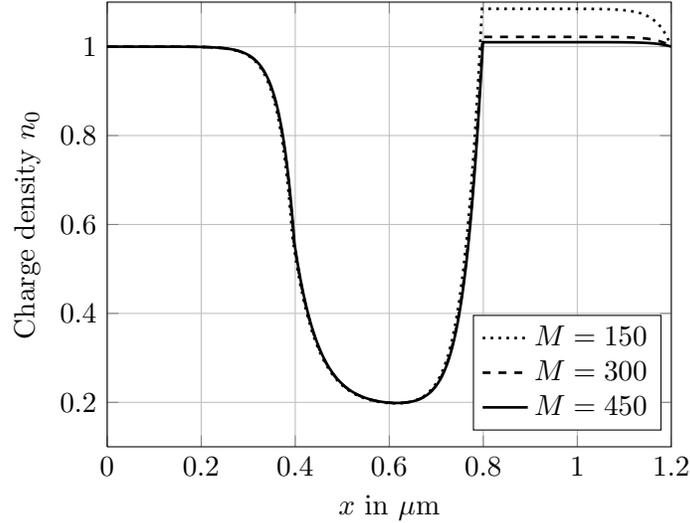
}

Figure \ref{fig.poispot} illustrates the densities $n_0$ and $n_3$ calculated
from the full model with self-consistent electric potential. 
Because of the rather small doping concentration and the large voltage,
electrons drift to the right contact, leading to a significant 
decrease of the charge density in the right layer.
For small values of the spin polarization $p$, the charge density is similar to the
charge density computed with vanishing magnetization. When the spin polarization
$p$ increases, the charge density is reduced in the ferromagnetic layer
and increases in the right nonmagnetic region. The self-consistent potential
leads to a slight reduction of the peaks of the spin density, compared with 
Figure \ref{fig.potlin} (right).

{\small
\begin{figure}[ht]
  \input{n0_poispot.tikz}\hspace*{2mm}
  \input{n3_poispot.tikz}
  \caption{Charge density $n_0$ (left) and spin density $n_3$ (right) 
	in the three-layer
	structure with self-consistent potential and different spin polarizations $p$.}
  \label{fig.poispot}
\end{figure}
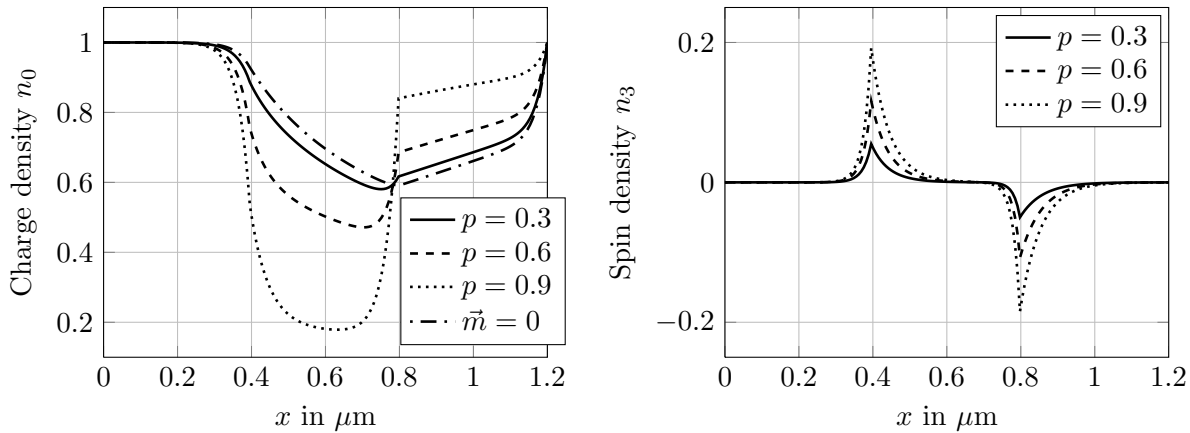
}

Next, we consider a smaller device with length $L=0.4\,\mu$m and a higher level
of doping, $C_{\rm max}=9\cdot 10^{21}\,$m$^{-3}$. With these parameters,
the scaled Debye length is the same as in the previous case. As a consequence,
the charge density without magnetization does not change. The influence of the
spin polarization is similar as in the previous test case (see Figure
\ref{fig.small}), but the charge density is larger in the ferromagnetic layer
due to the reduced size.

{\small
\begin{figure}[ht]
  \input{n0_smaller.tikz} 
  \caption{Charge density $n_0$ in the three-layer
	structure with self-consistent potential and smaller device length $L=0.4\,\mu$m.}
  \label{fig.small}
\end{figure}
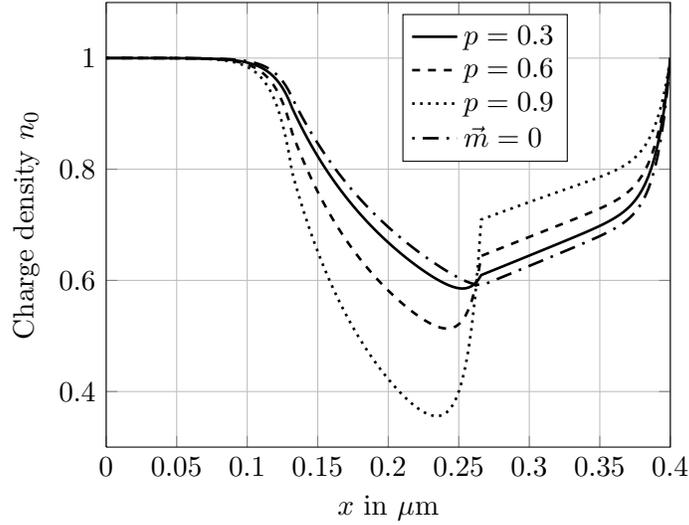
}

Our final example is concerned with a transient simulation. As initial data,
we choose $V=0$, $n_0=1$, and $\vec n=0$. 
Figure \ref{fig.entropy} presents the time decay of the entropy (free energy) $H_0$,
defined in \eqref{H0}. It turns out that $H_0(t)$ is decaying exponentially
fast. For times $t>270$\,ps, the equilibrium state is almost reached, and
since we reached the computer precision of about
$10^{-16}$, we observe some oscillations in the values for $H_0$.

{\small
\begin{figure}[ht]
  \input{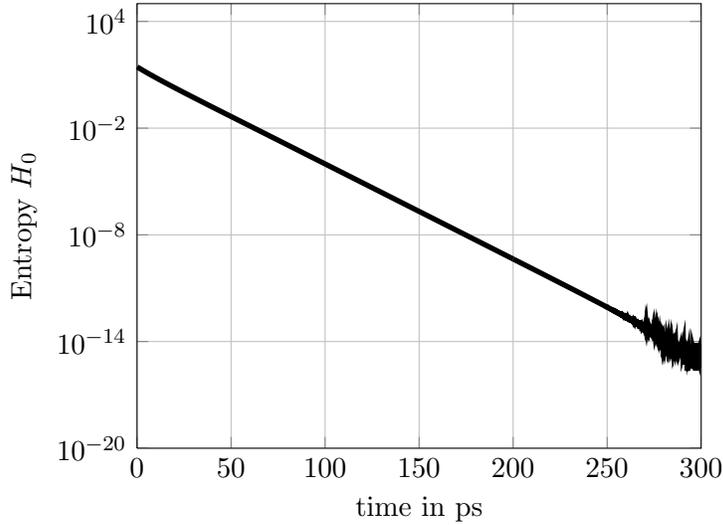}
  \caption{Semilogarithmic plot of the entropy $H_0(t)$ versus time.}
  \label{fig.entropy}
\end{figure}
}


\begin{appendix}
\section{A boundedness result}

We prove the following lemma which is an extension of a result due
to Kowalczyk \cite{Kow05}, based on the iteration technique of Alikakos \cite{Ali79}.

\begin{lemma}\label{lem.bd}
Let $\Omega\subset\R^d$ ($d\ge 1$) be a bounded domain and
let $u_i\ge 0$, $u_i^{q/2}\in L^2(0,T;H^1(\Omega))$ for all $2\le q<\infty$,
$i=1,\ldots,n$. Suppose that there exist constants $K_0$, $K_1$, $K_2$ such that
for all $q\ge 2$,
\begin{equation}\label{a.ineq}
  \frac{d}{dt} \sum_{i=1}^n\int_\Omega u_i^{q+1} dx 
	+ K_0\sum_{i=1}^n\int_\Omega |\na u_i^{(q+1)/2}|^2 dx
	\le q(q+1)K_1\sum_{i=1}^n\int_\Omega u_i^{q+1} dx + q(q+1)K_2.
\end{equation}
Then for all $0\le t\le T$,
$$
  \sum_{i=1}^n\|u_i(t)\|_{L^\infty(\Omega)}
	\le K^*
	\max\left\{1,\sum_{i=1}^n\|u_i\|_{L^\infty(0,T;L^1(\Omega))},
	n\sum_{i=1}^n\|u_i(0)\|_{L^\infty(\Omega)}\right\},
$$
where $K^*=2^{7+2d}(1+\max\{K_1,K_2\})$.
\end{lemma}

\begin{proof}
Note that the proof does not follow directly from the Gronwall lemma and the limit
$q\to \infty$ because of the quadratic term in $q$ on the right-hand side.
First, we estimate the $L^{q+1}$ norm of $u_i$. 
By the Gagliardo-Nirenberg inequality with 
$\theta=d/(d+2)<1$ and the Young inequality 
$ab\le \theta\eps a^{1/\theta} + (1-\theta)\eps^{-\theta/(1-\theta)}b^{1/(1-\theta)}$
for $0<\eps<1$, it follows that
\begin{align*}
  \int_\Omega u_i^{q+1} dx &= \|u_i^{(q+1)/2}\|_{L^2(\Omega)}^2
	\le K_3\|u_i^{(q+1)/2}\|_{H^1(\Omega)}^{2\theta}
	\|u_i^{(q+1)/2}\|_{L^1(\Omega)}^{2(1-\theta)} \\
	&\le \eps\big(\|\na u_i^{(q+1)/2}\|_{L^2(\Omega)}^2
	+ \|u_i^{(q+1)/2}\|_{L^2(\Omega)}^2\big)
	+ \frac{K_3^{1/(1-\theta)}}{\eps^{\theta/(1-\theta)}}
	\|u_i^{(q+1)/2}\|_{L^1(\Omega)}^2.
\end{align*}
Since $1/(1-\theta)=1+d/2$ and $\theta/(1-\theta)=d/2$, this gives
$$
  \int_\Omega u_i^{q+1} dx 
	\le \frac{\eps}{1-\eps}\|\na u_i^{(q+1)/2}\|_{L^2(\Omega)}^2
	+ \frac{K_3^{1+d/2}}{\eps^{d/2}(1-\eps)}\|u_i^{(q+1)/2}\|_{L^1(\Omega)}^2.
$$
We choose $\eps>0$ such that $(\eps/(1-\eps))(q(q+1)K_1+(q+1)^{-1})=K_0$
(then $\eps=O(q^{-2})$ as $q\to\infty$).
Hence, there exists a constant $K_4>0$ which is independent of $q$ such that
$K_3^{1+d/2}/(\eps^{d/2}(1-\eps))\le K_4(q+1)^d$.
Multiplying the above inequality by $q(q+1)K_1+(q+1)^{-1}$, we conclude that
\begin{align*}
  \big(q & (q+1) K_1+(q+1)^{-1}\big)\int_\Omega u_i^{q+1} dx \\
	&\le K_0\|\na u_i^{(q+1)/2}\|_{L^2(\Omega)}^2 
	+ K_4(q+1)^d\big(q(q+1)K_1+(q+1)^{-1}\big)\|u_i^{(q+1)/2}\|_{L^1(\Omega)}^2.
\end{align*}
We insert this estimate in \eqref{a.ineq}:
\begin{align*}
  \frac{d}{dt} & \sum_{i=1}^n \int_\Omega u_i^{q+1} dx 
	+ \frac{1}{q+1} \sum_{i=1}^n\int_\Omega u_i^{q+1} dx \\
	&\le K_4(q+1)^d\big(q(q+1)K_1+(q+1)^{-1}\big)
	\sum_{i=1}^n\|u_i^{(q+1)/2}\|_{L^1(\Omega)}^2 + q(q+1)K_2.
\end{align*}
Setting $q=\lambda_m=2^m-1$ for $m\in\N$, $\eps_m=2^{-m}$, 
$a_m=2^m(2^m-1)\max\{K_1,K_2\}$, and
$c_m=2^{dm}K_4$, the above inequality can be written as
\begin{equation}\label{a.aux}
  \frac{d}{dt}\sum_{i=1}^n\int_\Omega u_i^{\lambda_m+1}dx
	+ \eps_m\sum_{i=1}^n \int_\Omega u_i^{\lambda_m+1}dx
	\le (a_m+\eps_m)c_m\left(\sum_{i=1}^n
	\sup_{t>0}\int_\Omega u_i^{\lambda_{m-1}+1}dx\right)^2 + a_m.
\end{equation}
If $n=1$, the result follows directly from Lemma 5.1 in \cite{Kow05}.
An inspection of the proof of Lemma 5.1 shows that the conclusion still holds
with a slightly different $L^\infty$ bound which comes from the fact that
\begin{align*}
  \sum_{i=1}^n\int_\Omega u_i(0)^{\lambda_k+1}dx
	&\le n^{\lambda_k}\left(\sum_{i=1}^n\|u_i(0)\|_{L^{\lambda_k+1}(\Omega)}
	\right)^{\lambda_k+1} \\
	&\le \mbox{meas}(\Omega)\left(n\sum_{i=1}^n\|u_i(0)\|_{L^\infty(\Omega)}
	\right)^{\lambda_k+1}.
\end{align*}
This ends the proof.
\end{proof}
\end{appendix}


\end{document}

%% file: n0_linpot.tikz
%
%
%
%
\begin{tikzpicture}

\begin{axis}[%
width=\denslinwidth,
height=0.788709677419355\denslinwidth,
scale only axis,
xmin=0,
xmax=1.2,
xlabel={$x$ in $\mu$m},
xmajorgrids,
ymin=0.1,
ymax=1.1,
ylabel={Charge density $n_0$},
ymajorgrids,
legend style={at={(0.687202380952377,0.048492063492064)},anchor=south west,draw=black,fill=white,legend cell align=left}
]
\addplot [
color=black,
solid,
line width=1.0pt
]
table[row sep=crcr]{
0 1.00000002904727\\
0.00670391061452514 0.999999964192468\\
0.0134078212290503 0.999999911727579\\
0.0201117318435754 0.99999986400305\\
0.0268156424581006 0.999999804796914\\
0.0335195530726257 0.999999727642006\\
0.0402234636871508 0.999999629615564\\
0.046927374301676 0.999999507132886\\
0.0536312849162011 0.999999354925796\\
0.0603351955307262 0.999999165937837\\
0.0670391061452514 0.999998931226324\\
0.0737430167597765 0.999998639677612\\
0.0804469273743017 0.999998277531004\\
0.0871508379888268 0.999997827729582\\
0.0938547486033519 0.999997269099233\\
0.100558659217877 0.99999657533409\\
0.107262569832402 0.999995713746277\\
0.113966480446927 0.99999464372081\\
0.120670391061453 0.999993314800642\\
0.127374301675978 0.999991664309375\\
0.134078212290503 0.999989614398054\\
0.140782122905028 0.999987068375965\\
0.147486033519553 0.999983906152106\\
0.154189944134078 0.999979978572285\\
0.160893854748603 0.999975100384886\\
0.167597765363128 0.999969041503652\\
0.174301675977654 0.99996151615555\\
0.181005586592179 0.999952169401973\\
0.187709497206704 0.999940560397636\\
0.194413407821229 0.999926141597625\\
0.201117318435754 0.999908232931906\\
0.207821229050279 0.999885989729254\\
0.214525139664804 0.999858362877716\\
0.22122905027933 0.999824049342563\\
0.227932960893855 0.999781430707865\\
0.23463687150838 0.999728496842981\\
0.241340782122905 0.999662751093665\\
0.24804469273743 0.999581092526043\\
0.254748603351955 0.999479669669482\\
0.26145251396648 0.999353698860024\\
0.268156424581006 0.999197238616508\\
0.274860335195531 0.999002909407682\\
0.281564245810056 0.998761545593008\\
0.288268156424581 0.998461763120789\\
0.294972067039106 0.998089422593867\\
0.301675977653631 0.997626962378166\\
0.308379888268156 0.997052570299793\\
0.315083798882682 0.996339154863398\\
0.321787709497207 0.995453067468865\\
0.328491620111732 0.994352515359004\\
0.335195530726257 0.99298559044423\\
0.341899441340782 0.991287821032773\\
0.348603351955307 0.989179130992639\\
0.355307262569832 0.986560062922739\\
0.362011173184358 0.98330708719728\\
0.368715083798883 0.979266775632252\\
0.375418994413408 0.974248564972285\\
0.382122905027933 0.968015768884252\\
0.388826815642458 0.96027441453409\\
0.395530726256983 0.950659377218705\\
0.402234636871508 0.945712528887743\\
0.408938547486033 0.941367583529679\\
0.415642458100559 0.93755131468521\\
0.422346368715084 0.934199406782494\\
0.429050279329609 0.931255371327506\\
0.435754189944134 0.928669595048648\\
0.442458100558659 0.926398503962317\\
0.449162011173184 0.924403829282026\\
0.455865921787709 0.922651962812816\\
0.462569832402235 0.921113390984311\\
0.46927374301676 0.919762198005226\\
0.475977653631285 0.918575629792364\\
0.48268156424581 0.917533711357777\\
0.489385474860335 0.916618911246716\\
0.49608938547486 0.915815847421771\\
0.502793296089385 0.915111029699284\\
0.509497206703911 0.914492634475186\\
0.516201117318436 0.913950308040072\\
0.522905027932961 0.913474995288102\\
0.529608938547486 0.913058791080359\\
0.536312849162011 0.91269481193985\\
0.543016759776536 0.912377086140571\\
0.549720670391061 0.912100460615369\\
0.556424581005586 0.911860523454882\\
0.563128491620112 0.911653541110942\\
0.569832402234637 0.911476409761183\\
0.576536312849162 0.91132662064641\\
0.583240223463687 0.911202239568623\\
0.589944134078212 0.911101901145966\\
0.596648044692737 0.911024818873599\\
0.603351955307263 0.910970812549122\\
0.610055865921788 0.910940355201926\\
0.616759776536313 0.910934642332483\\
0.623463687150838 0.910955687035056\\
0.630167597765363 0.911006445459965\\
0.636871508379888 0.911090978080971\\
0.643575418994413 0.91121465337586\\
0.650279329608939 0.911384401799807\\
0.656983240223464 0.911609029311066\\
0.663687150837989 0.911899601147838\\
0.670391061452514 0.912269907961996\\
0.677094972067039 0.9127370276292\\
0.683798882681564 0.913321996814445\\
0.690502793296089 0.914050606261456\\
0.697206703910615 0.914954332150584\\
0.70391061452514 0.91607141175552\\
0.710614525139665 0.917448063554964\\
0.71731843575419 0.91913983773912\\
0.724022346368715 0.921213059474075\\
0.73072625698324 0.923746289634466\\
0.737430167597765 0.926831669116766\\
0.744134078212291 0.930575923333329\\
0.750837988826816 0.935100668653292\\
0.757541899441341 0.940541461635213\\
0.764245810055866 0.94704473504211\\
0.770949720670391 0.954761329036738\\
0.777653631284916 0.963834690345947\\
0.784357541899441 0.974380888981672\\
0.791061452513967 0.986456266319268\\
0.797765363128492 1.00000660178386\\
0.804469273743017 1.00000721154187\\
0.811173184357542 1.00000787416997\\
0.817877094972067 1.00000859379538\\
0.824581005586592 1.00000937482042\\
0.831284916201117 1.00001022193542\\
0.837988826815642 1.00001114013137\\
0.844692737430167 1.00001213471226\\
0.851396648044693 1.00001321130675\\
0.858100558659218 1.00001437587902\\
0.864804469273743 1.00001563473854\\
0.871508379888268 1.00001699454816\\
0.878212290502793 1.00001846233037\\
0.884916201117318 1.00002004547105\\
0.891620111731844 1.00002175172009\\
0.898324022346369 1.00002358918823\\
0.905027932960894 1.00002556633918\\
0.911731843575419 1.00002769197623\\
0.918435754189944 1.00002997522173\\
0.925139664804469 1.00003242548866\\
0.931843575418994 1.00003505244206\\
0.93854748603352 1.00003786594896\\
0.945251396648045 1.00004087601413\\
0.95195530726257 1.00004409269943\\
0.958659217877095 1.00004752602345\\
0.96536312849162 1.00005118583812\\
0.972067039106145 1.00005508167805\\
0.97877094972067 1.00005922257776\\
0.985474860335195 1.00006361685142\\
0.992178770949721 1.00006827182833\\
0.998882681564246 1.00007319353679\\
1.00558659217877 1.00007838632755\\
1.0122905027933 1.00008385242642\\
1.01899441340782 1.00008959140466\\
1.02569832402235 1.00009559955285\\
1.03240223463687 1.00010186914287\\
1.0391061452514 1.00010838755912\\
1.04581005586592 1.00011513627794\\
1.05251396648045 1.00012208967028\\
1.05921787709497 1.00012921359921\\
1.0659217877095 1.00013646377914\\
1.07262569832402 1.00014378385854\\
1.07932960893855 1.00015110318212\\
1.08603351955307 1.00015833418157\\
1.0927374301676 1.00016536933613\\
1.09944134078212 1.00017207763531\\
1.10614525139665 1.00017830046571\\
1.11284916201117 1.00018384683227\\
1.1195530726257 1.00018848781042\\
1.12625698324022 1.00019195011026\\
1.13296089385475 1.00019390861604\\
1.13966480446927 1.00019397774389\\
1.1463687150838 1.00019170143725\\
1.15307262569832 1.00018654159308\\
1.15977653631285 1.00017786468119\\
1.16648044692737 1.00016492628442\\
1.1731843575419 1.00014685324779\\
1.17988826815642 1.00012262307902\\
1.18659217877095 1.00009104019198\\
1.19329608938547 1.00005070852504\\
1.2 1\\
};
\addlegendentry{$p = 0.3$};

\addplot [
color=black,
dashed,
line width=1.0pt
]
table[row sep=crcr]{
0 1.00000007250146\\
0.00670391061452514 0.999999862955622\\
0.0134078212290503 0.999999688015878\\
0.0201117318435754 0.999999481652483\\
0.0268156424581006 0.999999227201277\\
0.0335195530726257 0.999998911091828\\
0.0402234636871508 0.999998517793462\\
0.046927374301676 0.99999802873704\\
0.0536312849162011 0.999997421020042\\
0.0603351955307262 0.999996666084111\\
0.0670391061452514 0.99999572832389\\
0.0737430167597765 0.999994563445853\\
0.0804469273743017 0.999993116422972\\
0.0871508379888268 0.999991318925435\\
0.0938547486033519 0.999989086108225\\
0.100558659217877 0.999986312611209\\
0.107262569832402 0.999982867588282\\
0.113966480446927 0.999978588533563\\
0.120670391061453 0.999973273614312\\
0.127374301675978 0.999966672149123\\
0.134078212290503 0.999958472782315\\
0.140782122905028 0.999948288796781\\
0.147486033519553 0.999935639872633\\
0.154189944134078 0.999919929431283\\
0.160893854748603 0.99990041649645\\
0.167597765363128 0.999876180744873\\
0.174301675977654 0.999846079098301\\
0.181005586592179 0.999808691809373\\
0.187709497206704 0.999762255498438\\
0.194413407821229 0.999704579982925\\
0.201117318435754 0.999632944976435\\
0.207821229050279 0.999543971785279\\
0.214525139664804 0.999433463950954\\
0.22122905027933 0.999296209322359\\
0.227932960893855 0.999125734222377\\
0.23463687150838 0.998913998113972\\
0.241340782122905 0.99865101436459\\
0.24804469273743 0.998324379222015\\
0.254748603351955 0.997918686785624\\
0.26145251396648 0.997414802379868\\
0.268156424581006 0.996788960058321\\
0.274860335195531 0.996011641671614\\
0.281564245810056 0.995046184630005\\
0.288268156424581 0.993847052695021\\
0.294972067039106 0.992357688241251\\
0.301675977653631 0.990507844689224\\
0.308379888268156 0.988210273292274\\
0.315083798882682 0.985356608008153\\
0.321787709497207 0.981812254363426\\
0.328491620111732 0.977410041241496\\
0.335195530726257 0.971942336177865\\
0.341899441340782 0.965151252276794\\
0.348603351955307 0.956716484853963\\
0.355307262569832 0.946240204114634\\
0.362011173184358 0.933228291323345\\
0.368715083798883 0.917067033460187\\
0.375418994413408 0.896994177156162\\
0.382122905027933 0.872062976652751\\
0.388826815642458 0.841097540090724\\
0.395530726256983 0.802637368014838\\
0.402234636871508 0.782849958678763\\
0.408938547486033 0.765470161425616\\
0.415642458100559 0.750205070213451\\
0.422346368715084 0.736797422556373\\
0.429050279329609 0.725021264283302\\
0.435754189944134 0.714678142113973\\
0.442458100558659 0.705593759918856\\
0.449162011173184 0.697615042353323\\
0.455865921787709 0.690607556433064\\
0.462569832402235 0.68445324766401\\
0.46927374301676 0.679048452658136\\
0.475977653631285 0.674302154847117\\
0.48268156424581 0.670134454028606\\
0.489385474860335 0.66647522411553\\
0.49608938547486 0.663262936670051\\
0.502793296089385 0.660443630646479\\
0.509497206703911 0.657970011291642\\
0.516201117318436 0.655800663402085\\
0.522905027932961 0.653899366156222\\
0.529608938547486 0.652234498564128\\
0.536312849162011 0.650778526243538\\
0.543016759776536 0.649507561771742\\
0.549720670391061 0.648400992312282\\
0.556424581005586 0.647441169605502\\
0.563128491620112 0.646613158776421\\
0.569832402234637 0.645904543786846\\
0.576536312849162 0.64530528877797\\
0.583240223463687 0.644807656054733\\
0.589944134078212 0.644406183097199\\
0.596648044692737 0.644097722794684\\
0.603351955307263 0.643881553137094\\
0.610055865921788 0.643759564920876\\
0.616759776536313 0.643736538693455\\
0.623463687150838 0.643820525229981\\
0.630167597765363 0.644023347366748\\
0.636871508379888 0.644361245053465\\
0.643575418994413 0.644855690056329\\
0.650279329608939 0.645534401830227\\
0.656983240223464 0.646432601597875\\
0.663687150837989 0.647594547431348\\
0.670391061452514 0.649075398758184\\
0.677094972067039 0.650943463570288\\
0.683798882681564 0.653282884651599\\
0.690502793296089 0.656196820697862\\
0.697206703910615 0.659811171707539\\
0.70391061452514 0.66427888156489\\
0.710614525139665 0.669784818440331\\
0.71731843575419 0.676551176768379\\
0.724022346368715 0.684843250256274\\
0.73072625698324 0.694975274767162\\
0.737430167597765 0.70731580552957\\
0.744134078212291 0.722291735078328\\
0.750837988826816 0.740389518997056\\
0.757541899441341 0.762151372850489\\
0.764245810055866 0.788163016281871\\
0.770949720670391 0.819027797883665\\
0.777653631284916 0.855319492021719\\
0.784357541899441 0.897502365996477\\
0.791061452513967 0.94580177279147\\
0.797765363128492 1.00000081847647\\
0.804469273743017 1.00000085170856\\
0.811173184357542 1.00000088862163\\
0.817877094972067 1.00000092958015\\
0.824581005586592 1.00000097497976\\
0.831284916201117 1.00000102524934\\
0.837988826815642 1.00000108085309\\
0.844692737430167 1.00000114229269\\
0.851396648044693 1.00000121010935\\
0.858100558659218 1.00000128488598\\
0.864804469273743 1.0000013672492\\
0.871508379888268 1.00000145787125\\
0.878212290502793 1.00000155747184\\
0.884916201117318 1.0000016668197\\
0.891620111731844 1.00000178673386\\
0.898324022346369 1.00000191808464\\
0.905027932960894 1.00000206179407\\
0.911731843575419 1.00000221883586\\
0.918435754189944 1.00000239023447\\
0.925139664804469 1.00000257706352\\
0.931843575418994 1.00000278044294\\
0.93854748603352 1.00000300153495\\
0.945251396648045 1.00000324153848\\
0.95195530726257 1.00000350168169\\
0.958659217877095 1.00000378321242\\
0.96536312849162 1.00000408738592\\
0.972067039106145 1.00000441544976\\
0.97877094972067 1.00000476862493\\
0.985474860335195 1.00000514808305\\
0.992178770949721 1.00000555491845\\
0.998882681564246 1.00000599011487\\
1.00558659217877 1.0000064545054\\
1.0122905027933 1.00000694872487\\
1.01899441340782 1.00000747315356\\
1.02569832402235 1.00000802785047\\
1.03240223463687 1.00000861247518\\
1.0391061452514 1.00000922619594\\
1.04581005586592 1.00000986758242\\
1.05251396648045 1.00001053448052\\
1.05921787709497 1.00001122386675\\
1.0659217877095 1.0000119316792\\
1.07262569832402 1.00001265262166\\
1.07932960893855 1.00001337993714\\
1.08603351955307 1.00001410514624\\
1.0927374301676 1.00001481774574\\
1.09944134078212 1.00001550486151\\
1.10614525139665 1.0000161508496\\
1.11284916201117 1.00001673683826\\
1.1195530726257 1.00001724020277\\
1.12625698324022 1.000017633964\\
1.13296089385475 1.00001788610018\\
1.13966480446927 1.00001795876034\\
1.1463687150838 1.00001780736597\\
1.15307262569832 1.00001737958621\\
1.15977653631285 1.00001661416936\\
1.16648044692737 1.0000154396119\\
1.1731843575419 1.00001377264323\\
1.17988826815642 1.00001151650193\\
1.18659217877095 1.00000855897576\\
1.19329608938547 1.00000477017449\\
1.2 1\\
};
\addlegendentry{$p = 0.6$};

\addplot [
color=black,
dotted,
line width=1.0pt
]
table[row sep=crcr]{
0 0.999999940248728\\
0.00670391061452514 0.999999668603295\\
0.0134078212290503 0.999999317613994\\
0.0201117318435754 0.99999886237632\\
0.0268156424581006 0.999998288239068\\
0.0335195530726257 0.999997572716888\\
0.0402234636871508 0.999996684029594\\
0.046927374301676 0.999995580884545\\
0.0536312849162011 0.999994211326591\\
0.0603351955307262 0.999992510681235\\
0.0670391061452514 0.999990398663547\\
0.0737430167597765 0.999987775629794\\
0.0804469273743017 0.999984517866083\\
0.0871508379888268 0.999980471731411\\
0.0938547486033519 0.999975446395626\\
0.100558659217877 0.999969204832797\\
0.107262569832402 0.999961452642355\\
0.113966480446927 0.999951824167026\\
0.120670391061453 0.999939865250261\\
0.127374301675978 0.999925011818743\\
0.134078212290503 0.999906563279562\\
0.140782122905028 0.999883649477494\\
0.147486033519553 0.999855189654182\\
0.154189944134078 0.999819841473734\\
0.160893854748603 0.999775937710656\\
0.167597765363128 0.999721407614079\\
0.174301675977654 0.999653679239441\\
0.181005586592179 0.999569558141093\\
0.187709497206704 0.999465076704339\\
0.194413407821229 0.999335307010629\\
0.201117318435754 0.999174128409584\\
0.207821229050279 0.998973938835316\\
0.214525139664804 0.998725296251152\\
0.22122905027933 0.998416473311315\\
0.227932960893855 0.998032904234964\\
0.23463687150838 0.99755649780416\\
0.241340782122905 0.996964784082959\\
0.24804469273743 0.996229854612229\\
0.254748603351955 0.995317046093964\\
0.26145251396648 0.99418330548046\\
0.268156424581006 0.992775159356978\\
0.274860335195531 0.991026191842883\\
0.281564245810056 0.98885391205526\\
0.288268156424581 0.986155863387458\\
0.294972067039106 0.982804791094826\\
0.301675977653631 0.978642640264567\\
0.308379888268156 0.973473101081092\\
0.315083798882682 0.967052349780793\\
0.321787709497207 0.95907754858903\\
0.328491620111732 0.949172562233373\\
0.335195530726257 0.936870217345723\\
0.341899441340782 0.921590268009625\\
0.348603351955307 0.902612028187402\\
0.355307262569832 0.879040380222589\\
0.362011173184358 0.849763556192628\\
0.368715083798883 0.813400700849366\\
0.375418994413408 0.768236742928675\\
0.382122905027933 0.712141503003704\\
0.388826815642458 0.642469222565451\\
0.395530726256983 0.555933775572257\\
0.402234636871508 0.51141207372533\\
0.408938547486033 0.472307502791813\\
0.415642458100559 0.437961023722995\\
0.422346368715084 0.407793795526449\\
0.429050279329609 0.381297420966478\\
0.435754189944134 0.358025379854228\\
0.442458100558659 0.337585505627317\\
0.449162011173184 0.31963337852213\\
0.455865921787709 0.303866524114538\\
0.462569832402235 0.290019319608667\\
0.46927374301676 0.277858522219374\\
0.475977653631285 0.267179344525251\\
0.48268156424581 0.257802010945361\\
0.489385474860335 0.249568737673051\\
0.49608938547486 0.242341085625501\\
0.502793296089385 0.235997642363653\\
0.509497206703911 0.230431994616609\\
0.516201117318436 0.225550958109061\\
0.522905027932961 0.221273035932527\\
0.529608938547486 0.217527080806379\\
0.536312849162011 0.214251140322988\\
0.543016759776536 0.211391467738721\\
0.549720670391061 0.208901684133337\\
0.556424581005586 0.206742080888133\\
0.563128491620112 0.204879054503142\\
0.569832402234637 0.203284668863933\\
0.576536312849162 0.201936343262097\\
0.583240223463687 0.200816667859718\\
0.589944134078212 0.199913351964726\\
0.596648044692737 0.199219314557496\\
0.603351955307263 0.198732931096226\\
0.610055865921788 0.198458455855311\\
0.616759776536313 0.198406645050372\\
0.623463687150838 0.198595612911069\\
0.630167597765363 0.19905196080661\\
0.636871508379888 0.199812228613806\\
0.643575418994413 0.200924727799593\\
0.650279329608939 0.202451827134364\\
0.656983240223464 0.204472774371134\\
0.663687150837989 0.207087150180501\\
0.670391061452514 0.210419063291197\\
0.677094972067039 0.214622206712489\\
0.683798882681564 0.219885901749336\\
0.690502793296089 0.226442255525388\\
0.697206703910615 0.234574543116803\\
0.70391061452514 0.244626888369208\\
0.710614525139665 0.257015244804287\\
0.71731843575419 0.272239550076737\\
0.724022346368715 0.290896715250726\\
0.73072625698324 0.313693771294344\\
0.737430167597765 0.341459967807575\\
0.744134078212291 0.375155813394351\\
0.750837988826816 0.415875833584319\\
0.757541899441341 0.464840013924921\\
0.764245810055866 0.523366224183272\\
0.770949720670391 0.592811999268992\\
0.777653631284916 0.67446833202252\\
0.784357541899441 0.769379824215524\\
0.791061452513967 0.878053520052643\\
0.797765363128492 1.00000140754077\\
0.804469273743017 1.00000143208972\\
0.811173184357542 1.00000145972518\\
0.817877094972067 1.00000149078855\\
0.824581005586592 1.00000152565359\\
0.831284916201117 1.00000156472885\\
0.837988826815642 1.00000160846013\\
0.844692737430167 1.00000165733302\\
0.851396648044693 1.00000171187542\\
0.858100558659218 1.00000177266012\\
0.864804469273743 1.00000184030742\\
0.871508379888268 1.00000191548748\\
0.878212290502793 1.00000199892276\\
0.884916201117318 1.00000209139016\\
0.891620111731844 1.00000219372292\\
0.898324022346369 1.00000230681221\\
0.905027932960894 1.00000243160814\\
0.911731843575419 1.00000256912041\\
0.918435754189944 1.00000272041796\\
0.925139664804469 1.00000288662796\\
0.931843575418994 1.00000306893346\\
0.93854748603352 1.00000326856994\\
0.945251396648045 1.00000348682002\\
0.95195530726257 1.00000372500637\\
0.958659217877095 1.00000398448228\\
0.96536312849162 1.00000426661947\\
0.972067039106145 1.0000045727928\\
0.97877094972067 1.00000490436104\\
0.985474860335195 1.00000526264342\\
0.992178770949721 1.00000564889092\\
0.998882681564246 1.0000060642516\\
1.00558659217877 1.00000650972915\\
1.0122905027933 1.0000069861332\\
1.01899441340782 1.00000749402063\\
1.02569832402235 1.00000803362605\\
1.03240223463687 1.00000860478025\\
1.0391061452514 1.00000920681443\\
1.04581005586592 1.00000983844851\\
1.05251396648045 1.00001049766104\\
1.05921787709497 1.00001118153805\\
1.0659217877095 1.00001188609811\\
1.07262569832402 1.00001260609004\\
1.07932960893855 1.00001333475968\\
1.08603351955307 1.00001406358135\\
1.0927374301676 1.00001478194945\\
1.09944134078212 1.00001547682453\\
1.10614525139665 1.00001613232808\\
1.11284916201117 1.00001672927902\\
1.1195530726257 1.00001724466424\\
1.12625698324022 1.00001765103471\\
1.13296089385475 1.00001791581711\\
1.13966480446927 1.00001800053031\\
1.1463687150838 1.00001785989405\\
1.15307262569832 1.00001744081593\\
1.15977653631285 1.00001668124089\\
1.16648044692737 1.00001550884507\\
1.1731843575419 1.00001383955403\\
1.17988826815642 1.00001157586186\\
1.18659217877095 1.00000860492517\\
1.19329608938547 1.0000047964018\\
1.2 1\\
};
\addlegendentry{$p = 0.9$};

\end{axis}
\end{tikzpicture}%

%% file: n3_linpot.tikz
%
%
%
%
\begin{tikzpicture}

\begin{axis}[%
width=\spdenslinwidth,
height=0.788709677419355\spdenslinwidth,
scale only axis,
xmin=0,
xmax=1.2,
xlabel={$x$ in $\mu$m},
xmajorgrids,
ymin=-0.25,
ymax=0.25,
ylabel={Spin density $n_3$},
ymajorgrids,
legend style={at={(0.572916666666664,0.61468253968254)},anchor=south west,draw=black,fill=white,legend cell align=left}
]
\addplot [
color=black,
solid,
line width=1.0pt
]
table[row sep=crcr]{
0 3.95118539802027e-10\\
0.00670391061452514 3.95118539802027e-10\\
0.0134078212290503 4.15025104462849e-10\\
0.0201117318435754 4.60659247042564e-10\\
0.0268156424581006 5.40547060225252e-10\\
0.0335195530726257 6.67003996606734e-10\\
0.0402234636871508 8.57672496249151e-10\\
0.046927374301676 1.13770015499165e-09\\
0.0536312849162011 1.542823205968e-09\\
0.0603351955307262 2.12372978192823e-09\\
0.0670391061452514 2.95223253737387e-09\\
0.0737430167597765 4.12999954425158e-09\\
0.0804469273743017 5.8009025276899e-09\\
0.0871508379888268 8.16848008607432e-09\\
0.0938547486033519 1.15206337221623e-08\\
0.100558659217877 1.62645514992681e-08\\
0.107262569832402 2.29760942684256e-08\\
0.113966480446927 3.24696330621376e-08\\
0.120670391061453 4.5896806058454e-08\\
0.127374301675978 6.48861701776477e-08\\
0.134078212290503 9.17406810867818e-08\\
0.140782122905028 1.29716947429889e-07\\
0.147486033519553 1.83420120727422e-07\\
0.154189944134078 2.59362303929997e-07\\
0.160893854748603 3.66752189268438e-07\\
0.167597765363128 5.18611674042919e-07\\
0.174301675977654 7.33354851123019e-07\\
0.181005586592179 1.03702083677327e-06\\
0.187709497206704 1.46643118031235e-06\\
0.194413407821229 2.07365471148797e-06\\
0.201117318435754 2.9323212165152e-06\\
0.207821229050279 4.14654951585965e-06\\
0.214525139664804 5.86357252954744e-06\\
0.22122905027933 8.2915901988397e-06\\
0.227932960893855 1.17250150436701e-05\\
0.23463687150838 1.65801715389627e-05\\
0.241340782122905 2.34457780844328e-05\\
0.24804469273743 3.31543328251839e-05\\
0.254748603351955 4.68830593041938e-05\\
0.26145251396648 6.62966522438753e-05\\
0.268156424581006 9.37491322832368e-05\\
0.274860335195531 0.000132569285830926\\
0.281564245810056 0.000187464301417033\\
0.288268156424581 0.000265090545997639\\
0.294972067039106 0.000374860691439784\\
0.301675977653631 0.000530085060364906\\
0.308379888268156 0.00074958558673446\\
0.315083798882682 0.00105997809408039\\
0.321787709497207 0.00149889963201376\\
0.328491620111732 0.00211957220592285\\
0.335195530726257 0.00299725628087719\\
0.341899441340782 0.00423837658779805\\
0.348603351955307 0.00599342679346415\\
0.355307262569832 0.00847521780700074\\
0.362011173184358 0.0119846824451346\\
0.368715083798883 0.0169473654346674\\
0.375418994413408 0.0239650233951688\\
0.382122905027933 0.0338885916253115\\
0.388826815642458 0.0479213653760133\\
0.395530726256983 0.0677649069906809\\
0.402234636871508 0.0595200380313324\\
0.408938547486033 0.0522783116824758\\
0.415642458100559 0.0459176768685079\\
0.422346368715084 0.0403309323090742\\
0.429050279329609 0.0354239197625913\\
0.435754189944134 0.0311139370950936\\
0.442458100558659 0.0273283444283698\\
0.449162011173184 0.0240033398754414\\
0.455865921787709 0.0210828842295603\\
0.462569832402235 0.0185177564832651\\
0.46927374301676 0.0162647242588801\\
0.475977653631285 0.014285815168364\\
0.48268156424581 0.0125476768211932\\
0.489385474860335 0.011021014692649\\
0.49608938547486 0.00968009837659564\\
0.502793296089385 0.00850232789861866\\
0.509497206703911 0.00746785277658105\\
0.516201117318436 0.00655923740313607\\
0.522905027932961 0.00576116710330963\\
0.529608938547486 0.00506018990277427\\
0.536312849162011 0.00444448964003259\\
0.543016759776536 0.0039036865779553\\
0.549720670391061 0.00342866212505278\\
0.556424581005586 0.00301140467112747\\
0.563128491620112 0.00264487388074235\\
0.569832402234637 0.00232288107488484\\
0.576536312849162 0.00203998356821532\\
0.583240223463687 0.00179139101622341\\
0.589944134078212 0.00157288196076355\\
0.596648044692737 0.00138072883779546\\
0.603351955307263 0.00121162971727505\\
0.610055865921788 0.00106264496557974\\
0.616759776536313 0.000931136830865016\\
0.623463687150838 0.000814709615033089\\
0.630167597765363 0.000711147560078162\\
0.636871508379888 0.000618346766234587\\
0.643575418994413 0.000534236267358234\\
0.650279329608939 0.000456681662878498\\
0.656983240223464 0.000383362229243215\\
0.663687150837989 0.000311608900668297\\
0.670391061452514 0.000238185485441406\\
0.677094972067039 0.000158988356171772\\
0.683798882681564 6.86297518136055e-05\\
0.690502793296089 -4.01444723698177e-05\\
0.697206703910615 -0.000177272050837555\\
0.70391061452514 -0.00035652450930741\\
0.710614525139665 -0.0005971289070578\\
0.71731843575419 -0.000926056954906393\\
0.724022346368715 -0.00138125836594996\\
0.73072625698324 -0.00201622987695883\\
0.737430167597765 -0.00290647341284605\\
0.744134078212291 -0.00415862600387525\\
0.750837988826816 -0.00592336809065559\\
0.757541899441341 -0.00841367505659873\\
0.764245810055866 -0.0119306247724112\\
0.770949720670391 -0.0168998901904835\\
0.777653631284916 -0.0239233416908162\\
0.784357541899441 -0.033852016062288\\
0.791061452513967 -0.0478892998676627\\
0.797765363128492 -0.0677368386466968\\
0.804469273743017 -0.0594953881835908\\
0.811173184357542 -0.05225666405109\\
0.817877094972067 -0.0458986658599562\\
0.824581005586592 -0.0403142368481244\\
0.831284916201117 -0.0354092578759945\\
0.837988826815642 -0.0311010611559601\\
0.844692737430167 -0.0273170369814001\\
0.851396648044693 -0.0239934099731348\\
0.858100558659218 -0.0210741642183803\\
0.864804469273743 -0.0185100991866451\\
0.871508379888268 -0.016258000511104\\
0.878212290502793 -0.0142799116599236\\
0.884916201117318 -0.0125424942223836\\
0.891620111731844 -0.0110164660281525\\
0.898324022346369 -0.00967610762985716\\
0.905027932960894 -0.00849882883127481\\
0.911731843575419 -0.00746478795547682\\
0.918435754189944 -0.00655655743612268\\
0.925139664804469 -0.00575883009583159\\
0.931843575418994 -0.00505816116128888\\
0.93854748603352 -0.0044427416670513\\
0.945251396648045 -0.00390219942903429\\
0.95195530726257 -0.00342742423332329\\
0.958659217877095 -0.00301041429407428\\
0.96536312849162 -0.00264414139273838\\
0.972067039106145 -0.00232243242570316\\
0.97877094972067 -0.00203986536399384\\
0.985474860335195 -0.00179167787158746\\
0.992178770949721 -0.00157368704225165\\
0.998882681564246 -0.00138221890222942\\
1.00558659217877 -0.00121404649071028\\
1.0122905027933 -0.00106633547463445\\
1.01899441340782 -0.000936596381413982\\
1.02569832402235 -0.000822642644767032\\
1.03240223463687 -0.000722553756943314\\
1.0391061452514 -0.000634642906832435\\
1.04581005586592 -0.000557428559266522\\
1.05251396648045 -0.00048960949755649\\
1.05921787709497 -0.00043004291010003\\
1.0659217877095 -0.000377725153811253\\
1.07262569832402 -0.000331774873095981\\
1.07932960893855 -0.000291418194011769\\
1.08603351955307 -0.000255975749944725\\
1.0927374301676 -0.000224851328433072\\
1.09944134078212 -0.000197521959525576\\
1.10614525139665 -0.000173529295216959\\
1.11284916201117 -0.000152472158132172\\
1.1195530726257 -0.000134000167052903\\
1.12625698324022 -0.000117808378766651\\
1.13296089385475 -0.000103632922274261\\
1.13966480446927 -9.12476455879734e-05\\
1.1463687150838 -8.04618512643103e-05\\
1.15307262569832 -7.1119270097611e-05\\
1.15977653631285 -6.30985209462938e-05\\
1.16648044692737 -5.63154395551075e-05\\
1.1731843575419 -5.0727846065342e-05\\
1.17988826815642 -4.63435816545702e-05\\
1.18659217877095 -4.32330104464158e-05\\
1.19329608938547 -4.15476973047954e-05\\
1.2 -4.15476973047954e-05\\
};
\addlegendentry{$p = 0.3$};

\addplot [
color=black,
dashed,
line width=1.0pt
]
table[row sep=crcr]{
0 7.90237332127949e-10\\
0.00670391061452514 7.90237332127949e-10\\
0.0134078212290503 8.30050475238555e-10\\
0.0201117318435754 9.21318792046505e-10\\
0.0268156424581006 1.08109447392474e-09\\
0.0335195530726257 1.33400843478959e-09\\
0.0402234636871508 1.71534556732878e-09\\
0.046927374301676 2.27540108121407e-09\\
0.0536312849162011 3.08564746839754e-09\\
0.0603351955307262 4.24746103096599e-09\\
0.0670391061452514 5.90446712997013e-09\\
0.0737430167597765 8.26000198328148e-09\\
0.0804469273743017 1.16018091462461e-08\\
0.0871508379888268 1.63369659648585e-08\\
0.0938547486033519 2.30412756564662e-08\\
0.100558659217877 3.25291146483421e-08\\
0.107262569832402 4.59522050691701e-08\\
0.113966480446927 6.49392895892391e-08\\
0.120670391061453 9.1793645423315e-08\\
0.127374301675978 1.29772387629925e-07\\
0.134078212290503 1.83481429270668e-07\\
0.140782122905028 2.59433990083461e-07\\
0.147486033519553 3.66840376583187e-07\\
0.154189944134078 5.18724799596886e-07\\
0.160893854748603 7.33504650569985e-07\\
0.167597765363128 1.03722373400359e-06\\
0.174301675977654 1.46671024967188e-06\\
0.181005586592179 2.07404244999807e-06\\
0.187709497206704 2.9328634618167e-06\\
0.194413407821229 4.14731098458613e-06\\
0.201117318435754 5.86464464737071e-06\\
0.207821229050279 8.29310217135683e-06\\
0.214525139664804 1.17271495103176e-05\\
0.22122905027933 1.65831867079073e-05\\
0.227932960893855 2.34500390322891e-05\\
0.23463687150838 3.31603557567488e-05\\
0.241340782122905 4.68915741389217e-05\\
0.24804469273743 6.63086911181071e-05\\
0.254748603351955 9.37661546997192e-05\\
0.26145251396648 0.000132593355630959\\
0.268156424581006 0.000187498337034526\\
0.274860335195531 0.000265138674340396\\
0.281564245810056 0.000374928748309016\\
0.288268156424581 0.000530181298092921\\
0.294972067039106 0.000749721674847379\\
0.301675977653631 0.00106017053432443\\
0.308379888268156 0.00149917175932846\\
0.315083798882682 0.00211995701799553\\
0.321787709497207 0.00299780043938486\\
0.328491620111732 0.00423914607651937\\
0.335195530726257 0.00599451491935282\\
0.341899441340782 0.00847675651441384\\
0.348603351955307 0.0119868583151555\\
0.355307262569832 0.016950442309567\\
0.362011173184358 0.0239693743713976\\
0.368715083798883 0.0338947442944226\\
0.375418994413408 0.0479300657993276\\
0.382122905027933 0.0677772101651164\\
0.388826815642458 0.0958427688585249\\
0.395530726256983 0.135529867932168\\
0.402234636871508 0.119040123080139\\
0.408938547486033 0.104556664258712\\
0.415642458100559 0.0918353892140631\\
0.422346368715084 0.0806618952949699\\
0.429050279329609 0.0708478659383429\\
0.435754189944134 0.062227896805594\\
0.442458100558659 0.0546567080777374\\
0.449162011173184 0.0480066959253114\\
0.455865921787709 0.0421657818854623\\
0.462569832402235 0.0370355238992468\\
0.46927374301676 0.0325294571719163\\
0.475977653631285 0.0285716368919785\\
0.48268156424581 0.0250953582463997\\
0.489385474860335 0.0220420321566765\\
0.49608938547486 0.0193601977839398\\
0.502793296089385 0.0170046551548293\\
0.509497206703911 0.0149357032822396\\
0.516201117318436 0.0131184709300173\\
0.522905027932961 0.0115223287278219\\
0.529608938547486 0.0101203727073997\\
0.536312849162011 0.00888897052669394\\
0.543016759776536 0.00780736269268379\\
0.549720670391061 0.00685731200370354\\
0.556424581005586 0.00602279522053697\\
0.563128491620112 0.00528973165315432\\
0.569832402234637 0.00464574392384857\\
0.576536312849162 0.00407994664155111\\
0.583240223463687 0.00358275909596995\\
0.589944134078212 0.00314573834849966\\
0.596648044692737 0.00276142924755418\\
0.603351955307263 0.00242322790822009\\
0.610055865921788 0.00212525503700692\\
0.616759776536313 0.00186223510250078\\
0.623463687150838 0.00162937667930031\\
0.630167597765363 0.00142224822079056\\
0.636871508379888 0.00123664189564479\\
0.643575418994413 0.00106841573896599\\
0.650279329608939 0.000913300916812529\\
0.656983240223464 0.000766655950055711\\
0.663687150837989 0.000623142677328515\\
0.670391061452514 0.000476288689756782\\
0.677094972067039 0.000317886714519643\\
0.683798882681564 0.000137161223261041\\
0.690502793296089 -8.03960618165489e-05\\
0.697206703910615 -0.0003546605713529\\
0.70391061452514 -0.000713175279798837\\
0.710614525139665 -0.00119439417316788\\
0.71731843575419 -0.00185226046159917\\
0.724022346368715 -0.00276267324844063\\
0.73072625698324 -0.00403262552825529\\
0.737430167597765 -0.0058131204539562\\
0.744134078212291 -0.00831743108602995\\
0.750837988826816 -0.0118469168852146\\
0.757541899441341 -0.0168275266160138\\
0.764245810055866 -0.0238614132137727\\
0.770949720670391 -0.0337999186685409\\
0.777653631284916 -0.0478467782981592\\
0.784357541899441 -0.0677040581359217\\
0.791061452513967 -0.0957785208716995\\
0.797765363128492 -0.135473443175933\\
0.804469273743017 -0.1189905637223\\
0.811173184357542 -0.104513135001264\\
0.817877094972067 -0.0917971564023031\\
0.824581005586592 -0.080628314555312\\
0.831284916201117 -0.070818371322244\\
0.837988826815642 -0.062201991257163\\
0.844692737430167 -0.0546339550651549\\
0.851396648044693 -0.0479867120961059\\
0.858100558659218 -0.0421482306234182\\
0.864804469273743 -0.0370201096765507\\
0.871508379888268 -0.0325159206044632\\
0.878212290502793 -0.0285597504189098\\
0.884916201117318 -0.0250849223672778\\
0.891620111731844 -0.0220328721716894\\
0.898324022346369 -0.0193521609946483\\
0.905027932960894 -0.0169976084958925\\
0.911731843575419 -0.0149295313691127\\
0.918435754189944 -0.0131130745249374\\
0.925139664804469 -0.011517623648041\\
0.931843575418994 -0.0101162892276916\\
0.93854748603352 -0.00888545336566935\\
0.945251396648045 -0.00780437172352343\\
0.95195530726257 -0.00685482390045299\\
0.958659217877095 -0.00602080634934547\\
0.96536312849162 -0.00528826265544263\\
0.972067039106145 -0.00464484663182166\\
0.97877094972067 -0.00407971423897914\\
0.985474860335195 -0.00358334082162524\\
0.992178770949721 -0.00314736058251279\\
0.998882681564246 -0.00276442558794519\\
1.00558659217877 -0.00242808192884502\\
1.0122905027933 -0.00213266095047762\\
1.01899441340782 -0.0018731837179983\\
1.02569832402235 -0.00164527710821663\\
1.03240223463687 -0.00144510011413456\\
1.0391061452514 -0.00126927912124242\\
1.04581005586592 -0.00111485106619638\\
1.05251396648045 -0.000979213521956573\\
1.05921787709497 -0.000860080871062662\\
1.0659217877095 -0.000755445832547123\\
1.07262569832402 -0.000663545699934638\\
1.07932960893855 -0.000582832729606172\\
1.08603351955307 -0.000511948192192097\\
1.0927374301676 -0.000449699666254471\\
1.09944134078212 -0.000395041215034853\\
1.10614525139665 -0.000347056145351859\\
1.11284916201117 -0.000304942104992107\\
1.1195530726257 -0.000267998333780974\\
1.12625698324022 -0.000235614947293253\\
1.13296089385475 -0.00020726420527469\\
1.13966480446927 -0.000182493805237165\\
1.1463687150838 -0.000160922353514356\\
1.15307262569832 -0.000142237312626858\\
1.15977653631285 -0.000126195920897241\\
1.16648044692737 -0.000112629850036435\\
1.1731843575419 -0.000101454740078294\\
1.17988826815642 -9.26862725356558e-05\\
1.18659217877095 -8.64651740453557e-05\\
1.19329608938547 -8.30945717124705e-05\\
1.2 -8.30945717124705e-05\\
};
\addlegendentry{$p = 0.6$};

\addplot [
color=black,
dotted,
line width=1.0pt
]
table[row sep=crcr]{
0 1.1853568181585e-09\\
0.00670391061452514 1.1853568181585e-09\\
0.0134078212290503 1.24507657409537e-09\\
0.0201117318435754 1.38197914391655e-09\\
0.0268156424581006 1.62164283235556e-09\\
0.0335195530726257 2.00101403581486e-09\\
0.0402234636871508 2.57302012989453e-09\\
0.046927374301676 3.41310398122513e-09\\
0.0536312849162011 4.62847440180215e-09\\
0.0603351955307262 6.3711959498095e-09\\
0.0670391061452514 8.85670681565448e-09\\
0.0737430167597765 1.23900115368454e-08\\
0.0804469273743017 1.74027257445598e-08\\
0.0871508379888268 2.45054658795917e-08\\
0.0938547486033519 3.45619373645816e-08\\
0.100558659217877 4.8793705684239e-08\\
0.107262569832402 6.89283552249033e-08\\
0.113966480446927 9.74090016795392e-08\\
0.120670391061453 1.3769056325667e-07\\
0.127374301675978 1.9465871591866e-07\\
0.134078212290503 2.75222334029849e-07\\
0.140782122905028 3.8915125394502e-07\\
0.147486033519553 5.50260944978548e-07\\
0.154189944134078 7.78087736864141e-07\\
0.160893854748603 1.10025773585021e-06\\
0.167597765363128 1.55583667566924e-06\\
0.174301675977654 2.20006689413527e-06\\
0.181005586592179 3.11106582383373e-06\\
0.187709497206704 4.39929823130634e-06\\
0.194413407821229 6.22097077361966e-06\\
0.201117318435754 8.79697304691846e-06\\
0.207821229050279 1.24396618486958e-05\\
0.214525139664804 1.75907364146443e-05\\
0.22122905027933 2.48747972415822e-05\\
0.227932960893855 3.5175082841692e-05\\
0.23463687150838 4.97405679873937e-05\\
0.241340782122905 7.03374097847568e-05\\
0.24804469273743 9.9463105367343e-05\\
0.254748603351955 0.00014064932918198\\
0.26145251396648 0.000198890170797946\\
0.268156424581006 0.00028124769977563\\
0.274860335195531 0.00039770828615542\\
0.281564245810056 0.000562393510828493\\
0.288268156424581 0.000795272496311381\\
0.294972067039106 0.00112458328883355\\
0.301675977653631 0.00159025689959215\\
0.308379888268156 0.00224875919177784\\
0.315083798882682 0.00317993772271995\\
0.321787709497207 0.00449670376395685\\
0.328491620111732 0.00635872350524699\\
0.335195530726257 0.00899177858738232\\
0.341899441340782 0.0127151435505499\\
0.348603351955307 0.0179802998865751\\
0.355307262569832 0.0254256810181201\\
0.362011173184358 0.0359540863789437\\
0.368715083798883 0.0508421515408429\\
0.375418994413408 0.0718951483307983\\
0.382122905027933 0.101665885429132\\
0.388826815642458 0.143764252527211\\
0.395530726256983 0.203294942226677\\
0.402234636871508 0.178560307874192\\
0.408938547486033 0.156835104642129\\
0.415642458100559 0.137753178897507\\
0.422346368715084 0.120992926441916\\
0.429050279329609 0.106271872236077\\
0.435754189944134 0.0933419096011164\\
0.442458100558659 0.0819851186583404\\
0.449162011173184 0.0720100935315836\\
0.455865921787709 0.0632487164106227\\
0.462569832402235 0.0555533241051725\\
0.46927374301676 0.0487942193336024\\
0.475977653631285 0.0428574848000238\\
0.48268156424581 0.0376430632157709\\
0.489385474860335 0.0330630709023474\\
0.49608938547486 0.0290403165480615\\
0.502793296089385 0.0255070001459415\\
0.509497206703911 0.022403570174065\\
0.516201117318436 0.0196777197419135\\
0.522905027932961 0.0172835047620555\\
0.529608938547486 0.0151805692540194\\
0.536312849162011 0.0133334646799889\\
0.543016759776536 0.0117110517786458\\
0.549720670391061 0.010285974728285\\
0.556424581005586 0.00903419865313336\\
0.563128491620112 0.00793460250317149\\
0.569832402234637 0.00696862019858643\\
0.576536312849162 0.00611992364102465\\
0.583240223463687 0.00537414175460185\\
0.589944134078212 0.00471861012209044\\
0.596648044692737 0.00414214600776062\\
0.603351955307263 0.00363484357671551\\
0.610055865921788 0.00318788388188534\\
0.616759776536313 0.00279335361988027\\
0.623463687150838 0.00244406564677253\\
0.630167597765363 0.00213337263713281\\
0.636871508379888 0.00185496283865024\\
0.643575418994413 0.00160262329869465\\
0.650279329608939 0.00136995076092962\\
0.656983240223464 0.0011499830008748\\
0.663687150837989 0.000934712770044296\\
0.670391061452514 0.000714431447658289\\
0.677094972067039 0.000476828115460479\\
0.683798882681564 0.00020573946960902\\
0.690502793296089 -0.000120596921199035\\
0.697206703910615 -0.00053199422226005\\
0.70391061452514 -0.00106976692133831\\
0.710614525139665 -0.00179159603316009\\
0.71731843575419 -0.00277839642232209\\
0.724022346368715 -0.0041440168126881\\
0.73072625698324 -0.00604894679259535\\
0.737430167597765 -0.00871969122720164\\
0.744134078212291 -0.0124761598988589\\
0.750837988826816 -0.0177703922696177\\
0.757541899441341 -0.0252413118698514\\
0.764245810055866 -0.0357921486462164\\
0.770949720670391 -0.0506999163546508\\
0.777653631284916 -0.0717702190653656\\
0.784357541899441 -0.101556157379638\\
0.791061452513967 -0.143667877533419\\
0.797765363128492 -0.203210297655508\\
0.804469273743017 -0.178485962121626\\
0.811173184357542 -0.156769804693932\\
0.817877094972067 -0.137695824210744\\
0.824581005586592 -0.120942550401081\\
0.831284916201117 -0.106227625868453\\
0.837988826815642 -0.0933030472777145\\
0.844692737430167 -0.0819509855405958\\
0.851396648044693 -0.0719801145538871\\
0.858100558659218 -0.0632223866153275\\
0.864804469273743 -0.0555302001704999\\
0.871508379888268 -0.0487739121563148\\
0.878212290502793 -0.0428396530144741\\
0.884916201117318 -0.0376274075494324\\
0.891620111731844 -0.0330493292859117\\
0.898324022346369 -0.0290282599163725\\
0.905027932960894 -0.0254964288854152\\
0.911731843575419 -0.0223943111940892\\
0.918435754189944 -0.019669624173694\\
0.925139664804469 -0.0172764463208484\\
0.931843575418994 -0.0151744433427918\\
0.93854748603352 -0.013328188368805\\
0.945251396648045 -0.0117065648706942\\
0.95195530726257 -0.0102822422292607\\
0.958659217877095 -0.00903121510804974\\
0.96536312849162 -0.00793239887107795\\
0.972067039106145 -0.00696727422581745\\
0.97877094972067 -0.0061195751023629\\
0.985474860335195 -0.00537501450843955\\
0.992178770949721 -0.00472104373998609\\
0.998882681564246 -0.00414664088927544\\
1.00558659217877 -0.00364212508639381\\
1.0122905027933 -0.00319899334371743\\
1.01899441340782 -0.00280977725413812\\
1.02569832402235 -0.00246791712862603\\
1.03240223463687 -0.00216765145296344\\
1.0391061452514 -0.00190391980212443\\
1.04581005586592 -0.00167227757823454\\
1.05251396648045 -0.00146882113822851\\
1.05921787709497 -0.00129012205372011\\
1.0659217877095 -0.001133169401334\\
1.07262569832402 -0.000995319119671422\\
1.07932960893855 -0.000874249591827066\\
1.08603351955307 -0.000767922722453137\\
1.0927374301676 -0.000674549878260362\\
1.09944134078212 -0.00059256215312017\\
1.10614525139665 -0.000520584506394005\\
1.11284916201117 -0.000457413409005016\\
1.1195530726257 -0.000401997720031543\\
1.12625698324022 -0.000353422612262416\\
1.13296089385475 -0.000310896474820429\\
1.13966480446927 -0.000273740853548159\\
1.1463687150838 -0.000241383657586516\\
1.15307262569832 -0.000213356080419663\\
1.15977653631285 -0.000189293979297487\\
1.16648044692737 -0.00016894486161834\\
1.1731843575419 -0.000152182187338274\\
1.17988826815642 -0.000139029478718215\\
1.18659217877095 -0.000129697825812349\\
1.19329608938547 -0.000124641919516106\\
1.2 -0.000124641919516106\\
};
\addlegendentry{$p = 0.9$};

\end{axis}
\end{tikzpicture}%

%% file: n0_artifact.tikz
%
%
%
%
\begin{tikzpicture}

\begin{axis}[%
width=\artifactwidth,
height=0.788709677419355\artifactwidth,
scale only axis,
xmin=0,
xmax=1.2,
xlabel={$x$ in $\mu$m},
xmajorgrids,
ymin=0.1,
ymax=1.1,
ylabel={Charge density $n_0$},
ymajorgrids,
legend style={at={(0.648273809523809,0.036587301587302)},anchor=south west,draw=black,fill=white,legend cell align=left}
]
\addplot [
color=black,
dotted,
line width=1.0pt
]
table[row sep=crcr]{
0 1\\
0.00805369127516778 0.999999636598013\\
0.0161073825503356 0.999999164891564\\
0.0241610738255034 0.999998552602723\\
0.0322147651006711 0.999997757833776\\
0.0402684563758389 0.999996726200262\\
0.0483221476510067 0.999995387109544\\
0.0563758389261745 0.999993648930293\\
0.0644295302013423 0.99999139272235\\
0.0724832214765101 0.99998846409789\\
0.0805369127516778 0.999984662656968\\
0.0885906040268456 0.99997972827457\\
0.0966442953020134 0.999973323300776\\
0.104697986577181 0.999965009456075\\
0.112751677852349 0.999954217840791\\
0.120805369127517 0.999940210006443\\
0.128859060402685 0.999922027425233\\
0.136912751677852 0.999898425899958\\
0.14496644295302 0.99986779042616\\
0.153020134228188 0.999828024680689\\
0.161073825503356 0.999776407574636\\
0.169127516778524 0.999709407054841\\
0.177181208053691 0.999622438412779\\
0.185234899328859 0.999509550562432\\
0.193288590604027 0.999363018819781\\
0.201342281879195 0.99917281631865\\
0.209395973154362 0.998925927893002\\
0.21744966442953 0.998605459476072\\
0.225503355704698 0.998189482074355\\
0.233557046979866 0.99764953121199\\
0.241610738255034 0.996948659165651\\
0.249664429530201 0.996038906708378\\
0.257718120805369 0.994858021359063\\
0.265771812080537 0.993325197574176\\
0.273825503355705 0.991335547391902\\
0.281879194630872 0.988752923166322\\
0.28993288590604 0.985400601266236\\
0.297986577181208 0.981049189243424\\
0.306040268456376 0.975400928982996\\
0.314093959731544 0.968069321732897\\
0.322147651006711 0.958552680795633\\
0.330201342281879 0.946199802147717\\
0.338255033557047 0.93016540389853\\
0.346308724832215 0.909352285403132\\
0.354362416107383 0.882336248097076\\
0.36241610738255 0.847268640541589\\
0.370469798657718 0.801749859038641\\
0.378523489932886 0.742665147725279\\
0.386577181208054 0.665971462288108\\
0.394630872483221 0.566420812823828\\
0.402684563758389 0.511979356601186\\
0.410738255033557 0.465389639411009\\
0.418791946308725 0.425519420822379\\
0.426845637583893 0.391399781491508\\
0.43489932885906 0.362201579407209\\
0.442953020134228 0.337215304475786\\
0.451006711409396 0.315833842221615\\
0.459060402684564 0.297537728193228\\
0.467114093959732 0.281882535342464\\
0.475167785234899 0.268488088670654\\
0.483221476510067 0.257029246087684\\
0.491275167785235 0.247228022808314\\
0.499328859060403 0.238846869668709\\
0.50738255033557 0.23168294431452\\
0.515436241610738 0.225563239018067\\
0.523489932885906 0.220340450572212\\
0.531543624161074 0.215889496864158\\
0.539597315436242 0.21210460188841\\
0.547651006711409 0.208896886618865\\
0.555704697986577 0.206192417815007\\
0.563758389261745 0.203930680978511\\
0.571812080536913 0.202063457814496\\
0.579865771812081 0.20055410323232\\
0.587919463087248 0.199377232742324\\
0.595973154362416 0.19851884873279\\
0.604026845637584 0.197976954289647\\
0.612080536912752 0.197762726775703\\
0.620134228187919 0.197902351211021\\
0.628187919463087 0.198439646514589\\
0.636241610738255 0.199439656742664\\
0.644295302013423 0.200993425225251\\
0.652348993288591 0.203224222063979\\
0.660402684563758 0.206295553875242\\
0.668456375838926 0.210421346112013\\
0.676510067114094 0.215878746667418\\
0.684563758389262 0.223024043139442\\
0.69261744966443 0.232312194465292\\
0.700671140939597 0.244320415351886\\
0.708724832214765 0.259776060444579\\
0.716778523489933 0.279588639000331\\
0.724832214765101 0.304884996541662\\
0.732885906040268 0.337045283713689\\
0.740939597315436 0.37773491002292\\
0.748993288590604 0.428923649602067\\
0.757046979865772 0.492876490618263\\
0.76510067114094 0.572090246589335\\
0.773154362416107 0.669133132041313\\
0.781208053691275 0.786317988737952\\
0.789261744966443 0.925098318167202\\
0.797315436241611 1.08501158125122\\
0.805369127516778 1.08501214706277\\
0.813422818791946 1.08501277032524\\
0.821476510067114 1.08501345419326\\
0.829530201342282 1.08501420113092\\
0.83758389261745 1.08501501253299\\
0.845637583892617 1.0850158882176\\
0.853691275167785 1.08501682575087\\
0.861744966442953 1.08501781955249\\
0.869798657718121 1.08501885971607\\
0.877852348993289 1.08501993045808\\
0.885906040268456 1.08502100808427\\
0.893959731543624 1.08502205832877\\
0.902013422818792 1.08502303287867\\
0.91006711409396 1.08502386484101\\
0.918120805369128 1.08502446283728\\
0.926174496644295 1.08502470331674\\
0.934228187919463 1.08502442055924\\
0.942281879194631 1.08502339368047\\
0.950335570469799 1.08502132974917\\
0.958389261744966 1.08501784186144\\
0.966442953020134 1.08501242067418\\
0.974496644295302 1.08500439745532\\
0.98255033557047 1.08499289613109\\
0.990604026845638 1.0849767710625\\
0.998657718120805 1.08495452631175\\
1.00671140939597 1.08492421089975\\
1.01476510067114 1.08488328292091\\
1.02281879194631 1.08482843326054\\
1.03087248322148 1.08475535690788\\
1.03892617449664 1.08465845628661\\
1.04697986577181 1.08453045639015\\
1.05503355704698 1.08436190549555\\
1.06308724832215 1.08414052742517\\
1.07114093959732 1.08385038119776\\
1.07919463087248 1.0834707707662\\
1.08724832214765 1.08297483048177\\
1.09530201342282 1.08232768978616\\
1.10335570469799 1.08148409190046\\
1.11140939597315 1.08038530399078\\
1.11946308724832 1.07895510789171\\
1.12751677852349 1.07709459765391\\
1.13557046979866 1.07467542865391\\
1.14362416107383 1.07153105718668\\
1.15167785234899 1.06744537211868\\
1.15973154362416 1.06213794191521\\
1.16778523489933 1.05524486897936\\
1.1758389261745 1.04629394292134\\
1.18389261744966 1.0346723945752\\
1.19194630872483 1.01958504662674\\
1.2 1\\
};
\addlegendentry{$M = 150$};

\addplot [
color=black,
dashed,
line width=1.0pt
]
table[row sep=crcr]{
0 1\\
0.0040133779264214 0.999999834513423\\
0.00802675585284281 0.999999646158924\\
0.0120401337792642 0.999999431776477\\
0.0160535117056856 0.999999187769384\\
0.020066889632107 0.999998910043935\\
0.0240802675585284 0.999998593940724\\
0.0280936454849498 0.999998234156482\\
0.0321070234113712 0.999997824655106\\
0.0361204013377926 0.999997358566381\\
0.040133779264214 0.999996828070732\\
0.0441471571906354 0.999996224268022\\
0.0481605351170569 0.999995537028246\\
0.0521739130434783 0.999994754821569\\
0.0561872909698997 0.999993864524898\\
0.0602006688963211 0.999992851201711\\
0.0642140468227425 0.999991697851469\\
0.0682274247491639 0.999990385124396\\
0.0722408026755853 0.999988890996851\\
0.0762541806020067 0.99998719040183\\
0.0802675585284281 0.999985254808424\\
0.0842809364548495 0.999983051743153\\
0.0882943143812709 0.999980544245156\\
0.0923076923076923 0.999977690246101\\
0.0963210702341137 0.999974441864401\\
0.100334448160535 0.999970744601905\\
0.104347826086957 0.999966536429586\\
0.108361204013378 0.999961746746872\\
0.112374581939799 0.999956295197184\\
0.116387959866221 0.999950090319792\\
0.120401337792642 0.999943028015374\\
0.124414715719064 0.999934989799545\\
0.128428093645485 0.99992584081504\\
0.132441471571906 0.999915427569211\\
0.136454849498328 0.99990357535888\\
0.140468227424749 0.999890085339343\\
0.144481605351171 0.999874731188346\\
0.148494983277592 0.999857255309079\\
0.152508361204013 0.999837364508465\\
0.156521739130435 0.999814725078259\\
0.160535117056856 0.999788957196421\\
0.164548494983278 0.999759628554834\\
0.168561872909699 0.999726247106467\\
0.17257525083612 0.999688252810294\\
0.176588628762542 0.999645008235479\\
0.180602006688963 0.999595787867192\\
0.184615384615385 0.999539765934634\\
0.188628762541806 0.999476002557072\\
0.192642140468227 0.999403427975446\\
0.196655518394649 0.999320824605008\\
0.20066889632107 0.999226806607887\\
0.204682274247492 0.999119796642869\\
0.208695652173913 0.998997999402322\\
0.212709030100334 0.998859371492297\\
0.216722408026756 0.998701587150481\\
0.220735785953177 0.998521999226852\\
0.224749163879599 0.998317594772409\\
0.22876254180602 0.998084944490886\\
0.232775919732441 0.997820145205408\\
0.236789297658863 0.997518754374836\\
0.240802675585284 0.997175715561201\\
0.244816053511706 0.996785273597779\\
0.248829431438127 0.996340878034579\\
0.252842809364549 0.995835073241365\\
0.25685618729097 0.995259373324446\\
0.260869565217391 0.994604119758728\\
0.264882943143813 0.993858319346522\\
0.268896321070234 0.993009459784511\\
0.272909698996655 0.992043299744684\\
0.276923076923077 0.990943629947386\\
0.280936454849498 0.989692001218036\\
0.28494983277592 0.988267414965064\\
0.288963210702341 0.986645970886294\\
0.292976588628762 0.984800465993251\\
0.296989966555184 0.982699938226273\\
0.301003344481605 0.980309147003639\\
0.305016722408027 0.977587981989875\\
0.309030100334448 0.974490790164134\\
0.31304347826087 0.970965609898889\\
0.317056856187291 0.96695329919901\\
0.321070234113712 0.962386543475781\\
0.325083612040134 0.957188726209153\\
0.329096989966555 0.951272643551412\\
0.333110367892977 0.944539041307094\\
0.337123745819398 0.936874949744071\\
0.341137123745819 0.928151788298894\\
0.345150501672241 0.91822320837903\\
0.349163879598662 0.90692263807063\\
0.353177257525084 0.894060487559386\\
0.357190635451505 0.879420968379773\\
0.361204013377926 0.862758473129178\\
0.365217391304348 0.843793454909344\\
0.369230769230769 0.822207737364451\\
0.373244147157191 0.797639176632295\\
0.377257525083612 0.769675585652039\\
0.381270903010033 0.737847818896593\\
0.385284280936455 0.701621901512047\\
0.389297658862876 0.660390070814666\\
0.393311036789298 0.613460579848518\\
0.397324414715719 0.560046091937876\\
0.40133779264214 0.532343405895605\\
0.405351170568562 0.506713028100744\\
0.409364548494983 0.482999957026307\\
0.413377926421405 0.461060787233466\\
0.417391304347826 0.44076284218279\\
0.421404682274247 0.421983371944149\\
0.425418060200669 0.404608810955013\\
0.42943143812709 0.388534091340249\\
0.433444816053512 0.373662007642817\\
0.437458193979933 0.359902629126049\\
0.441471571906354 0.347172756096179\\
0.445484949832776 0.335395416960501\\
0.449498327759197 0.324499402983295\\
0.453511705685619 0.31441883793021\\
0.45752508361204 0.305092780003371\\
0.461538461538462 0.296464853665457\\
0.465551839464883 0.288482909132502\\
0.469565217391304 0.28109870748345\\
0.473578595317726 0.274267629490303\\
0.477591973244147 0.267948406417431\\
0.481605351170569 0.262102871172634\\
0.48561872909699 0.256695728317182\\
0.489632107023411 0.251694341557717\\
0.493645484949833 0.247068537450552\\
0.497658862876254 0.242790424149028\\
0.501672240802675 0.238834224118077\\
0.505685618729097 0.235176119827256\\
0.509698996655518 0.23179411151521\\
0.51371237458194 0.228667886195017\\
0.517725752508361 0.22577869714184\\
0.521739130434782 0.22310925317227\\
0.525752508361204 0.220643617089082\\
0.529765886287625 0.218367112726322\\
0.533779264214047 0.216266240088286\\
0.537792642140468 0.214328598132266\\
0.541806020066889 0.212542814799566\\
0.545819397993311 0.210898483952582\\
0.549832775919732 0.209386108928194\\
0.553846153846154 0.207997052469719\\
0.557859531772575 0.206723492851835\\
0.561872909698997 0.205558386065528\\
0.565886287625418 0.204495433983923\\
0.569899665551839 0.203529058485197\\
0.573913043478261 0.202654381566352\\
0.577926421404682 0.201867211541936\\
0.581939799331104 0.201164035485518\\
0.585953177257525 0.20054201813947\\
0.589966555183946 0.199999007591131\\
0.593979933110368 0.199533548091374\\
0.597993311036789 0.199144900475831\\
0.602006688963211 0.198833070740253\\
0.606020066889632 0.198598847420506\\
0.610033444816053 0.198443848535475\\
0.614046822742475 0.198370578968215\\
0.618060200668896 0.198382499288122\\
0.622073578595318 0.198484107155167\\
0.626086956521739 0.198681032596964\\
0.63010033444816 0.198980148611194\\
0.634113712374582 0.199389698719547\\
0.638127090301003 0.199919443284856\\
0.642140468227425 0.200580826599746\\
0.646153846153846 0.20138716696129\\
0.650167224080268 0.202353872159841\\
0.654180602006689 0.203498683027703\\
0.65819397993311 0.204841947909543\\
0.662207357859532 0.20640693112464\\
0.666220735785953 0.208220158681196\\
0.670234113712375 0.210311804662432\\
0.674247491638796 0.212716121815684\\
0.678260869565217 0.215471919916695\\
0.682274247491639 0.218623095422638\\
0.68628762541806 0.222219215730993\\
0.690301003344482 0.226316160978778\\
0.694314381270903 0.230976825685537\\
0.698327759197324 0.236271881584565\\
0.702341137123746 0.242280601599815\\
0.706354515050167 0.249091742983452\\
0.710367892976589 0.256804484970881\\
0.71438127090301 0.265529412734417\\
0.718394648829431 0.275389534670392\\
0.722408026755853 0.286521313819728\\
0.726421404682274 0.299075686101945\\
0.730434782608696 0.313219027540701\\
0.734448160535117 0.32913401915658\\
0.738461538461538 0.347020340928537\\
0.74247491638796 0.36709510421921\\
0.746488294314381 0.389592904127762\\
0.750501672240803 0.414765337894023\\
0.754515050167224 0.442879790889069\\
0.758528428093645 0.474217235607032\\
0.762541806020067 0.509068718591033\\
0.766555183946488 0.547730121877458\\
0.77056856187291 0.59049467498732\\
0.774581939799331 0.637642555361045\\
0.778595317725753 0.689426742781281\\
0.782608695652174 0.746054078545485\\
0.786622073578595 0.807660212787324\\
0.790635451505017 0.874276790866033\\
0.794648829431438 0.94578881666196\\
0.798662207357859 1.02187961781885\\
0.802675585284281 1.02188009153542\\
0.806688963210702 1.02188058487344\\
0.810702341137124 1.02188109853604\\
0.814715719063545 1.02188163323829\\
0.818729096989967 1.02188218970543\\
0.822742474916388 1.02188276867074\\
0.826755852842809 1.02188337087299\\
0.830769230769231 1.02188399705336\\
0.834782608695652 1.02188464795189\\
0.838795986622074 1.02188532430322\\
0.842809364548495 1.02188602683166\\
0.846822742474916 1.02188675624539\\
0.850836120401338 1.02188751322977\\
0.854849498327759 1.02188829843953\\
0.858862876254181 1.02188911248976\\
0.862876254180602 1.02188995594548\\
0.866889632107023 1.02189082930966\\
0.870903010033445 1.02189173300937\\
0.874916387959866 1.02189266737994\\
0.878929765886288 1.02189363264666\\
0.882943143812709 1.02189462890392\\
0.88695652173913 1.02189565609123\\
0.890969899665552 1.02189671396583\\
0.894983277591973 1.02189780207127\\
0.898996655518395 1.0218989197016\\
0.903010033444816 1.02190006586043\\
0.907023411371237 1.02190123921414\\
0.911036789297659 1.02190243803855\\
0.91505016722408 1.02190366015809\\
0.919063545150502 1.02190490287639\\
0.923076923076923 1.02190616289719\\
0.927090301003344 1.02190743623427\\
0.931103678929766 1.02190871810874\\
0.935117056856187 1.02191000283219\\
0.939130434782609 1.02191128367353\\
0.94314381270903 1.02191255270744\\
0.947157190635451 1.02191380064191\\
0.951170568561873 1.02191501662188\\
0.955183946488294 1.02191618800587\\
0.959197324414716 1.02191730011174\\
0.963210702341137 1.02191833592756\\
0.967224080267559 1.02191927578257\\
0.97123745819398 1.0219200969729\\
0.975250836120401 1.02192077333594\\
0.979264214046823 1.0219212747661\\
0.983277591973244 1.02192156666411\\
0.987290969899665 1.02192160931066\\
0.991304347826087 1.02192135715411\\
0.995317725752508 1.02192075800028\\
0.99933110367893 1.02191975209122\\
1.00334448160535 1.0219182710574\\
1.00735785953177 1.02191623672615\\
1.01137123745819 1.02191355976664\\
1.01538461538462 1.02191013814877\\
1.01939799331104 1.02190585539073\\
1.02341137123746 1.02190057856596\\
1.02742474916388 1.02189415603661\\
1.0314381270903 1.02188641487585\\
1.03545150501672 1.02187715793639\\
1.03946488294314 1.02186616051657\\
1.04347826086957 1.02185316656865\\
1.04749163879599 1.02183788438649\\
1.05150501672241 1.02181998170096\\
1.05551839464883 1.02179908010177\\
1.05953177257525 1.02177474869293\\
1.06354515050167 1.02174649687663\\
1.06755852842809 1.02171376614551\\
1.07157190635452 1.02167592074697\\
1.07558528428094 1.02163223706424\\
1.07959866220736 1.02158189153781\\
1.08361204013378 1.02152394692601\\
1.0876254180602 1.02145733667641\\
1.09163879598662 1.02138084714778\\
1.09565217391304 1.02129309738669\\
1.09966555183946 1.02119251612195\\
1.10367892976589 1.02107731559388\\
1.10769230769231 1.02094546178211\\
1.11170568561873 1.02079464053608\\
1.11571906354515 1.02062221904351\\
1.11973244147157 1.02042520199456\\
1.12374581939799 1.02020018171058\\
1.12775919732441 1.01994328140582\\
1.13177257525084 1.01965009063538\\
1.13578595317726 1.01931559185232\\
1.13979933110368 1.01893407684836\\
1.1438127090301 1.01849905168316\\
1.14782608695652 1.018003128515\\
1.15183946488294 1.01743790252658\\
1.15585284280936 1.01679381189048\\
1.15986622073579 1.01605997843475\\
1.16387959866221 1.01522402634717\\
1.16789297658863 1.01427187588815\\
1.17190635451505 1.01318750866517\\
1.17591973244147 1.01195270054511\\
1.17993311036789 1.01054671773962\\
1.18394648829431 1.00894597098212\\
1.18795986622074 1.00712362201372\\
1.19197324414716 1.00504913579689\\
1.19598662207358 1.00268777096732\\
1.2 1\\
};
\addlegendentry{$M = 300$};

\addplot [
color=black,
solid,
line width=1.0pt
]
table[row sep=crcr]{
0 1\\
0.00267260579064588 0.999999893976932\\
0.00534521158129176 0.999999778417593\\
0.00801781737193764 0.999999652464241\\
0.0106904231625835 0.999999515181982\\
0.0133630289532294 0.999999365551838\\
0.0160356347438753 0.999999202463172\\
0.0187082405345212 0.999999024705456\\
0.021380846325167 0.999998830959278\\
0.0240534521158129 0.999998619786553\\
0.0267260579064588 0.999998389619844\\
0.0293986636971047 0.999998138750733\\
0.0320712694877506 0.999997865317138\\
0.0347438752783964 0.999997567289491\\
0.0374164810690423 0.999997242455675\\
0.0400890868596882 0.999996888404601\\
0.0427616926503341 0.999996502508317\\
0.04543429844098 0.999996081902496\\
0.0481069042316258 0.999995623465183\\
0.0507795100222717 0.999995123793612\\
0.0534521158129176 0.99999457917896\\
0.0561247216035635 0.99999398557881\\
0.0587973273942093 0.999993338587148\\
0.0614699331848552 0.999992633401662\\
0.0641425389755011 0.999991864788094\\
0.066815144766147 0.999991027041391\\
0.0694877505567929 0.999990113943354\\
0.0721603563474388 0.99998911871649\\
0.0748329621380846 0.999988033973702\\
0.0775055679287305 0.999986851663457\\
0.0801781737193764 0.999985563010027\\
0.0828507795100223 0.999984158448347\\
0.0855233853006682 0.999982627553023\\
0.088195991091314 0.999980958960942\\
0.0908685968819599 0.999979140286936\\
0.0935412026726058 0.99997715803185\\
0.0962138084632517 0.999974997482344\\
0.0988864142538975 0.999972642601681\\
0.101559020044543 0.9999700759107\\
0.104231625835189 0.999967278358067\\
0.106904231625835 0.999964229178875\\
0.109576837416481 0.999960905740512\\
0.112249443207127 0.999957283374668\\
0.114922048997773 0.999953335194238\\
0.117594654788419 0.999949031893749\\
0.120267260579065 0.999944341531841\\
0.12293986636971 0.999939229294178\\
0.125612472160356 0.999933657235043\\
0.128285077951002 0.999927583995679\\
0.130957683741648 0.999920964497308\\
0.133630289532294 0.99991374960653\\
0.13630289532294 0.999905885770627\\
0.138975501113586 0.999897314620073\\
0.141648106904232 0.999887972535273\\
0.144320712694878 0.999877790174362\\
0.146993318485523 0.999866691958496\\
0.149665924276169 0.999854595510876\\
0.152338530066815 0.999841411045303\\
0.155011135857461 0.999827040699734\\
0.157683741648107 0.999811377809903\\
0.160356347438753 0.999794306117594\\
0.163028953229399 0.999775698907727\\
0.165701559020045 0.999755418067794\\
0.16837416481069 0.99973331306273\\
0.171046770601336 0.999709219817551\\
0.173719376391982 0.999682959499512\\
0.176391982182628 0.999654337190715\\
0.179064587973274 0.999623140441317\\
0.18173719376392 0.999589137692629\\
0.184409799554566 0.99955207655835\\
0.187082405345212 0.999511681951242\\
0.189755011135857 0.999467654041267\\
0.192427616926503 0.999419666030109\\
0.195100222717149 0.999367361725491\\
0.197772828507795 0.999310352897338\\
0.200445434298441 0.999248216396125\\
0.203118040089087 0.999180491012037\\
0.205790645879733 0.999106674051631\\
0.208463251670379 0.999026217606577\\
0.211135857461025 0.998938524486797\\
0.21380846325167 0.998842943787801\\
0.216481069042316 0.99873876605934\\
0.219153674832962 0.998625218039486\\
0.221826280623608 0.998501456915075\\
0.224498886414254 0.998366564065907\\
0.2271714922049 0.998219538246261\\
0.229844097995546 0.998059288153128\\
0.232516703786192 0.997884624325976\\
0.235189309576837 0.997694250317954\\
0.237861915367483 0.997486753072986\\
0.240534521158129 0.997260592437312\\
0.243207126948775 0.997014089727665\\
0.245879732739421 0.996745415271206\\
0.248552338530067 0.996452574824713\\
0.251224944320713 0.996133394772274\\
0.253897550111359 0.995785505991554\\
0.256570155902004 0.995406326268934\\
0.25924276169265 0.994993041132951\\
0.261915367483296 0.994542582963812\\
0.264587973273942 0.994051608223917\\
0.267260579064588 0.993516472640346\\
0.269933184855234 0.992933204155157\\
0.27260579064588 0.992297473442681\\
0.275278396436526 0.99160456177499\\
0.277951002227172 0.990849325997007\\
0.280623608017817 0.99002616035131\\
0.283296213808463 0.989128954869249\\
0.285968819599109 0.98815105001954\\
0.288641425389755 0.987085187277718\\
0.291314031180401 0.985923455249555\\
0.293986636971047 0.984657230948523\\
0.296659242761693 0.983277115791447\\
0.299331848552338 0.981772865837282\\
0.302004454342984 0.980133315751186\\
0.30467706013363 0.97834629592953\\
0.307349665924276 0.976398542170693\\
0.310022271714922 0.974275597221178\\
0.312694877505568 0.971961703466274\\
0.315367483296214 0.969439685968732\\
0.31804008908686 0.966690824987354\\
0.320712694877506 0.963694717029218\\
0.323385300668151 0.960429123404216\\
0.326057906458797 0.956869805157787\\
0.328730512249443 0.952990343156645\\
0.331403118040089 0.948761941992069\\
0.334075723830735 0.944153216245214\\
0.336748329621381 0.939129957528027\\
0.339420935412027 0.933654880570576\\
0.342093541202673 0.927687346470166\\
0.344766146993318 0.921183061048002\\
0.347438752783964 0.914093746074486\\
0.35011135857461 0.906366780922799\\
0.352783964365256 0.897944811990896\\
0.355456570155902 0.888765326992901\\
0.358129175946548 0.878760190960012\\
0.360801781737194 0.867855140506916\\
0.36347438752784 0.85596923260984\\
0.366146993318485 0.843014243804843\\
0.368819599109131 0.828894015346834\\
0.371492204899777 0.813503739468747\\
0.374164810690423 0.796729181443144\\
0.376837416481069 0.778445831671943\\
0.379510022271715 0.758517981510653\\
0.382182628062361 0.736797715967391\\
0.384855233853007 0.713123815799961\\
0.387527839643653 0.687320560861795\\
0.390200445434298 0.659196425814544\\
0.392873051224944 0.62854265852621\\
0.39554565701559 0.595131730602965\\
0.398218262806236 0.558715648553662\\
0.400890868596882 0.54009869917009\\
0.403563474387528 0.522421439034975\\
0.406236080178174 0.505636442580443\\
0.40890868596882 0.489698678245361\\
0.411581291759465 0.474565387669585\\
0.414253897550111 0.460195970988044\\
0.416926503340757 0.446551877917009\\
0.419599109131403 0.433596504340396\\
0.422271714922049 0.421295094118814\\
0.424944320712695 0.409614645857982\\
0.427616926503341 0.398523824386535\\
0.430289532293987 0.387992876705833\\
0.432962138084632 0.377993552186447\\
0.435634743875278 0.368499026797356\\
0.438307349665924 0.359483831164737\\
0.44097995545657 0.350923782267554\\
0.443652561247216 0.342795918586862\\
0.446325167037862 0.335078438535101\\
0.448997772828508 0.327750642000389\\
0.451670378619154 0.320792874849306\\
0.4543429844098 0.314186476239528\\
0.457015590200445 0.307913728601275\\
0.459688195991091 0.30195781015375\\
0.462360801781737 0.296302749829531\\
0.465033407572383 0.290933384486404\\
0.467706013363029 0.285835318292271\\
0.470378619153675 0.280994884174687\\
0.473051224944321 0.276399107232104\\
0.475723830734967 0.272035670009242\\
0.478396436525612 0.26789287954407\\
0.481069042316258 0.263959636098656\\
0.483741648106904 0.260225403490749\\
0.48641425389755 0.256680180947322\\
0.489086859688196 0.25331447640543\\
0.491759465478842 0.250119281189723\\
0.494432071269488 0.247086045999748\\
0.497104677060134 0.244206658143728\\
0.499777282850779 0.24147341995901\\
0.502449888641425 0.238879028362632\\
0.505122494432071 0.23641655547862\\
0.507795100222717 0.234079430291653\\
0.510467706013363 0.231861421279605\\
0.513140311804009 0.229756619980294\\
0.515812917594655 0.22775942545039\\
0.518485523385301 0.225864529577058\\
0.521158129175947 0.224066903205355\\
0.523830734966592 0.22236178304683\\
0.526503340757238 0.220744659337101\\
0.529175946547884 0.219211264212409\\
0.53184855233853 0.217757560777422\\
0.534521158129176 0.216379732838629\\
0.537193763919822 0.215074175279838\\
0.539866369710468 0.213837485058324\\
0.542538975501114 0.212666452802222\\
0.545211581291759 0.211558054991809\\
0.547884187082405 0.210509446709273\\
0.550556792873051 0.209517954943634\\
0.553229398663697 0.208581072439433\\
0.555902004454343 0.207696452079875\\
0.558574610244989 0.206861901797104\\
0.561247216035635 0.20607538000439\\
0.563919821826281 0.205334991547107\\
0.566592427616926 0.204638984171524\\
0.569265033407572 0.203985745512657\\
0.571937639198218 0.20337380060471\\
0.574610244988864 0.202801809919967\\
0.57728285077951 0.202268567944482\\
0.579955456570156 0.201773002301415\\
0.582628062360802 0.201314173435566\\
0.585300668151448 0.200891274875385\\
0.587973273942094 0.200503634091712\\
0.590645879732739 0.200150713975508\\
0.593318485523385 0.199832114960109\\
0.595991091314031 0.199547577816903\\
0.598663697104677 0.199296987156935\\
0.601336302895323 0.199080375674741\\
0.604008908685969 0.198897929174693\\
0.606681514476615 0.198749992424417\\
0.609354120267261 0.19863707588426\\
0.612026726057906 0.19855986336661\\
0.614699331848552 0.19851922068376\\
0.617371937639198 0.198516205348432\\
0.620044543429844 0.198552077396489\\
0.62271714922049 0.198628311407395\\
0.625389755011136 0.198746609804027\\
0.628062360801782 0.198908917520005\\
0.630734966592428 0.199117438129432\\
0.633407572383074 0.199374651541086\\
0.636080178173719 0.199683333366371\\
0.638752783964365 0.200046576078083\\
0.641425389755011 0.200467812084798\\
0.644097995545657 0.200950838853879\\
0.646770601336303 0.201499846224185\\
0.649443207126949 0.202119446057978\\
0.652115812917595 0.202814704389734\\
0.654788418708241 0.203591176237797\\
0.657461024498886 0.204454943252703\\
0.660133630289532 0.205412654383601\\
0.662806236080178 0.206471569751065\\
0.665478841870824 0.207639607920728\\
0.66815144766147 0.208925396777108\\
0.670824053452116 0.210338328200516\\
0.673496659242762 0.21188861675147\\
0.676169265033408 0.213587362566354\\
0.678841870824053 0.215446618664211\\
0.681514476614699 0.217479462857262\\
0.684187082405345 0.219700074445755\\
0.686859688195991 0.222123815860556\\
0.689532293986637 0.224767319392902\\
0.692204899777283 0.227648579119055\\
0.694877505567929 0.230787048086269\\
0.697550111358575 0.234203740774067\\
0.70022271714922 0.237921340778798\\
0.702895322939866 0.241964313587689\\
0.705567928730512 0.24635902420793\\
0.708240534521158 0.25113385929354\\
0.710913140311804 0.256319353263766\\
0.71358574610245 0.261948317727189\\
0.716258351893096 0.26805597330996\\
0.718930957683742 0.27468008272881\\
0.721603563474387 0.281861083642157\\
0.724276169265033 0.289642219447847\\
0.726948775055679 0.298069665763742\\
0.729621380846325 0.307192649816851\\
0.732293986636971 0.317063559364595\\
0.734966592427617 0.327738037063591\\
0.737639198218263 0.339275055369421\\
0.740311804008909 0.35173696607553\\
0.742984409799554 0.365189517457432\\
0.7456570155902 0.379701830653151\\
0.748329621380846 0.395346325351686\\
0.751002227171492 0.412198583042535\\
0.753674832962138 0.430337133960086\\
0.756347438752784 0.449843151389724\\
0.75902004454343 0.470800034133179\\
0.761692650334076 0.493292854596272\\
0.764365256124722 0.517407646090064\\
0.767037861915367 0.543230498443368\\
0.769710467706013 0.570846425813959\\
0.772383073496659 0.60033796454721\\
0.775055679287305 0.631783451936392\\
0.777728285077951 0.665254928641425\\
0.780400890868597 0.700815598153638\\
0.783073496659243 0.7385167658576\\
0.785746102449889 0.778394167713743\\
0.788418708240535 0.820463584108845\\
0.79109131403118 0.864715617698482\\
0.793763919821826 0.911109494755118\\
0.796436525612472 0.959565727243938\\
0.799109131403118 1.00995744712439\\
0.801781737193764 1.00995792126356\\
0.80445434298441 1.00995840630862\\
0.807126948775056 1.00995890249262\\
0.809799554565702 1.00995941005225\\
0.812472160356348 1.00995992922776\\
0.815144766146993 1.00996046026282\\
0.817817371937639 1.00996100340452\\
0.820489977728285 1.00996155890321\\
0.823162583518931 1.00996212701238\\
0.825835189309577 1.00996270798849\\
0.828507795100223 1.00996330209085\\
0.831180400890869 1.00996390958141\\
0.833853006681514 1.0099645307245\\
0.83652561247216 1.00996516578665\\
0.839198218262806 1.00996581503627\\
0.841870824053452 1.00996647874334\\
0.844543429844098 1.00996715717909\\
0.847216035634744 1.00996785061557\\
0.84988864142539 1.00996855932526\\
0.852561247216036 1.00996928358052\\
0.855233853006681 1.00997002365318\\
0.857906458797327 1.00997077981379\\
0.860579064587973 1.0099715523311\\
0.863251670378619 1.00997234147126\\
0.865924276169265 1.00997314749705\\
0.868596881959911 1.00997397066702\\
0.871269487750557 1.00997481123447\\
0.873942093541203 1.00997566944644\\
0.876614699331849 1.00997654554254\\
0.879287305122494 1.00997743975362\\
0.88195991091314 1.00997835230042\\
0.884632516703786 1.00997928339201\\
0.887305122494432 1.00998023322407\\
0.889977728285078 1.00998120197709\\
0.892650334075724 1.00998218981434\\
0.89532293986637 1.00998319687959\\
0.897995545657016 1.00998422329478\\
0.900668151447661 1.00998526915731\\
0.903340757238307 1.00998633453716\\
0.906013363028953 1.00998741947374\\
0.908685968819599 1.00998852397243\\
0.911358574610245 1.0099896480008\\
0.914031180400891 1.00999079148453\\
0.916703786191537 1.00999195430286\\
0.919376391982183 1.00999313628376\\
0.922048997772828 1.00999433719852\\
0.924721603563474 1.009995556756\\
0.92739420935412 1.00999679459616\\
0.930066815144766 1.00999805028327\\
0.932739420935412 1.00999932329824\\
0.935412026726058 1.01000061303046\\
0.938084632516704 1.01000191876879\\
0.94075723830735 1.01000323969181\\
0.943429844097996 1.01000457485714\\
0.946102449888641 1.01000592318985\\
0.948775055679287 1.01000728346982\\
0.951447661469933 1.010008654318\\
0.954120267260579 1.0100100341813\\
0.956792873051225 1.01001142131639\\
0.959465478841871 1.01001281377182\\
0.962138084632517 1.01001420936866\\
0.964810690423162 1.01001560567939\\
0.967483296213808 1.01001700000498\\
0.970155902004454 1.01001838934974\\
0.9728285077951 1.01001977039422\\
0.975501113585746 1.01002113946545\\
0.978173719376392 1.01002249250463\\
0.980846325167038 1.01002382503204\\
0.983518930957684 1.01002513210865\\
0.98619153674833 1.01002640829451\\
0.988864142538975 1.0100276476033\\
0.991536748329621 1.01002884345292\\
0.994209354120267 1.01002998861176\\
0.996881959910913 1.01003107513997\\
0.999554565701559 1.01003209432586\\
1.0022271714922 1.01003303661632\\
1.00489977728285 1.01003389154134\\
1.0075723830735 1.01003464763164\\
1.01024498886414 1.01003529232912\\
1.01291759465479 1.01003581188927\\
1.01559020044543 1.01003619127502\\
1.01826280623608 1.01003641404106\\
1.02093541202673 1.01003646220799\\
1.02360801781737 1.01003631612528\\
1.02628062360802 1.01003595432198\\
1.02895322939866 1.01003535334433\\
1.03162583518931 1.01003448757884\\
1.03429844097996 1.01003332905963\\
1.0369710467706 1.01003184725885\\
1.03964365256125 1.01003000885824\\
1.04231625835189 1.0100277775006\\
1.04498886414254 1.01002511351903\\
1.04766146993318 1.01002197364223\\
1.05033407572383 1.01001831067355\\
1.05300668151448 1.01001407314146\\
1.05567928730512 1.0100092049191\\
1.05835189309577 1.0100036448098\\
1.06102449888641 1.00999732609595\\
1.06369710467706 1.00999017604753\\
1.06636971046771 1.00998211538711\\
1.06904231625835 1.00997305770713\\
1.071714922049 1.00996290883532\\
1.07438752783964 1.00995156614371\\
1.07706013363029 1.00993891779602\\
1.07973273942094 1.00992484192819\\
1.08240534521158 1.00990920575569\\
1.08507795100223 1.00989186460162\\
1.08775055679287 1.00987266083793\\
1.09042316258352 1.00985142273246\\
1.09309576837416 1.00982796319327\\
1.09576837416481 1.00980207840086\\
1.09844097995546 1.00977354631864\\
1.1011135857461 1.00974212507047\\
1.10378619153675 1.00970755117357\\
1.10645879732739 1.00966953761379\\
1.10913140311804 1.00962777174913\\
1.11180400890869 1.00958191302636\\
1.11447661469933 1.00953159049363\\
1.11714922048998 1.00947640009133\\
1.11982182628062 1.00941590170096\\
1.12249443207127 1.00934961593054\\
1.12516703786192 1.00927702061307\\
1.12783964365256 1.00919754699216\\
1.13051224944321 1.00911057556712\\
1.13318485523385 1.00901543156668\\
1.1358574610245 1.00891138001853\\
1.13853006681514 1.00879762037814\\
1.14120267260579 1.00867328067765\\
1.14387527839644 1.00853741115173\\
1.14654788418708 1.00838897729366\\
1.14922048997773 1.00822685229067\\
1.15189309576837 1.00804980878267\\
1.15456570155902 1.00785650988431\\
1.15723830734967 1.00764549940376\\
1.15991091314031 1.00741519118679\\
1.16258351893096 1.00716385750748\\
1.1652561247216 1.00688961642029\\
1.16792873051225 1.00659041798039\\
1.1706013363029 1.00626402923081\\
1.17327394209354 1.00590801784606\\
1.17594654788419 1.00551973431154\\
1.17861915367483 1.00509629250782\\
1.18129175946548 1.00463454855656\\
1.18396436525612 1.00413107777242\\
1.18663697104677 1.0035821495511\\
1.18930957683742 1.00298370000838\\
1.19198218262806 1.0023313021688\\
1.19465478841871 1.00162013348386\\
1.19732739420935 1.00084494044071\\
1.2 1\\
};
\addlegendentry{$M = 450$};

\end{axis}
\end{tikzpicture}%

%% file: n0_poispot.tikz
%
%
%
%
\begin{tikzpicture}

\begin{axis}[%
width=\poispotwidth,
height=0.788709677419355\poispotwidth,
scale only axis,
xmin=0,
xmax=1.2,
xlabel={$x$ in $\mu$m},
xmajorgrids,
ymin=0.1,
ymax=1.1,
ylabel={Charge density $n_0$},
ymajorgrids,
legend style={at={(0.668630952380948,0.034523809523809)},anchor=south west,draw=black,fill=white,legend cell align=left}
]
\addplot [
color=black,
solid,
line width=1.0pt
]
table[row sep=crcr]{
0 1.00000007142784\\
0.00670391061452514 0.99999954273869\\
0.0134078212290503 0.999998971290424\\
0.0201117318435754 0.999998294961169\\
0.0268156424581006 0.999997490326383\\
0.0335195530726257 0.999996532484509\\
0.0402234636871508 0.999995389639754\\
0.046927374301676 0.999994023059274\\
0.0536312849162011 0.999992386525058\\
0.0603351955307262 0.999990424934716\\
0.0670391061452514 0.999988072355921\\
0.0737430167597765 0.999985249680694\\
0.0804469273743017 0.999981861862959\\
0.0871508379888268 0.999977794646817\\
0.0938547486033519 0.999972910662315\\
0.100558659217877 0.999967044746056\\
0.107262569832402 0.999959998321745\\
0.113966480446927 0.999951532646609\\
0.120670391061453 0.999941360692661\\
0.127374301675978 0.999929137386092\\
0.134078212290503 0.999914447872674\\
0.140782122905028 0.999896793410086\\
0.147486033519553 0.999875574407543\\
0.154189944134078 0.999850070036172\\
0.160893854748603 0.999819413717118\\
0.167597765363128 0.999782563654368\\
0.174301675977654 0.999738267411053\\
0.181005586592179 0.999685019325854\\
0.187709497206704 0.999621009323289\\
0.194413407821229 0.999544061379864\\
0.201117318435754 0.999451559557669\\
0.207821229050279 0.99934035909613\\
0.214525139664804 0.999206679547301\\
0.22122905027933 0.999045976333572\\
0.227932960893855 0.998852786378852\\
0.23463687150838 0.998620542591241\\
0.241340782122905 0.998341350928521\\
0.24804469273743 0.998005722523582\\
0.254748603351955 0.997602251844965\\
0.26145251396648 0.997117230070733\\
0.268156424581006 0.996534180705918\\
0.274860335195531 0.995833301909467\\
0.281564245810056 0.994990796939612\\
0.288268156424581 0.993978070488857\\
0.294972067039106 0.992760764360409\\
0.301675977653631 0.991297600822912\\
0.308379888268156 0.989538995943323\\
0.315083798882682 0.987425398101716\\
0.321787709497207 0.984885298593874\\
0.328491620111732 0.981832851586847\\
0.335195530726257 0.978165029584808\\
0.341899441340782 0.973758227900973\\
0.348603351955307 0.968464217403171\\
0.355307262569832 0.962105329118511\\
0.362011173184358 0.954468737464075\\
0.368715083798883 0.9452996915537\\
0.375418994413408 0.934293527341202\\
0.382122905027933 0.9210862791493\\
0.388826815642458 0.905243700312977\\
0.395530726256983 0.886248503670157\\
0.402234636871508 0.872061533516122\\
0.408938547486033 0.858826358960174\\
0.415642458100559 0.846432580401112\\
0.422346368715084 0.834785240232594\\
0.429050279329609 0.823802477737162\\
0.435754189944134 0.813413560594238\\
0.442458100558659 0.803557229342263\\
0.449162011173184 0.794180302496666\\
0.455865921787709 0.785236499234933\\
0.462569832402235 0.776685444053089\\
0.46927374301676 0.76849182391551\\
0.475977653631285 0.760624673430687\\
0.48268156424581 0.753056767702027\\
0.489385474860335 0.745764105894245\\
0.49608938547486 0.738725471357777\\
0.502793296089385 0.73192205647451\\
0.509497206703911 0.725337142315955\\
0.516201117318436 0.718955824810731\\
0.522905027932961 0.712764780460114\\
0.529608938547486 0.706752065765752\\
0.536312849162011 0.700906945482009\\
0.543016759776536 0.695219745609156\\
0.549720670391061 0.689681727730666\\
0.556424581005586 0.684284981891602\\
0.563128491620112 0.679022335736357\\
0.569832402234637 0.673887278091488\\
0.576536312849162 0.668873895610717\\
0.583240223463687 0.663976821511475\\
0.589944134078212 0.659191195843655\\
0.596648044692737 0.654512637160612\\
0.603351955307263 0.64993722593195\\
0.610055865921788 0.645461500572772\\
0.616759776536313 0.641082467595951\\
0.623463687150838 0.636797628160399\\
0.630167597765363 0.632605024236511\\
0.636871508379888 0.628503308798622\\
0.643575418994413 0.624491845956499\\
0.650279329608939 0.62057084884457\\
0.656983240223464 0.616741565510169\\
0.663687150837989 0.613006526116315\\
0.670391061452514 0.609369868661324\\
0.677094972067039 0.605837765303342\\
0.683798882681564 0.602418977468016\\
0.690502793296089 0.599125575422224\\
0.697206703910615 0.595973867100238\\
0.70391061452514 0.592985591771679\\
0.710614525139665 0.590189446559086\\
0.71731843575419 0.587623027403553\\
0.724022346368715 0.585335279746974\\
0.73072625698324 0.583389565729388\\
0.737430167597765 0.581867459969189\\
0.744134078212291 0.580873377738647\\
0.750837988826816 0.580540105173628\\
0.757541899441341 0.581035220387582\\
0.764245810055866 0.582568233033622\\
0.770949720670391 0.585397972808339\\
0.777653631284916 0.589839235699316\\
0.784357541899441 0.596266807863188\\
0.791061452513967 0.605113502937954\\
0.797765363128492 0.61685640538715\\
0.804469273743017 0.619136073754502\\
0.811173184357542 0.621418355104465\\
0.817877094972067 0.623703106386039\\
0.824581005586592 0.625990183317791\\
0.831284916201117 0.62827944056163\\
0.837988826815642 0.630570731940762\\
0.844692737430167 0.632863910714579\\
0.851396648044693 0.635158829926978\\
0.858100558659218 0.637455342849422\\
0.864804469273743 0.639753303546272\\
0.871508379888268 0.642052567597907\\
0.878212290502793 0.644352993027399\\
0.884916201117318 0.646654441489707\\
0.891620111731844 0.648956779799261\\
0.898324022346369 0.651259881893507\\
0.905027932960894 0.653563631357747\\
0.911731843575419 0.655867924672199\\
0.918435754189944 0.658172675387675\\
0.925139664804469 0.66047781949443\\
0.931843575418994 0.662783322322978\\
0.93854748603352 0.665089187410447\\
0.945251396648045 0.667395467886844\\
0.95195530726257 0.669702281089605\\
0.958659217877095 0.672009827310821\\
0.96536312849162 0.674318413830967\\
0.972067039106145 0.676628485710141\\
0.97877094972067 0.678940665210756\\
0.985474860335195 0.681255802237274\\
0.992178770949721 0.683575038827796\\
0.998882681564246 0.685899891555479\\
1.00558659217877 0.688232356740937\\
1.0122905027933 0.69057504469782\\
1.01899441340782 0.692931350905886\\
1.02569832402235 0.695305674120904\\
1.03240223463687 0.697703694104959\\
1.0391061452514 0.700132725040055\\
1.04581005586592 0.702602164956778\\
1.05251396648045 0.705124066900235\\
1.05921787709497 0.707713864360795\\
1.0659217877095 0.710391292087278\\
1.07262569832402 0.713181554243613\\
1.07932960893855 0.716116805560171\\
1.08603351955307 0.719238028422801\\
1.0927374301676 0.722597410699052\\
1.09944134078212 0.726261356756814\\
1.10614525139665 0.730314299177396\\
1.11284916201117 0.734863523170393\\
1.1195530726257 0.740045272369459\\
1.12625698324022 0.74603247711008\\
1.13296089385475 0.75304453926111\\
1.13966480446927 0.761359727712773\\
1.1463687150838 0.771330894681135\\
1.15307262569832 0.783405427614966\\
1.15977653631285 0.798150622521301\\
1.16648044692737 0.816286027756567\\
1.1731843575419 0.838724800706189\\
1.17988826815642 0.866626800059533\\
1.18659217877095 0.901467089907675\\
1.19329608938547 0.945124892108145\\
1.2 1\\
};
\addlegendentry{$p = 0.3$};

\addplot [
color=black,
dashed,
line width=1.0pt
]
table[row sep=crcr]{
0 0.999999974936483\\
0.00670391061452514 0.999999416656412\\
0.0134078212290503 0.99999869184154\\
0.0201117318435754 0.999997825930972\\
0.0268156424581006 0.999996789472547\\
0.0335195530726257 0.99999553950496\\
0.0402234636871508 0.999994027712571\\
0.046927374301676 0.999992197719645\\
0.0536312849162011 0.999989981699492\\
0.0603351955307262 0.999987297279938\\
0.0670391061452514 0.99998404425857\\
0.0737430167597765 0.999980100773327\\
0.0804469273743017 0.999975318699125\\
0.0871508379888268 0.99996951806319\\
0.0938547486033519 0.99996248025262\\
0.100558659217877 0.999953939751391\\
0.107262569832402 0.999943574093579\\
0.113966480446927 0.999930991654237\\
0.120670391061453 0.999915716817854\\
0.127374301675978 0.999897171964505\\
0.134078212290503 0.999874655592677\\
0.140782122905028 0.999847315750946\\
0.147486033519553 0.99981411777286\\
0.154189944134078 0.999773805093709\\
0.160893854748603 0.999724851666266\\
0.167597765363128 0.999665404175029\\
0.174301675977654 0.999593211863178\\
0.181005586592179 0.999505541318693\\
0.187709497206704 0.999399072998608\\
0.194413407821229 0.999269775581789\\
0.201117318435754 0.999112753405519\\
0.207821229050279 0.998922061228559\\
0.214525139664804 0.998690479335922\\
0.22122905027933 0.998409240513515\\
0.227932960893855 0.998067698619919\\
0.23463687150838 0.99765292630321\\
0.241340782122905 0.997149226775263\\
0.24804469273743 0.996537541372002\\
0.254748603351955 0.995794730785846\\
0.26145251396648 0.994892703226701\\
0.268156424581006 0.993797357198455\\
0.274860335195531 0.992467299892887\\
0.281564245810056 0.990852294200234\\
0.288268156424581 0.988891377787374\\
0.294972067039106 0.986510586349818\\
0.301675977653631 0.983620199735143\\
0.308379888268156 0.980111413893718\\
0.315083798882682 0.975852323290787\\
0.321787709497207 0.970683077329782\\
0.328491620111732 0.964410050438823\\
0.335195530726257 0.956798838945383\\
0.341899441340782 0.947565869288336\\
0.348603351955307 0.936368372707473\\
0.355307262569832 0.92279245351285\\
0.362011173184358 0.90633895507728\\
0.368715083798883 0.8864068157585\\
0.375418994413408 0.862273615249515\\
0.382122905027933 0.83307305425198\\
0.388826815642458 0.797769207280189\\
0.395530726256983 0.755127569162925\\
0.402234636871508 0.731010549254726\\
0.408938547486033 0.709650647887648\\
0.415642458100559 0.690668165091748\\
0.422346368715084 0.67373903929956\\
0.429050279329609 0.658586239481852\\
0.435754189944134 0.644972578072787\\
0.442458100558659 0.632694690980144\\
0.449162011173184 0.621577980598522\\
0.455865921787709 0.611472356692157\\
0.462569832402235 0.602248640803999\\
0.46927374301676 0.593795524352534\\
0.475977653631285 0.586016990204664\\
0.48268156424581 0.578830123327141\\
0.489385474860335 0.57216324893336\\
0.49608938547486 0.565954346981542\\
0.502793296089385 0.560149700428543\\
0.509497206703911 0.554702741678292\\
0.516201117318436 0.549573067482308\\
0.522905027932961 0.544725597387383\\
0.529608938547486 0.540129854870701\\
0.536312849162011 0.535759353707656\\
0.543016759776536 0.53159107500701\\
0.549720670391061 0.527605022825139\\
0.556424581005586 0.523783848423766\\
0.563128491620112 0.520112535139753\\
0.569832402234637 0.516578137559806\\
0.576536312849162 0.513169570301026\\
0.583240223463687 0.509877443252616\\
0.589944134078212 0.506693941698524\\
0.596648044692737 0.503612751383745\\
0.603351955307263 0.500629030384696\\
0.610055865921788 0.497739431683931\\
0.616759776536313 0.494942182734357\\
0.623463687150838 0.492237231150339\\
0.630167597765363 0.489626469129024\\
0.636871508379888 0.487114053460022\\
0.643575418994413 0.484706843233007\\
0.650279329608939 0.482414983843972\\
0.656983240223464 0.480252673909465\\
0.663687150837989 0.47823916152976\\
0.670391061452514 0.476400028312514\\
0.677094972067039 0.474768833968686\\
0.683798882681564 0.47338921132195\\
0.690502793296089 0.472317521225451\\
0.697206703910615 0.471626198751995\\
0.70391061452514 0.471407944999743\\
0.710614525139665 0.471780940599963\\
0.71731843575419 0.472895273164669\\
0.724022346368715 0.474940773806939\\
0.73072625698324 0.47815643453756\\
0.737430167597765 0.482841507377058\\
0.744134078212291 0.489368232647837\\
0.750837988826816 0.498195852268633\\
0.757541899441341 0.509885044848095\\
0.764245810055866 0.525111031297076\\
0.770949720670391 0.544672117496938\\
0.777653631284916 0.569488006159205\\
0.784357541899441 0.600578251440582\\
0.791061452513967 0.639004828903713\\
0.797765363128492 0.685752466931239\\
0.804469273743017 0.687909385365078\\
0.811173184357542 0.690064380613112\\
0.817877094972067 0.692217308055687\\
0.824581005586592 0.694368023327383\\
0.831284916201117 0.696516382494132\\
0.837988826815642 0.698662242268901\\
0.844692737430167 0.700805460276981\\
0.851396648044693 0.702945895385\\
0.858100558659218 0.705083408111895\\
0.864804469273743 0.707217861145175\\
0.871508379888268 0.709349119992444\\
0.878212290502793 0.71147705380663\\
0.884916201117318 0.713601536434103\\
0.891620111731844 0.715722447748709\\
0.898324022346369 0.717839675352317\\
0.905027932960894 0.719953116744925\\
0.911731843575419 0.722062682096003\\
0.918435754189944 0.724168297785177\\
0.925139664804469 0.726269910926792\\
0.931843575418994 0.728367495151941\\
0.93854748603352 0.730461057996655\\
0.945251396648045 0.732550650340417\\
0.95195530726257 0.734636378460419\\
0.958659217877095 0.736718419420944\\
0.96536312849162 0.738797040712617\\
0.972067039106145 0.740872625304017\\
0.97877094972067 0.742945703582194\\
0.985474860335195 0.74501699405649\\
0.992178770949721 0.747087455203871\\
0.998882681564246 0.749158351471562\\
1.00558659217877 0.751231337259426\\
1.0122905027933 0.753308563724394\\
1.01899441340782 0.755392814538312\\
1.02569832402235 0.757487678359117\\
1.03240223463687 0.759597767831924\\
1.0391061452514 0.761728997533015\\
1.04581005586592 0.763888936546663\\
1.05251396648045 0.766087255499408\\
1.05921787709497 0.768336293092652\\
1.0659217877095 0.770651773754546\\
1.07262569832402 0.77305371633317\\
1.07932960893855 0.775567584226809\\
1.08603351955307 0.778225740567383\\
1.0927374301676 0.781069288769757\\
1.09944134078212 0.784150399865213\\
1.10614525139665 0.787535254748851\\
1.11284916201117 0.791307763332842\\
1.1195530726257 0.795574265619104\\
1.12625698324022 0.80046947451891\\
1.13296089385475 0.8061639903397\\
1.13966480446927 0.812873806900032\\
1.1463687150838 0.820872345556193\\
1.15307262569832 0.830505704744197\\
1.15977653631285 0.842212011135887\\
1.16648044692737 0.856546021466054\\
1.1731843575419 0.874210476462976\\
1.17988826815642 0.896096186731428\\
1.18659217877095 0.923333489547375\\
1.19329608938547 0.957358638238437\\
1.2 1\\
};
\addlegendentry{$p = 0.6$};

\addplot [
color=black,
dotted,
line width=1.0pt
]
table[row sep=crcr]{
0 1.00000003068507\\
0.00670391061452514 0.999999639968564\\
0.0134078212290503 0.99999914847784\\
0.0201117318435754 0.999998548850714\\
0.0268156424581006 0.999997816319899\\
0.0335195530726257 0.999996917188928\\
0.0402234636871508 0.999995809698491\\
0.046927374301676 0.999994443045217\\
0.0536312849162011 0.999992754940382\\
0.0603351955307262 0.999990668434989\\
0.0670391061452514 0.999988088210319\\
0.0737430167597765 0.999984896158879\\
0.0804469273743017 0.999980945972489\\
0.0871508379888268 0.999976056437871\\
0.0938547486033519 0.999970003115288\\
0.100558659217877 0.999962508018318\\
0.107262569832402 0.999953226825202\\
0.113966480446927 0.999941733037912\\
0.120670391061453 0.999927498363041\\
0.127374301675978 0.999909868413644\\
0.134078212290503 0.999888032615371\\
0.140782122905028 0.999860986933654\\
0.147486033519553 0.999827487708822\\
0.154189944134078 0.999785994477642\\
0.160893854748603 0.999734599154117\\
0.167597765363128 0.999670938316283\\
0.174301675977654 0.99959208457058\\
0.181005586592179 0.999494412005847\\
0.187709497206704 0.999373429561346\\
0.194413407821229 0.999223574663643\\
0.201117318435754 0.999037957669056\\
0.207821229050279 0.998808045399754\\
0.214525139664804 0.998523269281571\\
0.22122905027933 0.99817054015604\\
0.227932960893855 0.997733647595952\\
0.23463687150838 0.997192516316608\\
0.241340782122905 0.996522285816631\\
0.24804469273743 0.995692171426642\\
0.254748603351955 0.994664055157289\\
0.26145251396648 0.993390742718969\\
0.268156424581006 0.991813808356171\\
0.274860335195531 0.989860931136994\\
0.281564245810056 0.987442604410394\\
0.288268156424581 0.984448073549684\\
0.294972067039106 0.980740325027106\\
0.301675977653631 0.976149911460559\\
0.308379888268156 0.970467351725985\\
0.315083798882682 0.963433791891671\\
0.321787709497207 0.954729551352088\\
0.328491620111732 0.943960109630706\\
0.335195530726257 0.930639014751947\\
0.341899441340782 0.914167118002285\\
0.348603351955307 0.893807470160199\\
0.355307262569832 0.86865516455611\\
0.362011173184358 0.837601405273142\\
0.368715083798883 0.799291150525189\\
0.375418994413408 0.75207388757025\\
0.382122905027933 0.693947520294265\\
0.388826815642458 0.622496116137053\\
0.395530726256983 0.534823537717797\\
0.402234636871508 0.489756295241055\\
0.408938547486033 0.450533074968304\\
0.415642458100559 0.41636202659206\\
0.422346368715084 0.386564652852392\\
0.429050279329609 0.360558245831307\\
0.435754189944134 0.337841344759808\\
0.442458100558659 0.317981630372943\\
0.449162011173184 0.3006057980035\\
0.455865921787709 0.285391048256679\\
0.462569832402235 0.27205790819302\\
0.46927374301676 0.260364153196344\\
0.475977653631285 0.250099644292775\\
0.48268156424581 0.241081930683173\\
0.489385474860335 0.233152494926327\\
0.49608938547486 0.226173540259623\\
0.502793296089385 0.220025237243299\\
0.509497206703911 0.214603361231116\\
0.516201117318436 0.209817263843738\\
0.522905027932961 0.205588131221732\\
0.529608938547486 0.201847489807262\\
0.536312849162011 0.198535927097537\\
0.543016759776536 0.19560200050921\\
0.549720670391061 0.193001312419335\\
0.556424581005586 0.190695733795631\\
0.563128491620112 0.188652762759776\\
0.569832402234637 0.186845008087471\\
0.576536312849162 0.185249791172155\\
0.583240223463687 0.183848863495133\\
0.589944134078212 0.182628240283631\\
0.596648044692737 0.181578154936064\\
0.603351955307263 0.180693143096633\\
0.610055865921788 0.179972270129447\\
0.616759776536313 0.179419521353219\\
0.623463687150838 0.179044380947836\\
0.630167597765363 0.178862633149605\\
0.636871508379888 0.178897428444552\\
0.643575418994413 0.179180668187627\\
0.650279329608939 0.179754773649094\\
0.656983240223464 0.180674920106478\\
0.663687150837989 0.182011833359444\\
0.670391061452514 0.183855264872207\\
0.677094972067039 0.186318282270879\\
0.683798882681564 0.189542533272819\\
0.690502793296089 0.193704661635249\\
0.697206703910615 0.199024070459348\\
0.70391061452514 0.2057722363013\\
0.710614525139665 0.214283769147325\\
0.71731843575419 0.224969375924199\\
0.724022346368715 0.238330799305273\\
0.73072625698324 0.254977638773623\\
0.737430167597765 0.275645670161991\\
0.744134078212291 0.301215790210544\\
0.750837988826816 0.332731910410195\\
0.757541899441341 0.371414831421582\\
0.764245810055866 0.41866706421266\\
0.770949720670391 0.476060277540266\\
0.777653631284916 0.545291818956288\\
0.784357541899441 0.628088396950523\\
0.791061452513967 0.726021580584687\\
0.797765363128492 0.840178043232861\\
0.804469273743017 0.841627381850633\\
0.811173184357542 0.843068191337778\\
0.817877094972067 0.844500444112161\\
0.824581005586592 0.845924114260356\\
0.831284916201117 0.84733917761398\\
0.837988826815642 0.848745611847859\\
0.844692737430167 0.850143396606054\\
0.851396648044693 0.851532513663431\\
0.858100558659218 0.85291294713248\\
0.864804469273743 0.854284683727705\\
0.871508379888268 0.85564771310327\\
0.878212290502793 0.857002028283729\\
0.884916201117318 0.858347626213049\\
0.891620111731844 0.859684508453862\\
0.898324022346369 0.861012682077438\\
0.905027932960894 0.862332160795693\\
0.911731843575419 0.863642966400247\\
0.918435754189944 0.864945130590896\\
0.925139664804469 0.86623869729779\\
0.931843575418994 0.867523725629364\\
0.93854748603352 0.868800293613135\\
0.945251396648045 0.870068502940863\\
0.95195530726257 0.871328484985597\\
0.958659217877095 0.872580408428976\\
0.96536312849162 0.873824488926614\\
0.972067039106145 0.875061001352426\\
0.97877094972067 0.876290295305406\\
0.985474860335195 0.877512814742527\\
0.992178770949721 0.878729122828759\\
0.998882681564246 0.879939933382072\\
1.00558659217877 0.881146150653273\\
1.0122905027933 0.882348919636967\\
1.01899441340782 0.883549689685787\\
1.02569832402235 0.884750294925925\\
1.03240223463687 0.885953055887337\\
1.0391061452514 0.887160907915725\\
1.04581005586592 0.888377563387604\\
1.05251396648045 0.889607716582577\\
1.05921787709497 0.890857302376661\\
1.0659217877095 0.892133822831307\\
1.07262569832402 0.89344675942126\\
1.07932960893855 0.894808093268197\\
1.08603351955307 0.896232961576175\\
1.0927374301676 0.897740485815693\\
1.09944134078212 0.899354816476588\\
1.10614525139665 0.901106450916218\\
1.11284916201117 0.903033895617774\\
1.1195530726257 0.905185762875715\\
1.12625698324022 0.907623415608446\\
1.13296089385475 0.910424304041659\\
1.13966480446927 0.913686176201233\\
1.1463687150838 0.917532392849727\\
1.15307262569832 0.922118639797529\\
1.15977653631285 0.927641410561258\\
1.16648044692737 0.934348735725033\\
1.1731843575419 0.942553769746445\\
1.17988826815642 0.952652021976002\\
1.18659217877095 0.965143251320952\\
1.19329608938547 0.98065935475441\\
1.2 1\\
};
\addlegendentry{$p = 0.9$};

\addplot [
color=black,
dash pattern=on 1pt off 3pt on 6pt off 3pt,
line width=1.0pt
]
table[row sep=crcr]{
0 1.00000000000001\\
0.00670391061452514 0.999999627938142\\
0.0134078212290503 0.999999188021479\\
0.0201117318435754 0.99999866685427\\
0.0268156424581006 0.999998048397442\\
0.0335195530726257 0.999997313447043\\
0.0402234636871508 0.999996439009805\\
0.046927374301676 0.999995397555512\\
0.0536312849162011 0.999994156121891\\
0.0603351955307262 0.999992675242919\\
0.0670391061452514 0.999990907665702\\
0.0737430167597765 0.999988796814243\\
0.0804469273743017 0.999986274950142\\
0.0871508379888268 0.99998326097045\\
0.0938547486033519 0.999979657771109\\
0.100558659217877 0.999975349090262\\
0.107262569832402 0.999970195728855\\
0.113966480446927 0.999964031025678\\
0.120670391061453 0.999956655439779\\
0.127374301675978 0.999947830064205\\
0.134078212290503 0.99993726886024\\
0.140782122905028 0.999924629359825\\
0.147486033519553 0.999909501533997\\
0.154189944134078 0.999891394465696\\
0.160893854748603 0.999869720393925\\
0.167597765363128 0.999843775610942\\
0.174301675977654 0.999812717592005\\
0.181005586592179 0.9997755376149\\
0.187709497206704 0.999731027980223\\
0.194413407821229 0.999677742768276\\
0.201117318435754 0.999613950858952\\
0.207821229050279 0.999537579690376\\
0.214525139664804 0.999446147932299\\
0.22122905027933 0.999336684891697\\
0.227932960893855 0.999205634039385\\
0.23463687150838 0.999048737534001\\
0.241340782122905 0.998860898007411\\
0.24804469273743 0.998636013144119\\
0.254748603351955 0.998366777713944\\
0.26145251396648 0.998044446675077\\
0.268156424581006 0.99765855172186\\
0.274860335195531 0.997196562170757\\
0.281564245810056 0.996643479315244\\
0.288268156424581 0.995981351284346\\
0.294972067039106 0.995188692951045\\
0.301675977653631 0.994239792487181\\
0.308379888268156 0.993103882672862\\
0.315083798882682 0.991744150952772\\
0.321787709497207 0.990116557391856\\
0.328491620111732 0.988168424013291\\
0.335195530726257 0.985836752392648\\
0.341899441340782 0.983046218724558\\
0.348603351955307 0.979706786774448\\
0.355307262569832 0.975710869106445\\
0.362011173184358 0.970929955719331\\
0.368715083798883 0.965210616792266\\
0.375418994413408 0.958369772850612\\
0.382122905027933 0.950189111739277\\
0.388826815642458 0.940408518103019\\
0.395530726256983 0.928718368882534\\
0.402234636871508 0.917417862186126\\
0.408938547486033 0.906486987544501\\
0.415642458100559 0.895907088905402\\
0.422346368715084 0.885660752442182\\
0.429050279329609 0.875731705257964\\
0.435754189944134 0.866104723780019\\
0.442458100558659 0.856765550788185\\
0.449162011173184 0.847700820149904\\
0.455865921787709 0.83889798844617\\
0.462569832402235 0.830345272769459\\
0.46927374301676 0.822031594059094\\
0.475977653631285 0.813946525412962\\
0.48268156424581 0.806080244878825\\
0.489385474860335 0.798423492284792\\
0.49608938547486 0.790967529718097\\
0.502793296089385 0.783704105305045\\
0.509497206703911 0.776625419983762\\
0.516201117318436 0.769724096995862\\
0.522905027932961 0.762993153854169\\
0.529608938547486 0.756425976571583\\
0.536312849162011 0.750016295961901\\
0.543016759776536 0.743758165847242\\
0.549720670391061 0.737645943029378\\
0.556424581005586 0.731674268904382\\
0.563128491620112 0.725838052622135\\
0.569832402234637 0.720132455715334\\
0.576536312849162 0.714552878147628\\
0.583240223463687 0.709094945758514\\
0.589944134078212 0.703754499115359\\
0.596648044692737 0.698527583822154\\
0.603351955307263 0.693410442383072\\
0.610055865921788 0.688399507779939\\
0.616759776536313 0.683491399000777\\
0.623463687150838 0.678682918857692\\
0.630167597765363 0.673971054564438\\
0.636871508379888 0.669352981717623\\
0.643575418994413 0.66482607255509\\
0.650279329608939 0.660387909669361\\
0.656983240223464 0.656036306759391\\
0.663687150837989 0.651769338544151\\
0.670391061452514 0.647585382684299\\
0.677094972067039 0.643483177525748\\
0.683798882681564 0.639461900777103\\
0.690502793296089 0.635521275976977\\
0.697206703910615 0.631661715953833\\
0.70391061452514 0.627884515642796\\
0.710614525139665 0.624192110889901\\
0.71731843575419 0.620588425637901\\
0.724022346368715 0.617079337685395\\
0.73072625698324 0.613673303774807\\
0.737430167597765 0.610382199095634\\
0.744134078212291 0.607222445758258\\
0.750837988826816 0.604216531281138\\
0.757541899441341 0.601395054225314\\
0.764245810055866 0.598799483365884\\
0.770949720670391 0.596485884127892\\
0.777653631284916 0.594529958244203\\
0.784357541899441 0.59303386917925\\
0.791061452513967 0.592135500009679\\
0.797765363128492 0.592021030675717\\
0.804469273743017 0.592942053932733\\
0.811173184357542 0.59523891105407\\
0.817877094972067 0.597539970683414\\
0.824581005586592 0.59984510038941\\
0.831284916201117 0.602154166229516\\
0.837988826815642 0.604467032961274\\
0.844692737430167 0.606783564313365\\
0.851396648044693 0.609103623333978\\
0.858100558659218 0.611427072839104\\
0.864804469273743 0.613753775990117\\
0.871508379888268 0.616083597038409\\
0.878212290502793 0.618416402286047\\
0.884916201117318 0.620752061325465\\
0.891620111731844 0.623090448639501\\
0.898324022346369 0.62543144566645\\
0.905027932960894 0.627774943464824\\
0.911731843575419 0.630120846150995\\
0.918435754189944 0.632469075332109\\
0.925139664804469 0.634819575819736\\
0.931843575418994 0.637172322990256\\
0.93854748603352 0.639527332260961\\
0.945251396648045 0.641884671282154\\
0.95195530726257 0.644244475613162\\
0.958659217877095 0.646606968863628\\
0.96536312849162 0.648972488553331\\
0.972067039106145 0.651341519289727\\
0.97877094972067 0.653714735302128\\
0.985474860335195 0.656093054930235\\
0.992178770949721 0.658477710373909\\
0.998882681564246 0.660870336910988\\
1.00558659217877 0.6632730869305\\
1.0122905027933 0.665688775573768\\
1.01899441340782 0.668121066605363\\
1.02569832402235 0.670574709450625\\
1.03240223463687 0.673055841263716\\
1.0391061452514 0.675572371589618\\
1.04581005586592 0.6781344718573\\
1.05251396648045 0.680755197843309\\
1.05921787709497 0.683451280696232\\
1.0659217877095 0.686244131518081\\
1.07262569832402 0.689161116370928\\
1.07932960893855 0.692237173566247\\
1.08603351955307 0.695516864027177\\
1.0927374301676 0.699056969445018\\
1.09944134078212 0.702929783236499\\
1.10614525139665 0.707227277694992\\
1.11284916201117 0.712066379493927\\
1.1195530726257 0.717595647831546\\
1.12625698324022 0.724003728969476\\
1.13296089385475 0.7315300630444\\
1.13966480446927 0.740478451086202\\
1.1463687150838 0.751234262208364\\
1.15307262569832 0.764286287078923\\
1.15977653631285 0.780254544239479\\
1.16648044692737 0.799925750004127\\
1.1731843575419 0.82429871401102\\
1.17988826815642 0.854642686405756\\
1.18659217877095 0.892572759091059\\
1.19329608938547 0.94014796786179\\
1.2 1\\
};
\addlegendentry{$\vec m = 0$};

\end{axis}
\end{tikzpicture}%

%% file: n3_poispot.tikz
%
%
%
%
\begin{tikzpicture}

\begin{axis}[%
width=\sppoispotwidth,
height=0.788709677419355\sppoispotwidth,
scale only axis,
xmin=0,
xmax=1.2,
xlabel={$x$ in $\mu$m},
xmajorgrids,
ymin=-0.25,
ymax=0.25,
ylabel={Spin density $n_3$},
ymajorgrids,
legend style={at={(0.610416666666665,0.657539682539683)},anchor=south west,draw=black,fill=white,legend cell align=left}
]
\addplot [
color=black,
solid,
line width=1.0pt
]
table[row sep=crcr]{
0 1.9774826583251e-09\\
0.00670391061452514 1.9774826583251e-09\\
0.0134078212290503 2.07457164950453e-09\\
0.0201117318435754 2.29146665070387e-09\\
0.0268156424581006 2.66096683650467e-09\\
0.0335195530726257 3.22942791291438e-09\\
0.0402234636871508 4.06154443532821e-09\\
0.046927374301676 5.24691668332545e-09\\
0.0536312849162011 6.90905444093961e-09\\
0.0603351955307262 9.21771028334482e-09\\
0.0670391061452514 1.24057634650815e-08\\
0.0737430167597765 1.67923246829026e-08\\
0.0804469273743017 2.28143462127082e-08\\
0.0871508379888268 3.10698618644406e-08\\
0.0938547486033519 4.23771298284835e-08\\
0.100558659217877 5.78555222620676e-08\\
0.107262569832402 7.90361535440162e-08\\
0.113966480446927 1.08013176695278e-07\\
0.120670391061453 1.47650694687221e-07\\
0.127374301675978 2.01865727045263e-07\\
0.134078212290503 2.76015184935302e-07\\
0.140782122905028 3.7742508184162e-07\\
0.147486033519553 5.16114256742474e-07\\
0.154189944134078 7.05784099977498e-07\\
0.160893854748603 9.65172046267356e-07\\
0.167597765363128 1.31990252926043e-06\\
0.174301675977654 1.8050182254864e-06\\
0.181005586592179 2.46844160402158e-06\\
0.187709497206704 3.37570867485829e-06\\
0.194413407821229 4.61644246354224e-06\\
0.201117318435754 6.31320553227534e-06\\
0.207821229050279 8.63360577131391e-06\\
0.214525139664804 1.18068508749938e-05\\
0.22122905027933 1.61463860748543e-05\\
0.227932960893855 2.20808501310962e-05\\
0.23463687150838 3.01964054656636e-05\\
0.241340782122905 4.12946205310541e-05\\
0.24804469273743 5.64716165513685e-05\\
0.254748603351955 7.72262876023776e-05\\
0.26145251396648 0.000105608268789444\\
0.268156424581006 0.000144420243551305\\
0.274860335195531 0.000197494532130409\\
0.281564245810056 0.00027007121332708\\
0.288268156424581 0.000369315015885575\\
0.294972067039106 0.000505021848973377\\
0.301675977653631 0.000690584450956039\\
0.308379888268156 0.000944312029187442\\
0.315083798882682 0.00129123339806807\\
0.321787709497207 0.00176556033539633\\
0.328491620111732 0.0024140521938837\\
0.335195530726257 0.00330061035269464\\
0.341899441340782 0.00451255015020375\\
0.348603351955307 0.00616915965669585\\
0.355307262569832 0.00843337400434716\\
0.362011173184358 0.0115276910331194\\
0.368715083798883 0.015755855419593\\
0.375418994413408 0.0215323795473091\\
0.382122905027933 0.0294226965404694\\
0.388826815642458 0.0401977144706622\\
0.395530726256983 0.0549078383199852\\
0.402234636871508 0.0474263999519439\\
0.408938547486033 0.0409957151003551\\
0.415642458100559 0.0354634737361418\\
0.422346368715084 0.0307001937423838\\
0.429050279329609 0.0265956737906618\\
0.435754189944134 0.0230560198210133\\
0.442458100558659 0.0200011480926691\\
0.449162011173184 0.0173626850560847\\
0.455865921787709 0.0150821983147545\\
0.462569832402235 0.013109704356833\\
0.46927374301676 0.0114024080579181\\
0.475977653631285 0.00992363659418762\\
0.48268156424581 0.00864193668157704\\
0.489385474860335 0.00753030922827333\\
0.49608938547486 0.00656555975950304\\
0.502793296089385 0.00572774651008547\\
0.509497206703911 0.00499971101423568\\
0.516201117318436 0.00436667846104944\\
0.522905027932961 0.00381591711515177\\
0.529608938547486 0.00333644779613123\\
0.536312849162011 0.0029187958255102\\
0.543016759776536 0.00255477903347828\\
0.549720670391061 0.00223732640816966\\
0.556424581005586 0.00196032279957361\\
0.563128491620112 0.00171847578404771\\
0.569832402234637 0.00150720137472585\\
0.576536312849162 0.00132252574450002\\
0.583240223463687 0.00116100052448016\\
0.589944134078212 0.00101962956107106\\
0.596648044692737 0.000895805264515967\\
0.603351955307263 0.000787252862360109\\
0.610055865921788 0.000691980979304927\\
0.616759776536313 0.000608236990630851\\
0.623463687150838 0.000534465521538055\\
0.630167597765363 0.000469268259289054\\
0.636871508379888 0.00041136286174127\\
0.643575418994413 0.000359538112710451\\
0.650279329608939 0.000312601483096499\\
0.656983240223464 0.00026931374450076\\
0.663687150837989 0.00022830300706282\\
0.670391061452514 0.000187947154915662\\
0.677094972067039 0.000146208594205449\\
0.683798882681564 0.000100397710584801\\
0.690502793296089 4.68302639504898e-05\\
0.697206703910615 -1.96726495103266e-05\\
0.70391061452514 -0.000106518118539272\\
0.710614525139665 -0.00022442120394594\\
0.71731843575419 -0.000388998498525634\\
0.724022346368715 -0.000623124587794445\\
0.73072625698324 -0.000960422424913271\\
0.737430167597765 -0.00145044087922201\\
0.744134078212291 -0.00216634558895915\\
0.750837988826816 -0.00321635840495855\\
0.757541899441341 -0.00476079517957852\\
0.764245810055866 -0.00703747601573263\\
0.770949720670391 -0.0103996751865812\\
0.777653631284916 -0.015372881670947\\
0.784357541899441 -0.0227398250980717\\
0.791061452513967 -0.0336680519149245\\
0.797765363128492 -0.0499016826833099\\
0.804469273743017 -0.0445739089776629\\
0.811173184357542 -0.0398054587352766\\
0.817877094972067 -0.0355386728424159\\
0.824581005586592 -0.0317217198239427\\
0.831284916201117 -0.0283080193591922\\
0.837988826815642 -0.0252557215041597\\
0.844692737430167 -0.0225272363734096\\
0.851396648044693 -0.0200888095155109\\
0.858100558659218 -0.01791013865355\\
0.864804469273743 -0.0159640278610504\\
0.871508379888268 -0.0142260756067738\\
0.878212290502793 -0.0126743934324778\\
0.884916201117318 -0.0112893523285733\\
0.891620111731844 -0.0100533541463488\\
0.898324022346369 -0.00895062563437594\\
0.905027932960894 -0.00796703291305025\\
0.911731843575419 -0.00708991440694106\\
0.918435754189944 -0.0063079304415416\\
0.925139664804469 -0.00561092788078824\\
0.931843575418994 -0.00498981833587965\\
0.93854748603352 -0.00443646861586258\\
0.945251396648045 -0.00394360221743329\\
0.95195530726257 -0.00350471076659488\\
0.958659217877095 -0.00311397442927252\\
0.96536312849162 -0.00276619040268934\\
0.972067039106145 -0.00245670868513302\\
0.97877094972067 -0.00218137439950183\\
0.985474860335195 -0.00193647601644886\\
0.992178770949721 -0.00171869888671619\\
0.998882681564246 -0.00152508354998065\\
1.00558659217877 -0.00135298833978155\\
1.0122905027933 -0.0012000558513795\\
1.01899441340782 -0.00106418288217474\\
1.02569832402235 -0.000943493493022015\\
1.03240223463687 -0.000836314873815813\\
1.0391061452514 -0.000741155728450506\\
1.04581005586592 -0.000656686923031317\\
1.05251396648045 -0.000581724167350516\\
1.05921787709497 -0.000515212523468424\\
1.0659217877095 -0.000456212557069218\\
1.07262569832402 -0.000403887967428842\\
1.07932960893855 -0.00035749455069774\\
1.08603351955307 -0.000316370369177266\\
1.0927374301676 -0.000279927016852076\\
1.09944134078212 -0.000247641889255159\\
1.10614525139665 -0.000219051384600646\\
1.11284916201117 -0.000193744984113409\\
1.1195530726257 -0.000171360184109147\\
1.12625698324022 -0.000151578282715305\\
1.13296089385475 -0.0001341210631015\\
1.13966480446927 -0.000118748466867225\\
1.1463687150838 -0.000105257421763099\\
1.15307262569832 -9.3482085762995e-05\\
1.15977653631285 -8.32959070694987e-05\\
1.16648044692737 -7.46160950444357e-05\\
1.1731843575419 -6.74113769871502e-05\\
1.17988826815642 -6.17143197929084e-05\\
1.18659217877095 -5.76400834363149e-05\\
1.19329608938547 -5.54143357688697e-05\\
1.2 -5.54143357688697e-05\\
};
\addlegendentry{$p = 0.3$};

\addplot [
color=black,
dashed,
line width=1.0pt
]
table[row sep=crcr]{
0 2.79661355373819e-09\\
0.00670391061452514 2.79661355373818e-09\\
0.0134078212290503 2.93473449736679e-09\\
0.0201117318435754 3.24512908223015e-09\\
0.0268156424581006 3.77721644132181e-09\\
0.0335195530726257 4.60114377804812e-09\\
0.0402234636871508 5.81535246862093e-09\\
0.046927374301676 7.55703873723063e-09\\
0.0536312849162011 1.00165983113095e-08\\
0.0603351955307262 1.34575567944492e-08\\
0.0670391061452514 1.82440555824181e-08\\
0.0737430167597765 2.48787458697623e-08\\
0.0804469273743017 3.40550217769206e-08\\
0.0871508379888268 4.67290096552881e-08\\
0.0938547486033519 6.4218778229511e-08\\
0.100558659217877 8.83410556507762e-08\\
0.107262569832402 1.21599627206038e-07\\
0.113966480446927 1.67444944314056e-07\\
0.120670391061453 2.30631856771237e-07\\
0.127374301675978 3.17712551029813e-07\\
0.134078212290503 4.37715791720537e-07\\
0.140782122905028 6.03082874091916e-07\\
0.147486033519553 8.30957302710583e-07\\
0.154189944134078 1.14496187388418e-06\\
0.160893854748603 1.57764735466072e-06\\
0.167597765363128 2.17386655370228e-06\\
0.174301675977654 2.99542347933183e-06\\
0.181005586592179 4.12747941103682e-06\\
0.187709497206704 5.68737975662072e-06\\
0.194413407821229 7.83681637742773e-06\\
0.201117318435754 1.07985856014996e-05\\
0.207821229050279 1.48796781668914e-05\\
0.214525139664804 2.05030930872879e-05\\
0.22122905027933 2.82516707136484e-05\\
0.227932960893855 3.89284844264361e-05\\
0.23463687150838 5.36400439223645e-05\\
0.241340782122905 7.39109227770776e-05\\
0.24804469273743 0.000101841672130873\\
0.254748603351955 0.000140326355509555\\
0.26145251396648 0.000193352196866116\\
0.268156424581006 0.000266412306829464\\
0.274860335195531 0.000367074108594388\\
0.281564245810056 0.000505762115493611\\
0.288268156424581 0.000696835749201974\\
0.294972067039106 0.000960073165623501\\
0.301675977653631 0.00132271363079809\\
0.308379888268156 0.00182226803332429\\
0.315083798882682 0.00251038531566612\\
0.321787709497207 0.00345816967580406\\
0.328491620111732 0.00476348979123041\\
0.335195530726257 0.00656102115837251\\
0.341899441340782 0.0090360348629701\\
0.348603351955307 0.0124433159707903\\
0.355307262569832 0.0171330957035717\\
0.362011173184358 0.0235865574763658\\
0.368715083798883 0.0324643844114308\\
0.375418994413408 0.0446730271631574\\
0.382122905027933 0.0614549752042722\\
0.388826815642458 0.0845114193944349\\
0.395530726256983 0.116168420896105\\
0.402234636871508 0.10067503373284\\
0.408938547486033 0.0873054707255336\\
0.415642458100559 0.0757591247294816\\
0.422346368715084 0.0657796383618868\\
0.429050279329609 0.0571480581019334\\
0.435754189944134 0.0496771170833934\\
0.442458100558659 0.0432064452200058\\
0.449162011173184 0.037598544657807\\
0.455865921787709 0.0327353994375158\\
0.462569832402235 0.0285156126682119\\
0.46927374301676 0.0248519839463842\\
0.475977653631285 0.0216694553174132\\
0.48268156424581 0.0189033666138707\\
0.489385474860335 0.016497971160421\\
0.49608938547486 0.0144051711030459\\
0.502793296089385 0.0125834383833349\\
0.509497206703911 0.0109968929345838\\
0.516201117318436 0.00961451425896854\\
0.522905027932961 0.00840946633829651\\
0.529608938547486 0.00735851898084788\\
0.536312849162011 0.00644155133006592\\
0.543016759776536 0.00564112545054038\\
0.549720670391061 0.0049421197376909\\
0.556424581005586 0.00433141342991217\\
0.563128491620112 0.00379761478388273\\
0.569832402234637 0.00333082654351019\\
0.576536312849162 0.00292244322037213\\
0.583240223463687 0.00256497543064208\\
0.589944134078212 0.00225189711521552\\
0.596648044692737 0.00197751191336538\\
0.603351955307263 0.00173683526451518\\
0.610055865921788 0.00152548896601881\\
0.616759776536313 0.00133960489199746\\
0.623463687150838 0.00117573433579673\\
0.630167597765363 0.00103075890711664\\
0.636871508379888 0.000901797988189606\\
0.643575418994413 0.000786106272098178\\
0.650279329608939 0.00068095263356118\\
0.656983240223464 0.000583468168092814\\
0.663687150837989 0.000490446163039881\\
0.670391061452514 0.000398069273073356\\
0.677094972067039 0.000301528142961977\\
0.683798882681564 0.000194479500904699\\
0.690502793296089 6.82679010665514e-05\\
0.697206703910615 -8.91997720286933e-05\\
0.70391061452514 -0.000295089530082126\\
0.710614525139665 -0.000574023753576773\\
0.71731843575419 -0.000961563553264948\\
0.724022346368715 -0.00150930628000401\\
0.73072625698324 -0.00229235945973583\\
0.737430167597765 -0.00342031558428968\\
0.744134078212291 -0.00505339156902701\\
0.750837988826816 -0.00742619972791559\\
0.757541899441341 -0.0108828159561642\\
0.764245810055866 -0.0159286060779789\\
0.770949720670391 -0.0233069691486514\\
0.777653631284916 -0.0341132264867085\\
0.784357541899441 -0.0499640514486892\\
0.791061452513967 -0.0732502192751115\\
0.797765363128492 -0.107514802770992\\
0.804469273743017 -0.095739882099852\\
0.811173184357542 -0.0852371615201589\\
0.817877094972067 -0.0758711919228811\\
0.824581005586592 -0.0675207174813083\\
0.831284916201117 -0.060077213890826\\
0.837988826815642 -0.0534435743436208\\
0.844692737430167 -0.0475329286117261\\
0.851396648044693 -0.0422675820275639\\
0.858100558659218 -0.0375780624332415\\
0.864804469273743 -0.0334022643305543\\
0.871508379888268 -0.0296846805141471\\
0.878212290502793 -0.026375712420759\\
0.884916201117318 -0.0234310512871927\\
0.891620111731844 -0.0208111229870627\\
0.898324022346369 -0.0184805901191382\\
0.905027932960894 -0.0164079055552083\\
0.911731843575419 -0.0145649122291923\\
0.918435754189944 -0.0129264844674708\\
0.925139664804469 -0.0114702066283644\\
0.931843575418994 -0.0101760852410921\\
0.93854748603352 -0.00902629121573078\\
0.945251396648045 -0.00800492903958868\\
0.95195530726257 -0.00709783018555935\\
0.958659217877095 -0.00629236823767519\\
0.96536312849162 -0.00557729349115204\\
0.972067039106145 -0.00494258501137117\\
0.97877094972067 -0.00437931834088686\\
0.985474860335195 -0.00387954722785991\\
0.992178770949721 -0.00343619791527705\\
0.998882681564246 -0.00304297467971603\\
1.00558659217877 -0.00269427544287996\\
1.0122905027933 -0.00238511640012631\\
1.01899441340782 -0.00211106471909285\\
1.02569832402235 -0.00186817845949117\\
1.03240223463687 -0.0016529529533082\\
1.0391061452514 -0.00146227296404247\\
1.04581005586592 -0.0012933700151444\\
1.05251396648045 -0.00114378434240308\\
1.05921787709497 -0.00101133098345613\\
1.0659217877095 -0.000894069570702556\\
1.07262569832402 -0.000790277442477942\\
1.07932960893855 -0.00069842573225461\\
1.08603351955307 -0.0006171581377818\\
1.0927374301676 -0.000545272112553456\\
1.09944134078212 -0.000481702262080907\\
1.10614525139665 -0.000425505768799536\\
1.11284916201117 -0.000375849714211092\\
1.1195530726257 -0.000332000217968941\\
1.12625698324022 -0.000293313375057784\\
1.13296089385475 -0.000259228049591211\\
1.13966480446927 -0.000229260684929736\\
1.1463687150838 -0.000203002426007383\\
1.15307262569832 -0.000180119036971432\\
1.15977653631285 -0.000160354358562989\\
1.16648044692737 -0.000143538418540528\\
1.1731843575419 -0.000129601833695632\\
1.17988826815642 -0.000118598895480055\\
1.18659217877095 -0.000110742819707772\\
1.19329608938547 -0.000106458224591865\\
1.2 -0.000106458224591865\\
};
\addlegendentry{$p = 0.6$};

\addplot [
color=black,
dotted,
line width=1.0pt
]
table[row sep=crcr]{
0 2.08339699791424e-09\\
0.00670391061452514 2.08339699791424e-09\\
0.0134078212290503 2.18747993390096e-09\\
0.0201117318435754 2.42408133658249e-09\\
0.0268156424581006 2.83460515579041e-09\\
0.0335195530726257 3.47838733447668e-09\\
0.0402234636871508 4.43966331541657e-09\\
0.046927374301676 5.83733492776769e-09\\
0.0536312849162011 7.838649529643e-09\\
0.0603351955307262 1.06783481721115e-08\\
0.0670391061452514 1.46854603060789e-08\\
0.0737430167597765 2.03207906711375e-08\\
0.0804469273743017 2.82293580737459e-08\\
0.0871508379888268 3.93137436606337e-08\\
0.0938547486033519 5.48366808505929e-08\\
0.100558659217877 7.65645399790493e-08\\
0.107262569832402 1.06968005080931e-07\\
0.113966480446927 1.49502735531434e-07\\
0.120670391061453 2.09001889112058e-07\\
0.127374301675978 2.92225087041077e-07\\
0.134078212290503 4.08626168231465e-07\\
0.140782122905028 5.71426926972222e-07\\
0.147486033519553 7.99118776659736e-07\\
0.154189944134078 1.11756287745403e-06\\
0.160893854748603 1.56292722484485e-06\\
0.167597765363128 2.18579423376126e-06\\
0.174301675977654 3.05690525634911e-06\\
0.181005586592179 4.2751943238868e-06\\
0.187709497206704 5.9790232926685e-06\\
0.194413407821229 8.36189398216088e-06\\
0.201117318435754 1.16944210349239e-05\\
0.207821229050279 1.63550597086841e-05\\
0.214525139664804 2.28730761419412e-05\\
0.22122905027933 3.19886363165063e-05\\
0.227932960893855 4.47368311438044e-05\\
0.23463687150838 6.25651683466212e-05\\
0.241340782122905 8.74978536005654e-05\\
0.24804469273743 0.00012236548153267\\
0.254748603351955 0.000171126158428916\\
0.26145251396648 0.00023931441489368\\
0.268156424581006 0.000334668697327997\\
0.274860335195531 0.000468008365184593\\
0.281564245810056 0.000654459209184414\\
0.288268156424581 0.000915165657154932\\
0.294972067039106 0.00127968236234563\\
0.301675977653631 0.00178931373689561\\
0.308379888268156 0.00250177541976823\\
0.315083798882682 0.00349769794172515\\
0.321787709497207 0.00488969539984435\\
0.328491620111732 0.00683500174791436\\
0.335195530726257 0.00955306262027706\\
0.341899441340782 0.0133499991810808\\
0.348603351955307 0.0186525820552543\\
0.355307262569832 0.0260553322961888\\
0.362011173184358 0.036385683777487\\
0.368715083798883 0.0507938961623247\\
0.375418994413408 0.0708767132862274\\
0.382122905027933 0.0988467366721068\\
0.388826815642458 0.137763227369305\\
0.395530726256983 0.191844594690379\\
0.402234636871508 0.167320034046264\\
0.408938547486033 0.146003874519376\\
0.415642458100559 0.127461237547846\\
0.422346368715084 0.111319486321766\\
0.429050279329609 0.0972585813590337\\
0.435754189944134 0.0850030948512397\\
0.442458100558659 0.0743155625679312\\
0.449162011173184 0.0649909219112104\\
0.455865921787709 0.0568518377653335\\
0.462569832402235 0.0497447584449642\\
0.46927374301676 0.0435365754569664\\
0.475977653631285 0.038111785245413\\
0.48268156424581 0.0333700702705961\\
0.489385474860335 0.0292242319257983\\
0.49608938547486 0.0255984198492675\\
0.502793296089385 0.0224266118409266\\
0.509497206703911 0.0196513063714532\\
0.516201117318436 0.0172223959771654\\
0.522905027932961 0.0150961949753259\\
0.529608938547486 0.0132345991480851\\
0.536312849162011 0.0116043585129573\\
0.543016759776536 0.0101764471662506\\
0.549720670391061 0.00892551656464476\\
0.556424581005586 0.00782942058592429\\
0.563128491620112 0.00686880235002779\\
0.569832402234637 0.00602673413741112\\
0.576536312849162 0.00528840285101979\\
0.583240223463687 0.00464083435650554\\
0.589944134078212 0.00407265071643948\\
0.596648044692737 0.00357385480936141\\
0.603351955307263 0.00313563708032974\\
0.610055865921788 0.00275019917505066\\
0.616759776536313 0.00241058891013868\\
0.623463687150838 0.00211054034131\\
0.630167597765363 0.00184431147902985\\
0.636871508379888 0.00160651027504368\\
0.643575418994413 0.00139189658305295\\
0.650279329608939 0.00119514347668633\\
0.656983240223464 0.00101053500201219\\
0.663687150837989 0.000831568305633009\\
0.670391061452514 0.000650414895754042\\
0.677094972067039 0.000457176807292293\\
0.683798882681564 0.000238846118874035\\
0.690502793296089 -2.21630556001324e-05\\
0.697206703910615 -0.000350097617093148\\
0.70391061452514 -0.000778832761038309\\
0.710614525139665 -0.00135611062968546\\
0.71731843575419 -0.0021496081816603\\
0.724022346368715 -0.00325563975860768\\
0.73072625698324 -0.00481166011692082\\
0.737430167597765 -0.0070142596761366\\
0.744134078212291 -0.0101451146817393\\
0.750837988826816 -0.0146084901678033\\
0.757541899441341 -0.0209855738569599\\
0.764245810055866 -0.0301134206926174\\
0.770949720670391 -0.0432000358269696\\
0.777653631284916 -0.0619927780073118\\
0.784357541899441 -0.0890258541291341\\
0.791061452513967 -0.127985810431055\\
0.797765363128492 -0.184254130616132\\
0.804469273743017 -0.162988778517851\\
0.811173184357542 -0.144161494968921\\
0.817877094972067 -0.127494757554016\\
0.824581005586592 -0.112742388748646\\
0.831284916201117 -0.0996860454632495\\
0.837988826815642 -0.0881320981433138\\
0.844692737430167 -0.0779088566265752\\
0.851396648044693 -0.0688641046095002\\
0.858100558659218 -0.0608629087269159\\
0.864804469273743 -0.0537856719538072\\
0.871508379888268 -0.0475264043443248\\
0.878212290502793 -0.0419911870723929\\
0.884916201117318 -0.0370968083689781\\
0.891620111731844 -0.0327695522970729\\
0.898324022346369 -0.0289441233971884\\
0.905027932960894 -0.0255626921008676\\
0.911731843575419 -0.0225740474717474\\
0.918435754189944 -0.0199328453148033\\
0.925139664804469 -0.0175989410140526\\
0.931843575418994 -0.0155367976345649\\
0.93854748603352 -0.0137149608716907\\
0.945251396648045 -0.0121055933628367\\
0.95195530726257 -0.0106840617073269\\
0.958659217877095 -0.0094285702789414\\
0.96536312849162 -0.00831983657354313\\
0.972067039106145 -0.00734080341960563\\
0.97877094972067 -0.00647638390033308\\
0.985474860335195 -0.00571323529946082\\
0.992178770949721 -0.00503955879501857\\
0.998882681564246 -0.00444492199193912\\
1.00558659217877 -0.0039201017104071\\
1.0122905027933 -0.00345694473673142\\
1.01899441340782 -0.00304824450128348\\
1.02569832402235 -0.00268763187723492\\
1.03240223463687 -0.00236947849764766\\
1.0391061452514 -0.00208881116977848\\
1.04581005586592 -0.00184123612683843\\
1.05251396648045 -0.00162287200123186\\
1.05921787709497 -0.00143029053162598\\
1.0659217877095 -0.00126046413105006\\
1.07262569832402 -0.00111071954647027\\
1.07932960893855 -0.000978696933757753\\
1.08603351955307 -0.000862313757516619\\
1.0927374301676 -0.000759733004831702\\
1.09944134078212 -0.000669335277836598\\
1.10614525139665 -0.000589694404695596\\
1.11284916201117 -0.000519556285369801\\
1.1195530726257 -0.000457820771573622\\
1.12625698324022 -0.00040352647518669\\
1.13296089385475 -0.00035583851365145\\
1.13966480446927 -0.000314039345059214\\
1.1463687150838 -0.000277523034410691\\
1.15307262569832 -0.000245793546683862\\
1.15977653631285 -0.000218468011322025\\
1.16648044692737 -0.000195286388628031\\
1.1731843575419 -0.000176129651632092\\
1.17988826815642 -0.000161049564203848\\
1.18659217877095 -0.000150314513422757\\
1.19329608938547 -0.000144477825065407\\
1.2 -0.000144477825065407\\
};
\addlegendentry{$p = 0.9$};

\end{axis}
\end{tikzpicture}%

%% file: n0_smaller.tikz
%
%
%
%
\begin{tikzpicture}

\begin{axis}[%
width=\smalldevwidth,
height=0.788709677419355\smalldevwidth,
scale only axis,
xmin=0,
xmax=0.4,
xlabel={$x$ in $\mu$m},
xmajorgrids,
ymin=0.3,
ymax=1.1,
ylabel={Charge density $n_0$},
xticklabels={0,0,0.05,0.1,0.15,0.2,0.25,0.3,0.35,0.4},
ymajorgrids,
legend style={at={(0.525059523809522,0.641666666666667)},anchor=south west,draw=black,fill=white,legend cell align=left}
]
\addplot [
color=black,
solid,
line width=1.0pt
]
table[row sep=crcr]{
0 0.999999861932043\\
0.00223463687150838 0.999999132317737\\
0.00446927374301676 0.999998258940455\\
0.00670391061452514 0.999997197754695\\
0.00893854748603352 0.999995916479217\\
0.0111731843575419 0.999994386315671\\
0.0134078212290503 0.999992573101178\\
0.0156424581005587 0.9999904325552\\
0.017877094972067 0.9999879082611\\
0.0201117318435754 0.999984930595606\\
0.0223463687150838 0.999981415546101\\
0.0245810055865922 0.999977262961883\\
0.0268156424581006 0.999972354097179\\
0.0290502793296089 0.999966548413119\\
0.0312849162011173 0.999959679614634\\
0.0335195530726257 0.999951550867405\\
0.0357541899441341 0.999941929098077\\
0.0379888268156425 0.999930538237918\\
0.0402234636871508 0.999917051227874\\
0.0424581005586592 0.999901080559812\\
0.0446927374301676 0.999882167082457\\
0.046927374301676 0.999859766748531\\
0.0491620111731844 0.999833234919548\\
0.0513966480446927 0.999801807774338\\
0.0536312849162011 0.999764580284172\\
0.0558659217877095 0.999720480118931\\
0.0581005586592179 0.999668236731981\\
0.0603351955307263 0.999606344733078\\
0.0625698324022346 0.999533020494641\\
0.064804469273743 0.99944615074256\\
0.0670391061452514 0.999343231652723\\
0.0692737430167598 0.999221296702309\\
0.0715083798882682 0.999076831202819\\
0.0737430167597765 0.998905671060886\\
0.0759776536312849 0.998702882862489\\
0.0782122905027933 0.998462621843747\\
0.0804469273743017 0.998177963682548\\
0.0826815642458101 0.997840705302521\\
0.0849162011173184 0.997441129004537\\
0.0871508379888268 0.996967723207614\\
0.0893854748603352 0.996406851863854\\
0.0916201117318436 0.995742363179833\\
0.093854748603352 0.994955126593747\\
0.0960893854748603 0.994022484983059\\
0.0983240223463687 0.99291760676523\\
0.100558659217877 0.991608719853029\\
0.102793296089385 0.990058206279447\\
0.105027932960894 0.98822153265465\\
0.107262569832402 0.986045987395294\\
0.109497206703911 0.98346919081183\\
0.111731843575419 0.980417338594094\\
0.113966480446927 0.976803132953854\\
0.116201117318436 0.972523348643287\\
0.118435754189944 0.967455973289373\\
0.120670391061453 0.961456853049449\\
0.122905027932961 0.954355765686933\\
0.125139664804469 0.945951834122182\\
0.127374301675978 0.936008184885996\\
0.129608938547486 0.924245748563631\\
0.131843575418994 0.910336094586491\\
0.134078212290503 0.898280022472931\\
0.136312849162011 0.886645241300957\\
0.13854748603352 0.875409498130651\\
0.140782122905028 0.864552075884287\\
0.143016759776536 0.854053664205101\\
0.145251396648045 0.843896243076634\\
0.147486033519553 0.834062977775301\\
0.149720670391061 0.824538123909966\\
0.15195530726257 0.815306941458659\\
0.154189944134078 0.806355616848275\\
0.156424581005587 0.797671192241065\\
0.158659217877095 0.789241501295017\\
0.160893854748603 0.78105511075591\\
0.163128491620112 0.773101267319182\\
0.16536312849162 0.765369849271488\\
0.167597765363128 0.75785132248654\\
0.169832402234637 0.750536700409034\\
0.172067039106145 0.743417507715696\\
0.174301675977654 0.736485747394875\\
0.176536312849162 0.729733871037351\\
0.17877094972067 0.723154752182485\\
0.181005586592179 0.716741662617127\\
0.183240223463687 0.710488251581895\\
0.185474860335196 0.704388527902511\\
0.187709497206704 0.698436845135768\\
0.189944134078212 0.692627889903369\\
0.192178770949721 0.686956673686578\\
0.194413407821229 0.681418528475169\\
0.196648044692737 0.676009106811972\\
0.198882681564246 0.670724386957203\\
0.201117318435754 0.665560684124856\\
0.203351955307263 0.660514669029467\\
0.205586592178771 0.655583395341657\\
0.207821229050279 0.650764338105969\\
0.210055865921788 0.646055445750703\\
0.212290502793296 0.641455209050866\\
0.214525139664804 0.636962751334855\\
0.216759776536313 0.632577945408663\\
0.218994413407821 0.628301564178897\\
0.22122905027933 0.62413547387886\\
0.223463687150838 0.620082881256827\\
0.225698324022346 0.616148649222636\\
0.227932960893855 0.612339699460382\\
0.230167597765363 0.608665525649651\\
0.232402234636871 0.605138847515197\\
0.23463687150838 0.601776444357734\\
0.236871508379888 0.598600217540001\\
0.239106145251397 0.595638545302422\\
0.241340782122905 0.592928011156687\\
0.243575418994413 0.590515610117141\\
0.245810055865922 0.588461566695951\\
0.24804469273743 0.586842936893766\\
0.250279329608939 0.585758215970827\\
0.252513966480447 0.585333238027469\\
0.254748603351955 0.585728736932825\\
0.256983240223464 0.587150047037312\\
0.259217877094972 0.589859564650295\\
0.26145251396648 0.594192778721301\\
0.263687150837989 0.600578927062572\\
0.265921787709497 0.609567664491941\\
0.268156424581006 0.611850733666781\\
0.270391061452514 0.614136861199855\\
0.272625698324022 0.616425910856207\\
0.274860335195531 0.618717746027711\\
0.277094972067039 0.621012230143556\\
0.279329608938547 0.623309227182783\\
0.281564245810056 0.625608602315076\\
0.283798882681564 0.627910222702648\\
0.286033519553073 0.630213958504476\\
0.288268156424581 0.63251968413456\\
0.290502793296089 0.634827279838998\\
0.292737430167598 0.637136633673019\\
0.294972067039106 0.639447643979499\\
0.297206703910614 0.641760222495984\\
0.299441340782123 0.644074298248961\\
0.301675977653631 0.646389822433658\\
0.30391061452514 0.648706774526965\\
0.306145251396648 0.651025169942265\\
0.308379888268156 0.653345069611284\\
0.310614525139665 0.65566659197275\\
0.312849162011173 0.65798992796537\\
0.315083798882682 0.660315359768793\\
0.31731843575419 0.662643284217466\\
0.319553072625698 0.664974242037237\\
0.321787709497207 0.667308954333183\\
0.324022346368715 0.669648368102424\\
0.326256983240223 0.671993712973113\\
0.328491620111732 0.674346571899821\\
0.33072625698324 0.676708969199778\\
0.332960893854749 0.679083480123405\\
0.335195530726257 0.681473367152205\\
0.337430167597765 0.683882749451902\\
0.339664804469274 0.686316813433516\\
0.341899441340782 0.688782074257238\\
0.34413407821229 0.69128670043656\\
0.346368715083799 0.693840916565908\\
0.348603351955307 0.696457502730109\\
0.350837988826816 0.699152413514879\\
0.353072625698324 0.701945544917598\\
0.355307262569832 0.70486168409639\\
0.357541899441341 0.70793168509124\\
0.359776536312849 0.711193923775218\\
0.362011173184358 0.714696097812039\\
0.364245810055866 0.718497452894319\\
0.366480446927374 0.722671535755626\\
0.368715083798883 0.727309598334385\\
0.370949720670391 0.732524807232453\\
0.373184357541899 0.738457449830833\\
0.375418994413408 0.745281375163191\\
0.377653631284916 0.753211966640274\\
0.379888268156425 0.762516018636648\\
0.382122905027933 0.773523984786938\\
0.384357541899441 0.786645189470465\\
0.38659217877095 0.802386754998493\\
0.388826815642458 0.821377209060799\\
0.391061452513966 0.844396019559401\\
0.393296089385475 0.872410685440241\\
0.395530726256983 0.906623534334259\\
0.397765363128492 0.948531103103968\\
0.4 1\\
};
\addlegendentry{$p = 0.3$};

\addplot [
color=black,
dashed,
line width=1.0pt
]
table[row sep=crcr]{
0 1.0000000771752\\
0.00223463687150838 0.999999039988894\\
0.00446927374301676 0.999997870539585\\
0.00670391061452514 0.999996517840973\\
0.00893854748603352 0.999994929273069\\
0.0111731843575419 0.999993053911712\\
0.0134078212290503 0.999990836262023\\
0.0156424581005587 0.999988211393192\\
0.017877094972067 0.999985101774973\\
0.0201117318435754 0.999981414729838\\
0.0223463687150838 0.999977039765852\\
0.0245810055865922 0.99997184542094\\
0.0268156424581006 0.999965675426903\\
0.0290502793296089 0.999958344060047\\
0.0312849162011173 0.999949630548164\\
0.0335195530726257 0.99993927238176\\
0.0357541899441341 0.999926957342932\\
0.0379888268156425 0.999912314022304\\
0.0402234636871508 0.999894900544169\\
0.0424581005586592 0.999874191161981\\
0.0446927374301676 0.999849560319324\\
0.046927374301676 0.999820263693396\\
0.0491620111731844 0.999785415646204\\
0.0513966480446927 0.999743962400067\\
0.0536312849162011 0.99969465012502\\
0.0558659217877095 0.99963598697233\\
0.0581005586592179 0.99956619790579\\
0.0603351955307263 0.999483170965283\\
0.0625698324022346 0.999384393338842\\
0.064804469273743 0.999266875312349\\
0.0670391061452514 0.999127059801105\\
0.0692737430167598 0.998960714733969\\
0.0715083798882682 0.99876280504596\\
0.0737430167597765 0.998527340423904\\
0.0759776536312849 0.998247194224553\\
0.0782122905027933 0.997913888124308\\
0.0804469273743017 0.997517336040095\\
0.0826815642458101 0.997045539653287\\
0.0849162011173184 0.996484226439411\\
0.0871508379888268 0.995816419416944\\
0.0893854748603352 0.995021925833937\\
0.0916201117318436 0.994076729659979\\
0.093854748603352 0.992952269984637\\
0.0960893854748603 0.991614584175941\\
0.0983240223463687 0.990023290850625\\
0.100558659217877 0.988130383272094\\
0.102793296089385 0.9858787986383\\
0.105027932960894 0.983200722764512\\
0.107262569832402 0.980015582824632\\
0.109497206703911 0.976227673022174\\
0.111731843575419 0.97172334927958\\
0.113966480446927 0.966367719272588\\
0.116201117318436 0.960000743487249\\
0.118435754189944 0.952432651660239\\
0.120670391061453 0.943438567392054\\
0.122905027932961 0.93275222259878\\
0.125139664804469 0.920058633912351\\
0.127374301675978 0.904985606873451\\
0.129608938547486 0.887093933356309\\
0.131843575418994 0.865866156871956\\
0.134078212290503 0.850864480809591\\
0.136312849162011 0.836461490816343\\
0.13854748603352 0.82262305448816\\
0.140782122905028 0.809317547639857\\
0.143016759776536 0.796515630720207\\
0.145251396648045 0.784190048744786\\
0.147486033519553 0.772315451967704\\
0.149720670391061 0.760868234890724\\
0.15195530726257 0.749826391531991\\
0.154189944134078 0.739169385155687\\
0.156424581005587 0.728878030905546\\
0.158659217877095 0.718934389996158\\
0.160893854748603 0.709321674301477\\
0.163128491620112 0.700024160345089\\
0.16536312849162 0.691027111845631\\
0.167597765363128 0.682316710107241\\
0.169832402234637 0.673879991672793\\
0.172067039106145 0.665704792780451\\
0.174301675977654 0.657779700285408\\
0.176536312849162 0.650094008832451\\
0.17877094972067 0.642637684195387\\
0.181005586592179 0.635401332841222\\
0.183240223463687 0.628376177935755\\
0.185474860335196 0.621554042189734\\
0.187709497206704 0.614927338158743\\
0.189944134078212 0.608489066865436\\
0.192178770949721 0.602232825921667\\
0.194413407821229 0.596152828705459\\
0.196648044692737 0.590243936612233\\
0.198882681564246 0.584501706974778\\
0.201117318435754 0.578922459961219\\
0.203351955307263 0.573503368651623\\
0.205586592178771 0.568242577608068\\
0.207821229050279 0.563139356648205\\
0.210055865921788 0.558194298281918\\
0.212290502793296 0.5534095694671\\
0.214525139664804 0.548789231100448\\
0.216759776536313 0.544339642130017\\
0.218994413407821 0.540069969544079\\
0.22122905027933 0.535992830991503\\
0.223463687150838 0.53212510372072\\
0.225698324022346 0.52848894226598\\
0.227932960893855 0.525113058342399\\
0.230167597765363 0.522034330347672\\
0.232402234636871 0.51929982749089\\
0.23463687150838 0.516969355879648\\
0.236871508379888 0.515118662184642\\
0.239106145251397 0.513843466431931\\
0.241340782122905 0.513264541212787\\
0.243575418994413 0.513534112985095\\
0.245810055865922 0.514843935912516\\
0.24804469273743 0.517435484839572\\
0.250279329608939 0.521612838262849\\
0.252513966480447 0.527758983719422\\
0.254748603351955 0.536356489595112\\
0.256983240223464 0.548013766834005\\
0.259217877094972 0.56349851692678\\
0.26145251396648 0.583780465897527\\
0.263687150837989 0.610086172419025\\
0.265921787709497 0.643969653584918\\
0.268156424581006 0.646213153754249\\
0.270391061452514 0.648457423602724\\
0.272625698324022 0.65070231855947\\
0.274860335195531 0.652947694345509\\
0.277094972067039 0.655193407379344\\
0.279329608938547 0.657439315278663\\
0.281564245810056 0.659685277481939\\
0.283798882681564 0.661931156020187\\
0.286033519553073 0.664176816476801\\
0.288268156424581 0.666422129182984\\
0.290502793296089 0.66866697070811\\
0.292737430167598 0.670911225719186\\
0.294972067039106 0.67315478930206\\
0.297206703910614 0.675397569859894\\
0.299441340782123 0.677639492733111\\
0.301675977653631 0.679880504720476\\
0.30391061452514 0.682120579725209\\
0.306145251396648 0.684359725804913\\
0.308379888268156 0.68659799397223\\
0.310614525139665 0.688835489177772\\
0.312849162011173 0.691072384011776\\
0.315083798882682 0.693308935791056\\
0.31731843575419 0.695545507858986\\
0.319553072625698 0.697782596126033\\
0.321787709497207 0.700020862125421\\
0.324022346368715 0.702261174164553\\
0.326256983240223 0.704504658531068\\
0.328491620111732 0.706752763180213\\
0.33072625698324 0.709007336908291\\
0.332960893854749 0.711270727730984\\
0.335195530726257 0.71354590506714\\
0.337430167597765 0.715836611417019\\
0.339664804469274 0.718147550567077\\
0.341899441340782 0.720484621010344\\
0.34413407821229 0.722855205314848\\
0.346368715083799 0.725268528692688\\
0.348603351955307 0.727736103129318\\
0.350837988826816 0.730272277263889\\
0.353072625698324 0.732894916935593\\
0.355307262569832 0.735626247137976\\
0.357541899441341 0.738493893312979\\
0.359776536312849 0.741532168793413\\
0.362011173184358 0.744783666171767\\
0.364245810055866 0.748301223943369\\
0.366480446927374 0.752150356585252\\
0.368715083798883 0.756412257104325\\
0.370949720670391 0.761187507063603\\
0.373184357541899 0.766600661521298\\
0.375418994413408 0.772805916950523\\
0.377653631284916 0.779994121362012\\
0.379888268156425 0.788401450612576\\
0.382122905027933 0.798320157411494\\
0.384357541899441 0.810111905527642\\
0.38659217877095 0.824224339039562\\
0.388826815642458 0.84121171623648\\
0.391061452513966 0.861760675737968\\
0.393296089385475 0.886722521294846\\
0.395530726256983 0.917153844791272\\
0.397765363128492 0.95436790336363\\
0.4 1\\
};
\addlegendentry{$p = 0.6$};

\addplot [
color=black,
dotted,
line width=1.0pt
]
table[row sep=crcr]{
0 1.00000021346066\\
0.00223463687150838 0.999999159571251\\
0.00446927374301676 0.999997933458592\\
0.00670391061452514 0.999996500159039\\
0.00893854748603352 0.999994808841304\\
0.0111731843575419 0.999992800426872\\
0.0134078212290503 0.999990407416909\\
0.0156424581005587 0.99998755157377\\
0.017877094972067 0.99998414077769\\
0.0201117318435754 0.999980065387397\\
0.0223463687150838 0.999975194144827\\
0.0245810055865922 0.999969369533956\\
0.0268156424581006 0.999962402427411\\
0.0290502793296089 0.999954065806192\\
0.0312849162011173 0.999944087303214\\
0.0335195530726257 0.999932140289395\\
0.0357541899441341 0.999917833183629\\
0.0379888268156425 0.999900696619798\\
0.0402234636871508 0.999880168041533\\
0.0424581005586592 0.999855573216264\\
0.0446927374301676 0.999826104061846\\
0.046927374301676 0.999790792059179\\
0.0491620111731844 0.999748476379398\\
0.0513966480446927 0.999697765680167\\
0.0536312849162011 0.999636992317061\\
0.0558659217877095 0.999564157466395\\
0.0581005586592179 0.999476865357089\\
0.0603351955307263 0.999372244451659\\
0.0625698324022346 0.999246852988568\\
0.064804469273743 0.999096565786147\\
0.0670391061452514 0.998916438595587\\
0.0692737430167598 0.998700545557524\\
0.0715083798882682 0.998441784440106\\
0.0737430167597765 0.998131643288404\\
0.0759776536312849 0.99775992086289\\
0.0782122905027933 0.997314391749409\\
0.0804469273743017 0.996780405239087\\
0.0826815642458101 0.996140404949898\\
0.0849162011173184 0.995373353629435\\
0.0871508379888268 0.994454044567485\\
0.0893854748603352 0.99335227747278\\
0.0916201117318436 0.992031872434087\\
0.093854748603352 0.990449490582136\\
0.0960893854748603 0.988553224174063\\
0.0983240223463687 0.986280911903464\\
0.100558659217877 0.983558127157258\\
0.102793296089385 0.980295777555032\\
0.105027932960894 0.976387243288559\\
0.107262569832402 0.971704969429378\\
0.109497206703911 0.966096413450204\\
0.111731843575419 0.959379233774254\\
0.113966480446927 0.951335588457159\\
0.116201117318436 0.941705395617239\\
0.118435754189944 0.930178389867008\\
0.120670391061453 0.91638479327504\\
0.122905027932961 0.899884407707726\\
0.125139664804469 0.880153931454694\\
0.127374301675978 0.856572312342274\\
0.129608938547486 0.828403980142514\\
0.131843575418994 0.794779864530355\\
0.134078212290503 0.77445123676461\\
0.136312849162011 0.754987119328362\\
0.13854748603352 0.736336292398674\\
0.140782122905028 0.718451493304933\\
0.143016759776536 0.701289044519777\\
0.145251396648045 0.684808523256655\\
0.147486033519553 0.668972467499103\\
0.149720670391061 0.653746114039118\\
0.15195530726257 0.639097164742215\\
0.154189944134078 0.624995577804758\\
0.156424581005587 0.611413381241962\\
0.158659217877095 0.598324506256576\\
0.160893854748603 0.585704638500839\\
0.163128491620112 0.573531085568198\\
0.16536312849162 0.561782659345976\\
0.167597765363128 0.550439572134054\\
0.169832402234637 0.539483345695869\\
0.172067039106145 0.528896732664599\\
0.174301675977654 0.518663649987336\\
0.176536312849162 0.508769124362124\\
0.17877094972067 0.499199249916319\\
0.181005586592179 0.489941158700703\\
0.183240223463687 0.480983004944529\\
0.185474860335196 0.472313964446891\\
0.187709497206704 0.463924250987202\\
0.189944134078212 0.455805152243214\\
0.192178770949721 0.447949088434832\\
0.194413407821229 0.440349697797586\\
0.196648044692737 0.433001954069551\\
0.198882681564246 0.425902322497592\\
0.201117318435754 0.419048962491394\\
0.203351955307263 0.41244198704942\\
0.205586592178771 0.406083791538743\\
0.207821229050279 0.399979467441194\\
0.210055865921788 0.394137320418009\\
0.212290502793296 0.388569516663445\\
0.214525139664804 0.383292887224851\\
0.216759776536313 0.378329927023603\\
0.218994413407821 0.373710034043782\\
0.22122905027933 0.369471044970521\\
0.223463687150838 0.365661136967813\\
0.225698324022346 0.362341181929666\\
0.227932960893855 0.359587660233427\\
0.230167597765363 0.357496266808952\\
0.232402234636871 0.356186374547357\\
0.23463687150838 0.355806560440164\\
0.236871508379888 0.356541450633363\\
0.239106145251397 0.358620204810297\\
0.241340782122905 0.362327042023812\\
0.243575418994413 0.368014314783541\\
0.245810055865922 0.376118773467242\\
0.24804469273743 0.387181839614683\\
0.250279329608939 0.401874939522162\\
0.252513966480447 0.421031260635707\\
0.254748603351955 0.44568571451601\\
0.256983240223464 0.477125469059106\\
0.259217877094972 0.516954220547167\\
0.26145251396648 0.567174521708191\\
0.263687150837989 0.630294133262207\\
0.265921787709497 0.709464786586615\\
0.268156424581006 0.711556443335636\\
0.270391061452514 0.713644659903287\\
0.272625698324022 0.715729303015402\\
0.274860335195531 0.717810240994176\\
0.277094972067039 0.719887344123468\\
0.279329608938547 0.721960485094984\\
0.281564245810056 0.72402953955554\\
0.283798882681564 0.726094386780412\\
0.286033519553073 0.72815491050416\\
0.288268156424581 0.730210999947851\\
0.290502793296089 0.732262551091291\\
0.292737430167598 0.734309468250648\\
0.294972067039106 0.736351666036645\\
0.297206703910614 0.738389071786788\\
0.299441340782123 0.740421628587734\\
0.301675977653631 0.742449299032103\\
0.30391061452514 0.74447206988884\\
0.306145251396648 0.746489957909437\\
0.308379888268156 0.748503017045846\\
0.310614525139665 0.750511347422066\\
0.312849162011173 0.752515106483386\\
0.315083798882682 0.754514522848561\\
0.31731843575419 0.756509913515563\\
0.319553072625698 0.758501705226444\\
0.321787709497207 0.760490460988246\\
0.324022346368715 0.76247691298334\\
0.326256983240223 0.764462003394485\\
0.328491620111732 0.766446935030273\\
0.33072625698324 0.768433234081256\\
0.332960893854749 0.770422827885645\\
0.335195530726257 0.77241814125988\\
0.337430167597765 0.774422215783545\\
0.339664804469274 0.77643885745612\\
0.341899441340782 0.778472819410249\\
0.34413407821229 0.780530027927366\\
0.346368715083799 0.782617861925299\\
0.348603351955307 0.784745498457239\\
0.350837988826816 0.786924339681358\\
0.353072625698324 0.789168540357906\\
0.355307262569832 0.791495659364329\\
0.357541899441341 0.793927464184936\\
0.359776536312849 0.796490924073575\\
0.362011173184358 0.799219435909285\\
0.364245810055866 0.802154337045674\\
0.366480446927374 0.805346772171478\\
0.368715083798883 0.808859996954272\\
0.370949720670391 0.812772220795892\\
0.373184357541899 0.817180115364556\\
0.375418994413408 0.822203145945666\\
0.377653631284916 0.827988920712448\\
0.379888268156425 0.834719800917286\\
0.382122905027933 0.842621075616101\\
0.384357541899441 0.851971081725916\\
0.38659217877095 0.863113749248543\\
0.388826815642458 0.876474179646108\\
0.391061452513966 0.892578032808913\\
0.393296089385475 0.912075719248226\\
0.395530726256983 0.935772689816317\\
0.397765363128492 0.964667515585472\\
0.4 1\\
};
\addlegendentry{$p = 0.9$};

\addplot [
color=black,
dash pattern=on 1pt off 3pt on 6pt off 3pt,
line width=1.0pt
]
table[row sep=crcr]{
0 1.00000000000001\\
0.00223463687150838 0.999999627938142\\
0.00446927374301676 0.999999188021479\\
0.00670391061452514 0.99999866685427\\
0.00893854748603352 0.999998048397442\\
0.0111731843575419 0.999997313447043\\
0.0134078212290503 0.999996439009805\\
0.0156424581005587 0.999995397555512\\
0.017877094972067 0.999994156121891\\
0.0201117318435754 0.999992675242919\\
0.0223463687150838 0.999990907665702\\
0.0245810055865922 0.999988796814243\\
0.0268156424581006 0.999986274950142\\
0.0290502793296089 0.99998326097045\\
0.0312849162011173 0.999979657771109\\
0.0335195530726257 0.999975349090262\\
0.0357541899441341 0.999970195728855\\
0.0379888268156425 0.999964031025678\\
0.0402234636871508 0.999956655439779\\
0.0424581005586592 0.999947830064205\\
0.0446927374301676 0.99993726886024\\
0.046927374301676 0.999924629359825\\
0.0491620111731844 0.999909501533997\\
0.0513966480446927 0.999891394465696\\
0.0536312849162011 0.999869720393925\\
0.0558659217877095 0.999843775610942\\
0.0581005586592179 0.999812717592005\\
0.0603351955307263 0.9997755376149\\
0.0625698324022346 0.999731027980223\\
0.064804469273743 0.999677742768276\\
0.0670391061452514 0.999613950858952\\
0.0692737430167598 0.999537579690376\\
0.0715083798882682 0.999446147932299\\
0.0737430167597765 0.999336684891697\\
0.0759776536312849 0.999205634039385\\
0.0782122905027933 0.999048737534001\\
0.0804469273743017 0.998860898007411\\
0.0826815642458101 0.998636013144119\\
0.0849162011173184 0.998366777713944\\
0.0871508379888268 0.998044446675077\\
0.0893854748603352 0.99765855172186\\
0.0916201117318436 0.997196562170757\\
0.093854748603352 0.996643479315244\\
0.0960893854748603 0.995981351284346\\
0.0983240223463687 0.995188692951045\\
0.100558659217877 0.994239792487181\\
0.102793296089385 0.993103882672862\\
0.105027932960894 0.991744150952772\\
0.107262569832402 0.990116557391856\\
0.109497206703911 0.988168424013291\\
0.111731843575419 0.985836752392648\\
0.113966480446927 0.983046218724558\\
0.116201117318436 0.979706786774448\\
0.118435754189944 0.975710869106445\\
0.120670391061453 0.970929955719331\\
0.122905027932961 0.965210616792266\\
0.125139664804469 0.958369772850612\\
0.127374301675978 0.950189111739277\\
0.129608938547486 0.940408518103019\\
0.131843575418994 0.928718368882534\\
0.134078212290503 0.917417862186126\\
0.136312849162011 0.906486987544501\\
0.13854748603352 0.895907088905402\\
0.140782122905028 0.885660752442182\\
0.143016759776536 0.875731705257964\\
0.145251396648045 0.866104723780019\\
0.147486033519553 0.856765550788185\\
0.149720670391061 0.847700820149904\\
0.15195530726257 0.83889798844617\\
0.154189944134078 0.830345272769459\\
0.156424581005587 0.822031594059094\\
0.158659217877095 0.813946525412962\\
0.160893854748603 0.806080244878825\\
0.163128491620112 0.798423492284792\\
0.16536312849162 0.790967529718097\\
0.167597765363128 0.783704105305045\\
0.169832402234637 0.776625419983762\\
0.172067039106145 0.769724096995862\\
0.174301675977654 0.762993153854169\\
0.176536312849162 0.756425976571583\\
0.17877094972067 0.750016295961901\\
0.181005586592179 0.743758165847242\\
0.183240223463687 0.737645943029378\\
0.185474860335196 0.731674268904382\\
0.187709497206704 0.725838052622135\\
0.189944134078212 0.720132455715334\\
0.192178770949721 0.714552878147628\\
0.194413407821229 0.709094945758514\\
0.196648044692737 0.703754499115359\\
0.198882681564246 0.698527583822154\\
0.201117318435754 0.693410442383072\\
0.203351955307263 0.688399507779939\\
0.205586592178771 0.683491399000777\\
0.207821229050279 0.678682918857692\\
0.210055865921788 0.673971054564438\\
0.212290502793296 0.669352981717623\\
0.214525139664804 0.66482607255509\\
0.216759776536313 0.660387909669361\\
0.218994413407821 0.656036306759391\\
0.22122905027933 0.651769338544151\\
0.223463687150838 0.647585382684299\\
0.225698324022346 0.643483177525748\\
0.227932960893855 0.639461900777103\\
0.230167597765363 0.635521275976977\\
0.232402234636871 0.631661715953833\\
0.23463687150838 0.627884515642796\\
0.236871508379888 0.624192110889901\\
0.239106145251397 0.620588425637901\\
0.241340782122905 0.617079337685395\\
0.243575418994413 0.613673303774807\\
0.245810055865922 0.610382199095634\\
0.24804469273743 0.607222445758258\\
0.250279329608939 0.604216531281138\\
0.252513966480447 0.601395054225314\\
0.254748603351955 0.598799483365884\\
0.256983240223464 0.596485884127892\\
0.259217877094972 0.594529958244203\\
0.26145251396648 0.59303386917925\\
0.263687150837989 0.592135500009679\\
0.265921787709497 0.592021030675717\\
0.268156424581006 0.592942053932733\\
0.270391061452514 0.59523891105407\\
0.272625698324022 0.597539970683414\\
0.274860335195531 0.59984510038941\\
0.277094972067039 0.602154166229516\\
0.279329608938547 0.604467032961274\\
0.281564245810056 0.606783564313365\\
0.283798882681564 0.609103623333978\\
0.286033519553073 0.611427072839104\\
0.288268156424581 0.613753775990117\\
0.290502793296089 0.616083597038409\\
0.292737430167598 0.618416402286047\\
0.294972067039106 0.620752061325465\\
0.297206703910614 0.623090448639501\\
0.299441340782123 0.62543144566645\\
0.301675977653631 0.627774943464824\\
0.30391061452514 0.630120846150995\\
0.306145251396648 0.632469075332109\\
0.308379888268156 0.634819575819736\\
0.310614525139665 0.637172322990256\\
0.312849162011173 0.639527332260961\\
0.315083798882682 0.641884671282154\\
0.31731843575419 0.644244475613162\\
0.319553072625698 0.646606968863628\\
0.321787709497207 0.648972488553331\\
0.324022346368715 0.651341519289727\\
0.326256983240223 0.653714735302128\\
0.328491620111732 0.656093054930235\\
0.33072625698324 0.658477710373909\\
0.332960893854749 0.660870336910988\\
0.335195530726257 0.6632730869305\\
0.337430167597765 0.665688775573768\\
0.339664804469274 0.668121066605363\\
0.341899441340782 0.670574709450625\\
0.34413407821229 0.673055841263716\\
0.346368715083799 0.675572371589618\\
0.348603351955307 0.6781344718573\\
0.350837988826816 0.680755197843309\\
0.353072625698324 0.683451280696232\\
0.355307262569832 0.686244131518081\\
0.357541899441341 0.689161116370928\\
0.359776536312849 0.692237173566247\\
0.362011173184358 0.695516864027177\\
0.364245810055866 0.699056969445018\\
0.366480446927374 0.702929783236499\\
0.368715083798883 0.707227277694992\\
0.370949720670391 0.712066379493927\\
0.373184357541899 0.717595647831546\\
0.375418994413408 0.724003728969476\\
0.377653631284916 0.7315300630444\\
0.379888268156425 0.740478451086202\\
0.382122905027933 0.751234262208364\\
0.384357541899441 0.764286287078923\\
0.38659217877095 0.780254544239479\\
0.388826815642458 0.799925750004127\\
0.391061452513966 0.82429871401102\\
0.393296089385475 0.854642686405756\\
0.395530726256983 0.892572759091059\\
0.397765363128492 0.94014796786179\\
0.4 1\\
};
\addlegendentry{$\vec m = 0$};

\end{axis}
\end{tikzpicture}%